\documentclass[a4paper,10pt,english,reqno]{amsart}
\usepackage{graphicx, varioref, amscd ,color, bm, stmaryrd, mathrsfs, dsfont, esint}
\usepackage[colorlinks=false]{hyperref}

\usepackage{tikz} 
\usepackage{pgflibraryarrows} 
\usepackage{placeins} 

\newcommand{\R}{\mathbb{R}}
\newcommand{\N}{\mathbb{N}}

\newcommand{\D}{\mathbb{D}}
\newcommand{\Term}{\mathscr{T}}
\newcommand{\Zerm}{\mathscr{Z}}
\newcommand{\Borel}[1]{\mathscr{B}\left(#1\right)}
\newcommand{\Pred}{\mathscr{P}}
\newcommand{\Lebesgue}{\mathcal{L}}
\newcommand{\abs}[1]{\left|#1\right|}

\newcommand{\seq}[1]{\left\{#1\right\}}
\newcommand{\norm}[1]{\left\Vert #1\right\Vert}
\newcommand{\sgn}[1]{\mathrm{sign}\left(#1\right)}

\newcommand{\Dt}{{\Delta t}}
\newcommand{\st}{\,\mid|\,}

\newcommand{\test}{\varphi}

\newcommand{\loc}{{\mathrm{loc}}}

\newcommand{\downto}{\downarrow}
\newcommand{\Sm}{\mathcal{S}}

\newcommand{\ue}{{u^\varepsilon}}

\newcommand{\dX}{\,dX}

\newcommand{\Div}{\operatorname{div}}
\newcommand{\udt}{u_{\Dt}}
\newcommand{\vdt}{v_{\Dt}}
\newcommand{\edt}{\eta_\Dt}
\newcommand{\etat}{{\tilde{\eta}}}
\newcommand{\vdtj}{v_{\Dt_j}}
\newcommand{\udtj}{u_{\Dt_j}}

\newcommand{\Young}[1]{\mathcal{Y}\left(#1\right)}
\newcommand{\Rad}[1]{\mathcal{M}(#1)}

\newcommand{\ocal}{\mathcal{O}}
\newcommand{\fcal}{\mathscr{F}}
\newcommand{\pcal}{\mathscr{P}}

\newcommand{\E}[1]{E \left[#1\right]}
\newcommand{\Eb}[1]{E \Biggl[#1\Biggr]}
\newcommand{\kappas}{{\kappa_\sigma}}
\newcommand{\kappao}{{\kappa_0}}

\newcommand{\scl}{\mathcal{S}_{\text{CL}}}
\newcommand{\ssde}{\mathcal{S}_{\text{SDE}}}
\newcommand{\Set}[1]{\left\{#1\right\}}

\DeclareMathOperator*{\esssup}{ess\,sup}

\theoremstyle{plain} 

\theoremstyle{plain} 
\newtheorem{theorem}{Theorem}[section]
\newtheorem{lemma}{Lemma}[section]
\newtheorem{corollary}[lemma]{Corollary} 
\newtheorem{proposition}[lemma]{Proposition}
\newtheorem{limit}{Limit}
\newtheorem{estimate}{Estimate}[section]

\theoremstyle{definition}

\newtheorem{definition}{Definition}[section]
\newtheorem{remark}[lemma]{Remark}

\numberwithin{equation}{section}
\allowdisplaybreaks[1]

\title[Operator splitting for stochastic balance laws]
{Analysis of a splitting method for \\ stochastic balance laws}

\author[K. H. Karlsen]{K. H. Karlsen}
\address[Kenneth H. Karlsen]
{\newline Department of Mathematics
\newline University of Oslo
\newline P.O. Box 1053,  Blindern
\newline N--0316 Oslo, Norway} 
\email[]{kennethk@math.uio.no}

\author[E. B. Storr\o{}sten]{E. B. Storr\o{}sten}
\address[Erlend Briseid Storr\o{}sten]
{\newline Department of Mathematics
\newline University of Oslo
\newline P.O. Box 1053, Blindern
\newline N--0316 Oslo, Norway} 
\email[]{erlenbs@math.uio.no}

\thanks{We are grateful to Nils Henrik Risebro for 
valuable comments and suggestions.}

\date{\today}



\begin{document}

\begin{abstract}
We analyze a semi-discrete splitting method for 
conservation laws driven by a semilinear noise term. 
Making use of fractional $BV$ estimates, we 
show that the splitting method generates 
approximate solutions converging 
to the exact solution, as the time step $\Delta t \to 0$. 
Under the assumption of a homogenous noise function, and thus 
the availability of $BV$ estimates, we provide an $L^1$ error estimate. 
Bringing into play a generalization of 
Kru{\v{z}}kov's entropy condition, permitting 
the ``Kru{\v{z}}kov constants" to be Malliavin 
differentiable random variables, we establish an 
$L^1$ convergence rate of order $\frac13$ in $\Delta t$.
\end{abstract}

\maketitle
\allowdisplaybreaks

\tableofcontents

\section{Introdcution}
Recently there have been many works 
studying the effect of stochastic forcing on scalar conservation laws 
\cite{Bauzet:2012kx,Biswas:2014gd,Chen:2011fk,DebusscheVovelle2010,Debussche:2015aa,Debussche:2016aa,Hofmanova:2013aa,Kim2003,KarlsenStorroesten2014,FengNualart2008,Vallet:2009uq,Vallet:2000ys}, 
with emphasis on existence, uniqueness, and stability questions. 
Deterministic conservation laws exhibit shocks 
(discontinuous solutions), and a weak formulation coupled 
with an appropriate entropy condition is required to 
ensure the well-posedness \cite{Kruzkov:1970kx}. 
The question of uniqueness gets somewhat more difficult 
by adding a stochastic source term, due to the 
interaction between noise and nonlinearity. 
A pathwise theory for conservation laws with stochastic fluxes have been 
developed in \cite{Gess:2014aa,Gess:2015aa,Lions:2013aa,Lions:2014aa}.

In this paper we are interested in the convergence of approximate solutions to 
conservation laws driven by a multiplicative Wiener noise term, i.e., 
stochastic balance laws of the form
\begin{equation}\label{eq:SBL}
	du+ \Div f(u)\, dt
	=\sigma(x,u)\, d B, \qquad (t,x)\in \Pi_T,
\end{equation}
with initial data:
\begin{equation}\label{eq:SBL-data}
	u(0,x)=u_0(x),\qquad x\in \R^d.
\end{equation}
We denote by $\nabla$ and $\Div=\nabla\cdot$ the spatial 
gradient and divergence, respectively. Moreover, 
$\Pi_T = \R^d \times (0,T)$ for 
some fixed final time $T>0$, and $u(x,t)$ is the scalar 
unknown function that is sought. 
The random force in \eqref{eq:SBL} is driven by a 
Wiener process $B=B(t)=B(t,\omega)$, $\omega\in \Omega$, over 
a stochastic basis $(\Omega,\fcal,\Set{\fcal_t}_{t\ge 0},P)$, 
where $P$ is a probability measure, $\fcal$ is a $\sigma$-algebra, 
and $\Set{\fcal_t}_{t\ge 0}$ is a right-continuous 
filtration on $(\Omega,\fcal)$ such that $\fcal_0$ 
contains all the $P$--negligible subsets.

The convection flux $f:\R\to\R^d$ satisfies
\begin{equation}\label{eq:fluxLip-ass}
	\text{$f$ is (globally) Lipschitz continuous on $\R$.}
	\tag{$\mathcal{A}_{f}$}
\end{equation}
Furthermore, we will sometimes make use of the assumption 
\begin{equation}\label{eq:fluxDoubleDerBounded-ass}
	\text{$f''$ is uniformly bounded on $\R$.}
	\tag{$\mathcal{A}_{f,1}$}
\end{equation}
The noise coefficient $\sigma:\R^d\times \R\to \R$ is assumed to satisfy 
\begin{equation}\label{eq:sigma-ass}
	\norm{\sigma}_{\mathrm{Lip}} = 
	\sup_{x \in \R^d}
	\sup_{u \neq v}
	\left\{\frac{\abs{\sigma(x,u)-\sigma(x,v)}}{\abs{u-v}}\right\}
	< \infty,
	\quad
	\abs{\sigma(\cdot,0)} \in L^\infty(\R^d).
	\tag{$\mathcal{A}_{\sigma}$}
\end{equation}
These assumptions imply
\begin{equation*}
	\begin{split}
		&\abs{\sigma(x,u)-\sigma(x,v)} \le \norm{\sigma}_{\mathrm{Lip}}\abs{u-v},
		\\ & \abs{\sigma(x,u)}\le \max\left\{\norm{\sigma}_\mathrm{Lip},
		\norm{\sigma(\cdot,0)}_{L^\infty(\R^d)}\right\}\left(1+\abs{u}\right).
	\end{split}
\end{equation*}
Furthermore, we often assume the existence of constants $M_\sigma$ and $\kappas$ such that 
\begin{equation}\label{eq:sigma-ass-xdep}
 \abs{\sigma(x,u)-\sigma(y,u)}\le M_\sigma\abs{x-y}^{\kappas + 1/2}(1 + \abs{u}),
  \qquad \kappas \in (0,1/2].
  \tag{$\mathcal{A}_{\sigma,1}$}
\end{equation}

A prevailing difficulty affecting convergence/error analysis is 
related to the time discretization and the interplay 
between noise and nonlinearity. Up to now there are only a few studies 
investigating this problem. Holden and Risebro \cite{HR1997} study 
a one-dimensional equation with bounded initial data and 
a compactly supported, homogeneous noise function $\sigma=\sigma(u)$, ensuring 
$L^\infty$-bounds on the solution. An operator splitting method is used to construct 
approximate solutions, and it is shown that a subsequence of 
these approximations converges to a (possible non-unique) weak solution.
Recently this work was generalized to stochastic entropy solutions 
and extended to the multi-dimensional case 
by Bauzet \cite{Bauzet:2015aa}. Kr\"oker and Rohde \cite{Kroker:2012fk} analyze 
semi-discrete (time continuous) finite volume methods. They use the 
compensated compactness method to prove convergence to a 
stochastic entropy solution for one-dimensional
equations, with non-homogeneous noise function $\sigma=\sigma(x,u)$.
Bauzet, Charrier, and Gallo{\"u}et \cite{Bauzet:2014aa} 
analyze fully discrete finite volume methods for multi-dimensional 
equations, with homogeneous noise function $\sigma=\sigma(u)$. 
Their proof relies on weak $BV$ (energy) estimates and a uniqueness result 
for measure-valued stochastic entropy solutions.

In this paper, as in  \cite{HR1997,Bauzet:2015aa}, we will investigate 
the semi-discrete splitting method for calculating approximations to 
stochastic entropy solutions of \eqref{eq:SBL}.  Roughly speaking, this method 
is based on ``splitting off" the effect of the stochastic source $\sigma(x,u)\, d B$. 
This Godunov-type operator splitting can be used to extend sophisticated numerical methods 
for deterministic conservation laws to stochastic balance laws. 
Generally speaking, the tag ``operator splitting" refers to the well-known idea 
of constructing numerical methods for complicated partial differential equations 
by reducing them to a progression of simpler equations, each of 
which can be solved by some tailor-made numerical method. 
The operator splitting approach is described in a large number of articles and books. 
We do not survey the literature here, referring the reader instead 
to the bibliography in \cite{Holden:2010fk}. The main focus of the book \cite{Holden:2010fk} is 
on convergence results, within classes of discontinuous functions, for general 
splitting algorithms for deterministic nonlinear partial differential equations.

Compared to the earlier results of Holden-Risebro and Bauzet 
\cite{Bauzet:2015aa,HR1997}, the main contributions of the present paper are 
twofold. First, we establish convergence of the splitting approximations 
to a stochastic entropy solution in the case of non-homogeneous 
noise functions $\sigma = \sigma(x,u)$.  Whenever $\sigma$ has a 
dependency on the spatial position $x$, $BV$ estimates are no longer 
available and the approach resorted to in \cite{HR1997,Bauzet:2015aa} 
does not apply. Following an idea laid out in \cite{Chen:2011fk}, and independently in 
\cite{DebusscheVovelle2010}, we derive a fractional $BV_x$ estimate, which, via 
an interpolation argument \textit{\`{a} la} Kru{\v{z}}kov, is turned into a 
temporal equicontinuity estimate. These a priori estimates, along with Young measures and 
an earlier uniqueness result, are used to show that splitting approximations converge 
to a stochastic entropy solution. 

Let us make a few comments about the convergence proof. 
In the deterministic case, the spatial and temporal estimates would imply strong 
($L^1$) compactness of the splitting approximations.  
In the stochastic setting, we have the randomness 
variable $\omega$ for which there is no compactness; as a 
matter of fact, possible ``oscillations" in $\omega$ may prevent 
strong compactness. In the literature, the standard way of dealing 
with this issue is to look for tightness (weak compactness) 
of the probability laws of the approximations. Then an application of 
the Skorokhod representation theorem provides a 
new probability space and new random 
variables, with the same laws as the original variables, 
that do converge strongly (almost surely) in $\omega$ to some limit. 
Equipped with almost sure convergence, it is not difficult to show that 
the limit variable is a so-called martingale solution, i.e., the limit
is probabilistic weak in the sense that the stochastic basis is now 
viewed as part of the solution. One can pass (\textit{\`{a} la} Yamada \& Watanabe)
from martingale to pathwise solutions provided 
there is a strong uniqueness result. In the present paper we will 
\textit{not} follow this ``traditional" approach. 
Instead we will utilize Young measures, parametrized 
over $(t,x,\omega)$, to represent weak limits of nonlinear functions, thereby 
obtaining weak convergence of the splitting approximations towards a so-called 
Young measure-valued stochastic entropy solution. We use the spatial 
and temporal translation estimates to conclude that the limit 
is a solution in this sense. Weak convergence is then upgraded to 
strong convergence in $(t,x,\omega)$ \textit{a posteriori}, 
thanks to the fact that these measure-valued solutions are $L^1$ stable (unique). 
After the works of Tartar, DiPerna, and others, weak compactness 
arguments of this type (propagation of 
compactness) are frequently used in the nonlinear PDE 
literature, cf., e.g., \cite{Eymard:2000fr,Malek1996,Panov:1996aa,Szepessy:1989vn}, and 
recently in the context of stochastic equations \cite{Bauzet:2015aa,Bauzet:2014aa,Bauzet:2012kx,Biswas:2014gd,KarlsenStorroesten2014,Vallet:2009uq}.

Our second main contribution is an $L^1$ error estimate of 
the form $O(\Delta t^{\frac13})$, for 
homogeneous noise functions $\sigma=\sigma(u)$. 
Except for the expected convergence rate for the vanishing viscosity method 
\cite{Chen:2011fk}, this appears to be the first 
error estimate derived for approximate solutions to stochastic conservation laws. 
The rate $\frac13$ should be compared to the first 
order convergence rate available for conservation 
laws with deterministic source \cite{Langseth:1996jw}.
Our proof relies on $BV$ estimates and a generalization 
of the Kru{\v{z}}kov entropy condition, allowing 
the ``Kru{\v{z}}kov constants" to be Malliavin 
differentiable random variables, which was put 
forward in the recent work \cite{KarlsenStorroesten2014}.

The remaining part of this paper is organized as follows: Section \ref{seq:prelim} 
collects some preliminarily material along with the 
relevant notion of (stochastic entropy) solution.
The operator splitting method is defined precisely in Section \ref{seq:SD-OS}. 
A series of a priori estimates are derived in Section \ref{sec:apriori}, which 
are subsequently used in Section \ref{sec:conv} to prove 
convergence towards a stochastic entropy solution. Section \ref{sec:error} 
is devoted to the proof of the error estimate. 
Section \ref{sec:appendix} is an appendix collecting 
some definitions and useful results used elsewhere in the paper.
 
\section{Preliminaries}\label{seq:prelim}
In this article, as in \cite{KarlsenStorroesten2014}, we apply certain 
weighted $L^p$ spaces. Since we do \emph{not} assume $\sigma(x,0) \equiv 0$, 
weighted spaces on $\R^d$ provide a convenient alternative to working on 
the torus as in \cite{DebusscheVovelle2010, Debussche:2016aa}. 
The weights used herein turns out to be suitable also for the 
fractional $BV_x$ estimates, cf. Proposition~\ref{prop:fractional-BV}.

Let $\mathfrak{N}$ be the set of all nonzero $\phi \in C^1(\R^d) \cap L^1(\R^d)$ for
which there exists a constant $C$ such that $\abs{\nabla\phi} \le C
\phi$. An example is $\phi(x) = e^{-\sqrt{1 + \abs{x}^2}}$. Set 
\begin{equation*}
  C_\phi=\inf\seq{C \st \abs{\nabla\phi}\le C\phi}.
\end{equation*}
For $\phi \in \mathfrak{N}$, we use the weighted 
$L^p$-norm $\norm{\cdot}_{p,\phi}$ defined by
\begin{displaymath}
 \norm{u}_{p,\phi} := \left(\int_{\R^d} \abs{u(x)}^p \phi(x)\,dx\right)^{1/p}.
\end{displaymath}
The corresponding weighted $L^p$-space is denoted by
$L^p(\R^d,\phi)$. Similarly, we define
\begin{equation}\label{eq:DefWInfNorm}
 \hphantom{XXX} \norm{u}_{\infty,\phi^{-1}} := 
 \sup_{x \in \R^d}\Set{\frac{\abs{u(x)}}{\phi(x)}}, \quad u \in C(\R^d).
\end{equation}
Some useful results regarding functions in $\mathfrak{N}$ are 
collected in Section~\ref{sec:WeightedLp}. 

We denote by $\mathscr{E}$ the set of non-negative convex 
functions in $C^2(\R)$ such that $S'$ is bounded and $S''$ compactly supported. 
A pair of functions $(S,Q) $ is called an entropy/entropy-flux pair if 
$S:\R\to\R$ is $C^2$ and $Q=(Q_1,\ldots,Q_d):\R \mapsto\R^d $ satisfies 
$Q' = S' f'$. An entropy/entropy-flux pair $(S,Q)$ is said to 
belong to $\mathscr{E}$ if $S$ belongs to $\mathscr{E}$. 

Let $\Pred$ denote the predictable $\sigma$-algebra 
on $[0,T] \times \Omega$ with respect to $\seq{\fcal_t}$, 
see, e.g., \cite[\S~2.2]{ChungWilliams2014}. 
In general we are working with eqiuvalence classes of functions with respect 
to the measure $dt \otimes dP$. The equivalence class $u$ is 
said to be \emph{predictable} 
if it has a version $\tilde{u}$ that is $\Pred$-measurable. 
Equivalently, we could ask for any representative to be $\Pred^*$ 
measurable, where $\Pred^*$ is the completion of $\Pred$ with 
respect to $dt \otimes dP$. Note that any (jointly) measurable 
and adapted process is $\Pred^*$-measurable, 
cf., e.g., \cite[Theorem~3.7]{ChungWilliams2014}.

Next we collect some basic material related to Malliavin calculus. 
We refer to \cite{Nualart2006} for an introduction to the topic. 
The Malliavin calculus is developed with respect to the isonormal 
Gaussian process $W:L^2([0,T]) \rightarrow \mathcal{H}^1$, defined 
by $W(h) := \int_0^Th\,dB$. Here $\mathcal{H}^1$ is the 
subspace of $L^2(\Omega,\fcal,P)$ consisting of zero-mean Gaussian random variables. 
We denote by $\mathcal{S}$ the class of smooth random variables of the form
\begin{displaymath}
 V = f(W(h_1), \dots, W(h_n)),
\end{displaymath}
where $f \in C^\infty_c(\R^n)$, $h_1, \dots, h_n \in L^2([0,T])$ and $n \geq 1$. 
For such random variables, the Malliavin derivative is defined by
\begin{displaymath}
 DV = \sum_{i = 1}^n \partial_if(W(h_1), \dots, W(h_n))h_i,
\end{displaymath}
where $\partial_i$ denotes the derivative with respect to the $i$-th variable. 
The space $\mathcal{S}$ is dense in $L^2(\Omega,\fcal,P)$. Furthermore, the operator 
$D$ is closable as a map from $L^2(\Omega)$ to 
$L^2(\Omega;L^2([0,T]))$ \cite[Proposition~1.2.1]{Nualart2006}. 
The domain of $D$ in $L^2(\Omega)$ is denoted by $\D^{1,2}$. 
That is, $\D^{1,2}$ is the closure of $\mathcal{S}$ with respect to the norm
\begin{displaymath}
 \norm{V}_{\D^{1,2}} = \seq{\E{\abs{V}^2} + \E{\norm{DV}_{L^2([0,T])}^2}}^{1/2}.
\end{displaymath}
For the generalization of the above notations and results to Hilbert space-valued 
random variables, see \cite[Remark~2, p.31]{Nualart2006}.

We use the notion of stochastic entropy solution 
introduced in \cite{KarlsenStorroesten2014}, which is a refinement 
of the notion introduced by Feng and Nualart \cite{FengNualart2008}.
\begin{definition}\label{def:stoch-es}
Fix $\phi \in \mathfrak{N}$. A stochastic entropy solution $u$ of
 \eqref{eq:SBL}--\eqref{eq:SBL-data} with $u_0 \in 
 L^2(\Omega,\fcal_0,P;L^2(\R^d,\phi))$, is a stochastic process 
 $$
  u=\Set{u(t)= u(t,x)=u(t,x;\omega)}_{t\in [0,T]}
 $$ satisfying the following conditions:
 \begin{itemize}
  \item[(i)] $u$ is a predictable process in $L^2([0,T] \times \Omega;L^2(\R^d,\phi))$.
  \item[(ii)] For any random variable $V \in \D^{1,2}$ and 
  any entropy, entropy-flux pair $(S,Q) \in \mathscr{E}$,
    \begin{equation*}
     \begin{split}
    \qquad E\bigg[\iint_{\Pi_T} S(u-V)\partial_t\test & 
    + Q(u,V)\cdot \nabla \test \,dxdt + \int_{\R^d} S(u_0(x)-V)\test(0,x) \,dx\bigg] \\
    &-\E{\iint_{\Pi_T} S''(u-V)\sigma(x,u)D_tV \test \,dxdt} \\
    &+\frac{1}{2}\E{\iint_{\Pi_T} S''(u-V)\sigma(x,u)^2 \test \,dxdt} \ge 0,
     \end{split}
    \end{equation*}
 for all non-negative $\test \in C^\infty_c([0,T) \times \R^d)$.
 \end{itemize}
\end{definition} 
Here $L^2([0,T] \times \Omega;L^2(\R^d,\phi))$ denotes the Lebesgue-Bochner 
space and $D_tV$ denotes the Malliavin derivative of $V$ evaluated at time $t$. 
By \cite[Lemma~2.2]{KarlsenStorroesten2014} it suffices to consider $V \in \mathcal{S}$ in (ii).
In \cite{KarlsenStorroesten2014}, the existence and 
uniqueness of entropy solutions in the sense of 
Definition \ref{def:stoch-es} is established under 
assumptions \eqref{eq:fluxLip-ass}, \eqref{eq:sigma-ass}, 
and \eqref{eq:sigma-ass-xdep}. We also mention that whenever 
$u^0 \in L^p(\Omega;L^p(\R^d,\phi))$ with $2 \le p < \infty$,
\begin{equation*}
 \esssup_{0 \le t \le T}\seq{\E{\norm{u(t)}_{p,\phi}^p}} < \infty.
\end{equation*}

Let $\seq{J_\delta}_{\delta > 0}$ be a sequence of 
symmetric mollifiers on $\R^d$, i.e.,  
\begin{equation}\label{eq:MollifierDef}
 J_\delta(x) =\tfrac{1}{\delta^d}J\left(\tfrac{x}{\delta}\right),                                                                    
\end{equation}
where $J \geq 0$ is a smooth, symmetric function 
satisfying $\text{supp}\, (J)\subset B(0,1)$ and $\int J=1$. 
For $d = 1$, we set $J^+(x) = J(x - 1)$, so that 
$\text{supp}\,(J^+) \subset (0,2)$.

Under the additional assumption 
\eqref{eq:sigma-ass-xdep}, \cite[Proposition~5.2]{KarlsenStorroesten2014} 
asserts that the entropy solution $u$ satisfies
\begin{multline}\label{eq:EntSolSpatReg}
 \E{\iint_{\R^d \times \R^d} \abs{u(t,x+z)-u(t,x-z)}J_r(z)\phi(x)\,dx} \\
    \le e^{C_\phi\norm{f}_{\mathrm{Lip}}t}\E{\iint_{\R^d \times \R^d} \abs{u_0(x+z)-u_0(x-z)}
    J_r(z)\phi(x)\,dx} + \ocal(r^{\kappas}),
\end{multline}
where $\kappas$ is given in \eqref{eq:sigma-ass-xdep}. 
Whenever $\sigma(x,u) = \sigma(u)$, the last term 
on the right-hand side vanishes, i.e., $\ocal(\ldots) = 0$.

\section{Operator splitting}\label{seq:SD-OS}
We will now describe the basic operator splitting method 
for \eqref{eq:SBL}. Let $\scl(t)$ be the solution operator that maps an initial 
function $v_0(x)$ to the unique entropy solution 
of the deterministic conservation law
\begin{equation}\label{eq:CL}
	\partial_t v+ \Div f(v)=0, \qquad v(0,x)=v_0(x),
\end{equation}
that is, if $v(t) := \scl(t)v_0$, then $v$ is the 
unique entropy solution of \eqref{eq:CL}. 
More precisely, for each $\tau \in [0,T]$,
\begin{multline*}
 \int_{\R^d} \abs{v^0(x)-c}\test(0,x)\,dx - \int_{\R^d} \abs{v(\tau)-c}\test(\tau,x)\,dx \\
  + \int_0^\tau\int_{\R^d} \abs{v-c}\partial_t\test + \sgn{v-c}(f(v)-f(c))\cdot \nabla \test \,dxdt \ge 0,
\end{multline*}
for all $c \in \R$ and all non-negative $\test \in C^\infty_c([0,T) \times \R)$. 
Note that the integrals are well defined due to the 
global Lipschitz assumption \eqref{eq:fluxLip-ass}.
Recall that the entropy solution has a version that belongs to 
$C([0,T];L^1_{\mathrm{loc}}(\R^d))$ \cite{CancesClementGallouet2011}. 
As we frequently need to consider the evaluation $v(t)$ it is convenient 
for us to assume that $v$ has this property.
Let $u,v \in L^1(\R^d,\phi)$ where $\phi \in \mathfrak{N}$. 
Then, for any $t \in [0,T]$,
\begin{equation*}
 \norm{\scl(t)v-\scl(t)u}_{1,\phi} \le 
 e^{C_\phi\norm{f}_{\mathrm{Lip}}t}\norm{u-v}_{1,\phi}.
\end{equation*}
Suppose $u \in L^1(\Omega,\fcal_s,P;L^1(\R^d,\phi))$ for some $s \in [0,T]$. 
Let $s \leq t \leq T$. By considering the composition $\Omega \ni \omega \mapsto u(\omega) 
\mapsto \scl(t-s)u(\omega)$, it follows that $\scl(t-s)u$ is $\fcal_s$-measurable as an 
element in $L^1(\R^d,\phi)$, cf.~\cite[\S~3.3]{MishraSchwab2012}.

Similarly, for $s \leq t \leq T$, we let $\ssde(t,s)$ denote the 
two-paramater semigroup defined by $\ssde(t,s)w^s = w(t)$, where $w$ is the strong solution of
\begin{equation*}
 w(t,x) = w^s(x) + \int_s^t \sigma(x,w(r,x)) \,dB(r).
\end{equation*}
Suppose $w^s,v^s \in L^1(\Omega,\fcal_s,P;L^1(\R^d,\phi))$. Then
\begin{equation}\label{eq:SSDE1Contraction}
 \E{\norm{\ssde(t,s)w^s-\ssde(t,s)v^s}_{1,\phi}} = \E{\norm{w^s-v^s}_{1,\phi}}.
\end{equation}
To see this, let $S_\delta \rightarrow \abs{\cdot}$ as 
$\delta \downto 0$ and consider the quantity $S_\delta(w(t,x)-v(t,x))$. 
Next, apply It\^o's formula, multiply by $\phi$ and let $\delta \downarrow 0$. 
Due to \eqref{eq:SSDE1Contraction},
\begin{equation*}
 \ssde(\cdot,s): L^1(\Omega,\fcal_s,P;L^1(\R^d,\phi)) \rightarrow 
 L^1([s,T] \times \Omega,\Pred_{[s,T]},dt \otimes \,dP;L^1(\R^d,\phi)),
\end{equation*}
where $\Pred_{[s,T]}$ denotes the predictable 
$\sigma$-algebra relative to $\seq{\fcal_t}_{s \le t \le T}$ on $[s,T] \times \Omega$.

Fix $N \in \N$, specify $\Dt = T/N$, and set $t_n = n\Dt$. The operator splitting 
for \eqref{eq:SBL}, with initial condition $u^0 = u^0(x;\omega)$, is the 
sequence $\seq{u^n = u^n(x;\omega)}_{n = 0}^N$ defined recursively by
\begin{equation}\label{eq:SD-split}
 u^{n+1}(x;\omega) = \left[\ssde(t_{n+1},t_n;\omega) \circ \scl(\Dt)\right]u^n(x;\omega),
\end{equation} 
for $n = 0,1, \dots ,N-1$. A graphical representation of 
the splitting is given in Figure~\ref{fig:udtvdtRepr}.

\begin{figure}[h]
\begin{tikzpicture}

\draw[thick, ->] (0,0) -- (4,0) node[right]{$\ssde$};
\draw[thick, ->] (0,0) -- (0,4) node[above]{$\scl$};

\draw (1,-0.1) -- (1,0.1) node[below=8pt]{$t_{n-1}$};
\draw (2,-0.1) -- (2,0.1) node[below=8pt]{$t_n$};
\draw (3,-0.1) -- (3,0.1) node[below=8pt]{$t_{n+1}$};

\draw (-0.1,1) -- (0.1,1) node[left=8pt]{$t_{n-1}$};
\draw (-0.1,2) -- (0.1,2) node[left=8pt]{$t_n$};
\draw (-0.1,3) -- (0.1,3) node[left=8pt]{$t_{n+1}$};

\filldraw (1,1) circle (2pt) node[right]{$u^{n-1}$};
\filldraw (2,2) circle (2pt) node[right]{$u^{n}$};
\filldraw (3,3) circle (2pt) node[right]{$u^{n+1}$};

\draw (1,1) -- (1,2) -- (2,2) -- (2,3) -- (3,3);
\draw[thick, loosely dotted] (3.3,3.3) -- (3.7,3.7);
\draw[thick, loosely dotted] (0.3,0.3) -- (0.7,0.7);

\draw [-o] (2,2) -- (2,2.66) node[left=2pt]{$\vdt$};
\draw [-o] (2,3) -- (2.66,3) node[above=2pt]{$\udt$};

\end{tikzpicture}
\caption{ A graphical representation of $\seq{u^n}, \udt, \vdt$.}
\label{fig:udtvdtRepr}
\end{figure}

To investigate the convergence of the semi-discrete splitting algorithm 
\eqref{eq:SD-split},  we need to work with functions that are not only defined 
for each $t_n = n\Dt$, but in the entire interval $[0,T]$. 
To this end, we introduce two different ``time-interpolants" 
$\udt(t)=\udt(t,x;\omega)$ and $\vdt(t)=\vdt(t,x;\omega)$, defined 
for $n=0,\ldots,N-1$ by
\begin{equation}\label{eq:udt}
	\udt(t)= \ssde\left(t,t_n\right) \circ \scl(\Dt)u^n, \qquad t\in (t_n,t_{n+1}],	 
\end{equation}
and
\begin{equation}\label{eq:vdt}
	\vdt(t) = \scl(t-t_n)u^n,  \qquad  t \in [t_n,t_{n+1}),
\end{equation}
respectively, cf.~Figure~\ref{fig:udtvdtRepr}. As $\udt$ is 
discontinuous at $t_n$ we introduce the right limit 
$\udt((t_n)+) = \scl(\Dt)u^n$. Similarly, let 
$\vdt((t_{n+1})-) = \scl(\Dt)u^n$.

\section{A priori estimates}\label{sec:apriori}
To establish the convergence of $\Set{\udt}_{\Dt>0}, \Set{\vdt}_{\Dt > 0}$ 
we will need a series of a priori estimates. 
These are also crucial when deriving the error estimate. 
The following result explains the introduction of the weight functions $\mathfrak{N}$.
\begin{proposition}[Local $L^p$ estimates]\label{prop:lpLoc-est}
  Suppose $\eqref{eq:fluxLip-ass}$ and $\eqref{eq:sigma-ass}$ are 
  satisfied, $2 \nobreak \le  \nobreak p < \infty$ and $M \ge \norm{f}_{\mathrm{Lip}}$. 
  Let $\seq{u^n}$ be the splitting solutions defined by $\eqref{eq:SD-split}$, 
  with initial condition $u^0 \in L^p(\Omega,\fcal_0,P;L^p_\mathrm{loc}(\R^d))$. 
  For $t\in (0,T)$ and $R > 0$, set $\Gamma(t) = \max\{0,R-Mt\}$. 
  Suppose $\phi \in C^1(\R)$ is non-negative and 
  satisfies $\abs{\nabla \phi} \le C_\phi\phi$. Then there 
  exist constants $C_1$ and $C_2$ depending only on $p,\sigma,f,C_\phi$ such that
  \begin{multline}\label{eq:lpLoc-est}
    \E{\int_{B(0,\Gamma(t_n))} \abs{u^n(x)}^p\phi(x)\,dx} \le e^{C_1t_n}
    \E{\int_{B(0,R)} \abs{u^0(x)}^p\phi(x)\,dx}\\
    + C_2t_ne^{C_1t_n}\int_{B(0,R)} \phi(x)\,dx.
  \end{multline}
  If $\sigma(x,0) = 0$, then $C_2 = 0$. Here, $B(0,R)$ denotes the 
  open ball with radius $R$ centered at $0$.
\end{proposition}
\begin{remark}
 Suppose $\phi \in \mathfrak{N}$ and $u^0 \in L^p(\Omega;L^p(\R^d,\phi))$. Then $\phi \in L^1(\R^d)$ and the right hand side of \eqref{eq:lpLoc-est} is bounded independently of $R > 0$. It follows that $u^n \in L^p(\Omega;L^p(\R^d,\phi))$.
\end{remark}
\begin{proof}
\textit{1.~Deterministic step.} We want to prove the following:
 With $1 \le p < \infty$, let $v^0 \in L^p_{\mathrm{loc}}(\R^d)$ 
 and $v(t) = \scl(t)v^0$. Then, for any $0 < \tau \leq T$,
 \begin{equation}\label{eq:LpLocEstDetStep}
  \int_{B(0,\Gamma(\tau))} \abs{v(\tau,x)}^p\phi(x) \,dx 
  \le e^{\norm{f}_\mathrm{Lip}C_\phi t}\int_{B(0,R)} \abs{v^0(x)}^p\phi(x) \,dx.
 \end{equation}
 We might as well assume $\Gamma(\tau) > 0$. 
 As $v$ is an entropy solution of \eqref{eq:CL}, 
 \begin{equation}\label{eq:entIneqv}
  \iint_{\Pi_T}\int_{\R^d} S(v(t,x))\partial_t\test 
+ Q(v(t,x)) \cdot \nabla_x \test \,dxdt
  + \int_\R S(v^0(x))\test(0,x)\,dx \ge 0,
 \end{equation}
 for all nonnegative $\test \in C^\infty_c([0,T) \times \R^d)$, for any 
 convex $S\in C^2$ with $S'$ bounded and $Q'=S'f'$.
 Let $0 < \delta < \mathrm{min}\seq{\Gamma(\tau),\frac{1}{2}\tau}$. 
 Take $\test(t,x) = \psi_\delta(t)H_\delta(\Gamma(t),\abs{x})\phi(x)$, where
 \begin{equation*}
  \psi_\delta(t) = 1- \int_0^t J_\delta^+(\tau-\zeta)\,d\zeta \quad \mbox{ and } 
  \quad H_\delta(L,r) = \int_{-\delta}^{L} J_\delta(\zeta-r)\,d\zeta.
 \end{equation*}
 Under the assumption $\phi \in C^\infty(\R^d)$ it follows that $\test$ is a 
 non-negative function in $C^\infty_c([0,T) \times \R^d)$. However, by 
 approximation, it suffices with $\phi \in C^1(\R^d)$ for inequality \eqref{eq:entIneqv} to hold true. 
 Recall that $\frac{d}{dt} \Gamma(t) = -M$ for all $0 \leq t \leq \tau$ and observe that 
 \begin{align*}
  \partial_t \test(t,x) &= -J_\delta^+(\tau-t)H_\delta(\Gamma(t),\abs{x})\phi(x) 
  -M\psi_\delta(t)J_\delta(\Gamma(t)-\abs{x})\phi(x), \\
  \nabla \test(t,x) &= -\psi_\delta(t)J_\delta(\Gamma(t)-\abs{x})\frac{x}{\abs{x}}\phi(x) 
  + \psi_\delta(t)H_\delta(\Gamma(t),\abs{x})\nabla\phi(x). 
 \end{align*}
 Hence, 
 \begin{equation}\label{eq:LplocCLIneq}
  \begin{split}
  \int_\R S(v^0(x))&H_\delta(R,\abs{x})\phi(x)\,dx \ge
    \iint_{\Pi_T} S(v(t,x))J_\delta^+(\tau-t)
    H_\delta(\Gamma(t),\abs{x})\phi(x)\,dxdt \\
   &+\underbrace{\iint_{\Pi_T} \left(Q(v(t,x))\cdot\frac{x}{\abs{x}} 
   + MS(v(t,x))\right) \psi_\delta(t)J_\delta(\Gamma(t)-\abs{x})\phi(x)\,dxdt}_{\Term^1} \\
   &-\underbrace{\iint_{\Pi_T}Q(v(t,x))\psi_\delta(t)H_\delta(\Gamma(t),\abs{x})
   \cdot \nabla \phi(x)\,dxdt}_{\Term^2}.
  \end{split}
 \end{equation}
 Suppose $S'(0) = S(0) = 0$. Then
 \begin{equation*}
  \abs{Q(v)} = \abs{\int_0^v S'(z)f'(z)\,dz} 
\le \norm{f}_{\mathrm{Lip}}S(v).
 \end{equation*}
 It follows as $M \ge \norm{f}_\mathrm{Lip}$ that $\Term^1 \ge 0$. 
 Due to the assumption on $\phi$,
 \begin{equation*}
  \abs{\Term^2} \le \norm{f}_{\mathrm{Lip}}
  C_\phi\iint_{\Pi_T} S(v)\psi_\delta(t)H_\delta(R,\abs{x})\phi(x)\,dxdt.
 \end{equation*}
 Sending $\delta \downarrow 0$, inequality \eqref{eq:LplocCLIneq} then takes the form
 \begin{equation*}
  X(\tau) \le X(0) + \norm{f}_\mathrm{Lip}C_\phi \int_0^\tau X(r)\,dr,
 \end{equation*}
 where 
 \begin{equation*}
  X(t) = \int_{B(0,\Gamma(t))} S(v(t,x))\phi(x)\,dx.
 \end{equation*}
 Next, apply Gr\"onwall's inequality. 
 The estimate \eqref{eq:LpLocEstDetStep} 
 follows upon letting $S \rightarrow \abs{\cdot}^p$ and 
 applying the dominated convergence theorem.

 \textit{2.~Stochastic step.} We want to prove the following: Fix $2 \le p < \infty$. 
 Suppose $w(s) \in L^p(\Omega,\fcal_s,P;L^p_\mathrm{loc}(\R))$ and 
 take $w(t) = \ssde(t,s)w(s)$ for $s \leq t$. For any $R > 0$ there exist 
 constants $C_3$ and $C_2$ depending only on $p$ and $\sigma$ such that 
 \begin{multline}\label{eq:LpLocEstStochastic}
  \E{\int_B \abs{w(t,x)}^p\phi(x)\,dx} \le e^{C_3(t-s)}\Bigl(\E{\int_B \abs{w(s,x)}^p\phi(x)\,dx} \\
  + C_2(t-s)\int_B\phi(x)\,dx\Bigr).
 \end{multline}
 If $\sigma(x,0) = 0$, then $C_2 = 0$. 

 By Ito's lemma, 
 \begin{equation*}
  dS(w)=\frac{1}{2}S^{\prime\prime}(w) \sigma(x,w)^2\,dt 
  +S'(w) \sigma(x,w)\, dB,
 \end{equation*}
 for any $S \in C^2$.  Without loss of generality, we can assume $p=2,4,6,\ldots$. 
 Taking $S(u)=\abs{u}^p$, multiplying by $\phi$, and integrating over $B = B(0,R)$, we arrive at
 \begin{multline*}
   \E{\int_B \abs{w(t,x)}^p\phi(x)\,dx} - \E{\int_B\abs{w(s,x)}^p\phi(x)\,dx} \\
   \le \frac{p(p-1)}{2} \int_s^t \E{\int_{B}
     w(r,x)^{p-2}\sigma(x,w(r,x))^2\phi(x)\,dx} \,dr.
 \end{multline*}
 Recall that $\sigma(x,w) \le \abs{\sigma(x,0)} +
 \norm{\sigma}_{\mathrm{Lip}}\abs{w}$. Hence, according to
 assumption~\eqref{eq:sigma-ass},
 \begin{align*}
   \Term^3 &:= \frac{p(p-1)}{2}\E{\int_{B}w(r,x)^{p-2}
	\sigma(x,w(r,x))^2\phi(x) \, dx} \\
   & \le 
   p(p-1)\Bigl(\norm{\sigma(\cdot,0)}_\infty^2
	\E{\int_{B}\abs{w(r,x)}^{p-2}\phi(x)\,dx} 
   \\
   & \hphantom{\le 
     p(p-1)\Bigl(}\quad 
   +
   \norm{\sigma}_{\mathrm{Lip}}^2
	\E{\int_B\abs{w(r,x)}^p\phi(x)\,dx}\Bigr).
 \end{align*}
 Applying H\"older's inequlity with $\theta = \frac{p}{p-2}$ and $\theta' = \frac{p}{2}$,
 \begin{equation*}
   \int_B \underbrace{\left(\abs{w(r,x)}^p\phi(x)\right)^{\frac{1}{\theta}}
   \phi(x)^{\frac{1}{\theta'}}}_{\abs{w(r,x)}^{p-2}\phi(x)}\,dx 
   	\leq \underbrace{\left(\int_B \abs{w(r,x)}^p\phi(x)\,dx\right)^{\frac{1}{\theta}}}_A
	\underbrace{\left(\int_B\phi(x)\,dx\right)^{\frac{1}{\theta'}}}_B
 \end{equation*} 
 Due to Young's inequality $AB \leq \frac{1}{\theta}A^\theta + \frac{1}{\theta'}B^{\theta'}$. It follows that
 \begin{equation*}
   \int_{B}\abs{w(r,x)}^{p-2}\phi(x)\,dx \le
   \frac{p-2}{p}\int_B\abs{w(r,x)}^p\phi(x)\,dx +
   \frac{2}{p}\int_B\phi(x)\,dx. 
 \end{equation*}
 Consequently,
 \begin{multline*}
   \Term^3 \le \underbrace{(p-1)\left((p-2)
\norm{\sigma(\cdot,0)}_\infty +
p\norm{\sigma}_{\mathrm{Lip}}^2\right)}_{C_3}
\E{\int_B\abs{w(r,x)}^p\phi(x)\,dx}
\\
+ \underbrace{2(p-1)
\norm{\sigma(\cdot,0)}_\infty^2}_{C_2}\int_B
   \phi(x)\,dx.
 \end{multline*}
 It follows that 
 \begin{multline*}
   \E{\int_B \abs{w(t,x)}^p\phi(x)\,dx} 
	\le \E{\int_B \abs{w(s,x)}^p\phi(x)\,dx} \\
   + C_3 \int_s^t \E{\int_B \abs{w(r,x)}^p\phi(x)\,dx}\,dr +
   C_2\left(\int_B \phi(x)\,dx\right)(t-s).
 \end{multline*}
 This inequality is of the general form
 \begin{equation}\label{eq:X-eq}
	X(t) \le X(s)  + \int_s^t  K(r) X(r)\, dr + \int_s^t H(r)\,dr.
 \end{equation}
 Appealing to Gr\"onwall's inequality,
 \begin{equation}\label{eq:gronwall}
 	X(t) \le \exp\left[\int_s^t K(r)\, dr\right]  X(s) 
 	+ \int_s^t \exp\left[\int_r^t K(u)\, du\right]  
	H(r)\, dr.
 \end{equation}
 Identifying $K=C_3$ and $H=C_2\norm{\phi}_{L^1(B)}$, it follows that 
 \begin{multline*}
 	\E{\int_B \abs{w(t,x)}^p\phi(x)\,dx}
 	\le e^{C_3(t-s)}\E{\int_B 
	\abs{w(s,x)}^p\phi(x)\,dx} \\
 	+C_2\norm{\phi}_{L^1(B)}\int_s^t e^{C_3(t-r)}\,dr.
 \end{multline*}
 Next, observe that $e^{C_3(t-r)} \leq e^{C_3(t-s)}$ for 
 all $s \leq r \leq t$, and so \eqref{eq:LpLocEstStochastic} follows.

 \textit{3.~Inductive step.} Let $P_n$ be the statement 
that \eqref{eq:lpLoc-est} is true,
and note that $P_0$ is trivially true. We must show that $P_n$ implies $P_{n+1}$. 
 By \eqref{eq:SD-split}, $u^{n+1} = \ssde(t_{n+1},t_n)\scl(\Dt)u^n$. 
 Recall that $\vdt((t_{n+1})-) = \scl(\Dt)u^n$. By \eqref{eq:LpLocEstDetStep},
 \begin{multline*}
  \E{\int_{B(0,\Gamma(t_{n+1}))} \abs{\vdt((t_{n+1})-,x)}^p\phi(x)\,dx} \\
  \le e^{\norm{f}_\mathrm{Lip}C_\phi \Dt}\E{\int_{B(0,\Gamma(t_n))} \abs{u^n(x)}^p\phi(x) \,dx}.
 \end{multline*}
 Since $u^{n+1} = \ssde(t_{n+1},t_n)\vdt((t_{n+1})-)$ it 
 follows from \eqref{eq:LpLocEstStochastic} that
 \begin{align*}
  &\E{\int_{B(0,\Gamma(t_{n+1}))} \abs{u^{n+1}(x)}^p\phi(x) \,dx} \le e^{C_3\Dt}
   \\ & \quad 
   \times 
   \Bigg(\E{\int_{B(0,\Gamma(t_{n+1}))} \abs{\vdt((t_{n+1})-,x)}^p\phi(x)\,dx} 
  + C_2\int_{B(0,\Gamma(t_{n+1}))}\phi(x)\,dx\Dt\Bigg).
 \end{align*}
 Combining the two previous estimates,
 \begin{align*}
  &\E{\int_{B(0,\Gamma(t_{n+1}))} \abs{u^{n+1}(x)}^p\phi(x) \,dx} \\
  & \quad \le e^{C_3\Dt}
  \Bigg(e^{\norm{f}_\mathrm{Lip}C_\phi \Dt}\E{\int_{B(0,\Gamma(t_n))} \abs{u^n(x)}^p\phi(x) \,dx} \\
  & \hphantom{XXXXXXXXXXX}+ C_2\Dt\int_{B(0,\Gamma(t_{n+1}))}\phi(x)\,dx\Bigg) \\
  & \quad \le  e^{C_1\Dt}\Bigg(\E{\int_{B(0,\Gamma(t_n))} \abs{u^n(x)}^p\phi(x) \,dx} \\
  & \hphantom{XXXXXXXXXXX}+ C_2\Dt\int_{B(0,R)}\phi(x)\,dx\Bigg),
  \qquad C_1 = \norm{f}_\mathrm{Lip}C_\phi + C_3.
 \end{align*}
 Inserting the induction hypothesis brings 
 to an end the proof of \eqref{eq:lpLoc-est}.
\end{proof}

\begin{corollary}\label{cor:LpLocTimeIntPol}
 Let $\udt$ and $\vdt$ be defined by \eqref{eq:udt} 
and \eqref{eq:vdt}, respectively, and 
suppose $u^0$ belongs to $L^q(\Omega,\fcal_0,P;L^q(\R^d,\phi))$, $2 \le q < \infty$, 
 $\phi \in \mathfrak{N}$. Then, for each $1 \le p \le q$, there exists 
 a finite constant $C$ independent of $\Dt$ (but 
 dependent on $T,p,\phi,f,\sigma,u^0$) such that  
 \begin{equation*}
  \max\seq{\E{\norm{\udt(t)}_{p,\phi}^p},\E{\norm{\vdt(t)}_{p,\phi}^p}} \le C, \qquad t \in [0,T].
 \end{equation*}
\end{corollary}

\begin{proof}
It suffices to prove the result for $p = q$. To this end, suppose 
$1 \le p < q$ and $w \in L^q(\R^d,\phi)$. Let $r = q/p$, $r' = q/(q-p)$, so 
that $\frac{1}{r}+ \frac{1}{r'} = 1$. Take $f = \abs{u}^p\phi^{1/r}$, $g = \phi^{1/r'}$ 
and apply H\"older's inequality. The result is
\begin{equation}\label{eq:HolderFFMSpace}
  \int_{\R^d}\abs{w(x)}^p\phi(x) \,dx \le 
  \left(\int_{\R^d}\abs{w(x)}^q\phi(x)\,dx\right)^{p/q}
  \left(\int_{\R^d}\phi(x)\,dx\right)^{1-p/q}.
 \end{equation}
 Consider the case $p = q$. By Proposition~\ref{prop:lpLoc-est}, there 
 exists a constant $C > 0$ depending only on $q,f,\sigma,u^0,T,\phi$ such that  
 \begin{equation*}
  \E{\norm{u^n}_{q,\phi}^q} \le C, \qquad 0 \le n \le N.
 \end{equation*} 
 Let $t \in [t_n, t_{n+1})$. By \eqref{eq:LpLocEstDetStep},
 \begin{equation*}
  E\Bigl[\|\underbrace{\scl(t-t_n)u^n}_{\vdt(t)}\|_{q,\phi}^q\Bigr] 
  \le e^{\norm{f}_{\mathrm{Lip}}C_\phi\Dt}\E{\norm{u^n}_{q,\phi}^q}.
 \end{equation*}
This finishes the proof for $\vdt$. For $\udt$ the result follows 
by \eqref{eq:LpLocEstStochastic}.
\end{proof}

The next result should be compared 
to \cite[Proposition~5.2]{KarlsenStorroesten2014} and \cite[\S~6]{Chen:2011fk}. 
It can be turned into a fractional $BV_x$ estimate ($L^1$ space 
translation estimate) along the lines of \cite{Chen:2011fk}, but we 
will not need this fact here.

\begin{proposition}[fractional $BV_x$ estimates]\label{prop:fractional-BV}
Suppose \eqref{eq:fluxLip-ass}, \eqref{eq:fluxDoubleDerBounded-ass}, \eqref{eq:sigma-ass}, 
and \eqref{eq:sigma-ass-xdep} are satisfied. Let $\phi \in \mathfrak{N}$. 
Suppose $u^0 \in L^2(\Omega,\fcal_0,P;L^2(\R^d,\phi))$. Let $u_\Dt$ and $\vdt$ be 
defined by \eqref{eq:udt} and \eqref{eq:vdt}, respectively. 
Then there exists a constant $C_{T}$, independent of $\Dt$, such that 
\begin{align*}
	& \Eb{\;\,\iint\limits_{\R^d \times \R^d} 
	\abs{\udt(t,x+z)-\udt(t,x-z)} J_r(z)\phi(x) \,dx\, dz}
	\\ & \quad 
	\le e^{C_\phi \norm{f}_{\mathrm{Lip}}t} \Eb{\;\,\iint\limits_{\R^d\times \R^d} 
	\abs{u^0(x+z)-u^0(x-z))} J_r(z)\phi(x) \,dx\, dz}
	+C_{T} r^\kappas, 
\end{align*}
for any $t\in (0,T)$. Here $\kappas\in (0,1/2]$ is defined 
in \eqref{eq:sigma-ass-xdep}. If $\sigma(x,u) = \sigma(u)$, then 
we may take $C_{T} = 0$. The same result 
holds with $\udt$ replaced by $\vdt$.
\end{proposition}

\begin{remark}
In the deterministic case or whenever $\sigma=\sigma(u)$ is 
independent of the spatial location $x$, we recover the 
usual $BV$ bound. To this end, note that $C_T = 0$, apply 
the weight $\phi_\rho(x)=e^{-\rho\sqrt{1+\abs{x}^2}}$ ($\rho > 0$), 
and  then send $\rho \downarrow 0$. 
\end{remark}

Before we proceed to the proof, we fix some notation 
and make a few observations. Let us define 
$C^2$-approximations $\Set{S_\delta}_{\delta > 0}$ 
of the absolute value function by asking that
\begin{equation}\label{eq:SdeltaDef}
	S_\delta'(\sigma) = 2\int_0^\sigma J_\delta(z)\,dz,
	\qquad S_\delta(0) = 0.
\end{equation}
Then
\begin{equation}\label{eq:Sdelta-est}
	\abs{r}-\delta
	\le S_\delta(r)\le \abs{r},
	\qquad
	\abs{S_\delta''(r)}\le  \frac{2}{\delta}\norm{J}_\infty
	\mathbf{1}_{\abs{r}<\delta}.
\end{equation}
Given $S_\delta$, we define $Q_\delta$ by
\begin{equation}\label{eq:QdeltaDef}
	Q_\delta(u,v)=\int_v^u 
	S_\delta '(\xi-v)f'(\xi)\, d\xi,
	\qquad u,v\in \R.
\end{equation}
This function satisfies
\begin{equation}\label{eq:QdeltaDef-prop}
	\abs{\partial_u \left(Q_\delta(u,v)
	-Q_\delta(v,u)\right)}
	\le \norm{f^{\prime \prime}}_{L^\infty} \delta
\end{equation}
and
\begin{equation}\label{eq:QdeltaDef-bound}
	\abs{Q_\delta(u,v)} \le 
	\norm{f}_{\mathrm{Lip}} S_\delta(u-v). 
\end{equation}

Let us state two convenient 
identities. First, for $h = h(\cdot,\cdot)\in L^1_{\loc}$,
\begin{multline}\label{eq:ChVarBV}
 \frac{1}{2^d}\iint_{\R^d \times \R^d} h(x,y) 
\phi\left(\frac{x+y}{2}\right)J_r\left(\frac{x-y}{2}\right)\,dxdy \\
  = \iint_{\R^d \times \R^d} 
h(\tilde{x}+z,\tilde{x}-z)\phi(\tilde{x})J_r(z)\,d\tilde{x}dz.
\end{multline}
This follows by a change of variables: 
$(\tilde{x},z) = \left(\frac{x+y}{2},
\frac{x-y}{2}\right)$, $dy = 2^d dz$. Next,
\begin{equation}\label{eq:symmetricMollExpr}
 \frac{1}{2^d}\int_{\R^d}\phi\left(\frac{x+y}{2}\right)
J_r\left(\frac{x-y}{2}\right)\,dy 
= (\phi \star J_r)(x).
\end{equation} 

\begin{proof}[Proof of Proposition~\ref{prop:fractional-BV}]
  Given $u=u(t)=u(t,x;\omega)$, we introduce the
  quantity
  \begin{equation*}
      \mathcal{D}_{r}^u(t) := \Eb{\; \, 
        \frac{1}{2^d}\iint\limits_{\R^d\times\R^d} \abs{u(t,x)-u(t,y)}
        J_r(\tfrac{x-y}{2}) \phi(\tfrac{x+y}{2})\,dxdy \,}.
  \end{equation*}
  Actually, at first we are not going to work with this quantity but
  rather
  \begin{equation*}
      \mathcal{D}_{r,\delta}^u(t) := \Eb{\, \,
        \frac{1}{2^d}\iint\limits_{\R^d\times\R^d}
        S_\delta(u(t,x)-u(t,y)) J_r(\tfrac{x-y}{2})
        \phi(\tfrac{x+y}{2})\,dxdy\,},
  \end{equation*}
  where the regularized entropy $S_\delta$ is defined in
  \eqref{eq:SdeltaDef}.  In view of \eqref{eq:Sdelta-est} and
  \eqref{eq:symmetricMollExpr},
  \begin{equation}\label{eq:D-Drho}
    \abs{\mathcal{D}_{r}^u(t)-\mathcal{D}_{r,\delta}^u(t)} 
    \le \norm{\phi}_{L^1(\R^d)} \delta, 
    \qquad t>0.
  \end{equation}

  \medskip

  \textit{1.~Deterministic step.}  Let $v(t,x)$ be the unique entropy
  solution of \eqref{eq:CL}. We want to prove the
  following claim: There exists a constant $C_1$ depending only on $J$
  and $C_\phi$ such that for all $0 < r \le 1$,
  \begin{equation}\label{eq:claimDetStepFracBV}
    \mathcal{D}_{r,\delta}^v(t) \le 
	e^{C_\phi \norm{f}_{\mathrm{Lip}}t} \left(\mathcal{D}_{r,\delta}^v(0)
      +C_1 \norm{f''}_\infty\E{\norm{v_0}_{1,\phi}}\, t 
	\left(\frac{\delta}{r}\right)\right).
  \end{equation}

  Let $Q_\delta$ be defined in \eqref{eq:QdeltaDef}. Using
  the entropy inequalities and Kru{\v{z}}kov's method of doubling the
  variables, it follows in a standard way that for $t>0$
  \begin{equation*}
    \begin{split}
      &\frac{1}{2^d}\iint\limits_{\R^d\times \R^d}
      S_\delta(v(t,x)-v(t,y)) J_r(\tfrac{x-y}{2}) \phi(\tfrac{x+y}{2})
      \,dx\, dy \\ & \qquad\qquad -
      \frac{1}{2^d}\iint\limits_{\R^d\times\R^d}
      S_\delta(v_0(x)-v_0(y)) J_r(\tfrac{x-y}{2}) \phi(\tfrac{x+y}{2})
      \,dxdy \\ & \le \frac{1}{2^d}\int_0^t\,\,
      \iint\limits_{\R^d\times\R^d} Q_\delta(v(s,x),v(s,y)) \cdot
      \nabla \phi(\tfrac{x+y}{2}) J_r(\tfrac{x-y}{2}) \,dxdy\,ds \\
      & \qquad +\frac{1}{2^d}\int_0^t\,\,
      \iint\limits_{\R^d\times\R^d} \bigl(Q_\delta(v(s,y),v(s,x))
      -Q_\delta(v(s,x),v(s,y))\bigr) \\ &
      \qquad\qquad\qquad\qquad\qquad\qquad \cdot
      \nabla_y\Bigl(\phi(\tfrac{x+y}{2}) J_r(\tfrac{x-y}{2}) \Bigr)
      \,dxdy\,ds \\ & =: \Term_{\text{CL}}^1+\Term_{\text{CL}}^2.
    \end{split}
  \end{equation*}
  By \eqref{eq:QdeltaDef-bound},
  \begin{align*}
    \abs{\Term_{\text{CL}}^1}\le
    C_\phi\norm{f}_{\mathrm{Lip}}\frac{1}{2^d}\int_0^t\,\,
    \iint\limits_{\R^d\times\R^d} S_\delta(v(s,x)-v(s,y))
    J_r(\tfrac{x-y}{2}) \phi(\tfrac{x+y}{2}) \,dxdy\,ds.
  \end{align*}
  Consider $\Term_{\text{CL}}^2$. Thanks to \eqref{eq:QdeltaDef-prop},
  \begin{equation*}
    \abs{Q_\delta(v,u)-Q_\delta(u,v)} = \Bigl|\int_v^u \partial_\xi 
	\left(Q_\delta(\xi,v)-Q_\delta(v,\xi)\right)\,d\xi\Bigr|
    \le \norm{f''}_\infty\abs{u-v}\delta,
  \end{equation*}
  so that
  \begin{align*}
    \abs{\Term_{\text{CL}}^2} &\le \frac{\norm{f''}_\infty}{2} \,
    \delta \, \frac{1}{2^d}\int_0^t\,\, \iint\limits_{\R^d\times\R^d}
    \abs{v(s,x)-v(s,y)} \abs{\nabla J_r(\tfrac{x-y}{2})}
    \phi(\tfrac{x+y}{2}) \,dxdy\,ds \\ & \qquad +
    \frac{\norm{f''}_\infty}{2}\, \delta\, \frac{1}{2^d}\int_0^t\,\,
    \iint\limits_{\R^d\times\R^d} \abs{v(s,x)-v(s,y)}
    J_r(\tfrac{x-y}{2}) \abs{\nabla \phi(\tfrac{x+y}{2})}
    \,dxdy\, ds \\
    & =: \Term_{\text{CL}}^{2,1} 
	+ \Term_{\text{CL}}^{2,2}.
  \end{align*}
  Consider $\Term_{\text{CL}}^{2,1}$. Setting 
$\varphi_r(z) = \norm{\nabla
    J}_1^{-1}\frac{1}{r^d}
\abs{\nabla J(\frac{z}{r})}$, we write
$$
\abs{\nabla J_r\left(\frac{x-y}{2}\right)} 
= \norm{\nabla J}_1\frac{1}{r}
\varphi_r\left(\frac{x-y}{2}\right).
$$ 
By the triangle inequality and \eqref{eq:symmetricMollExpr},
  \begin{multline*}
    \frac{1}{2^d}\iint\limits_{\R^d\times\R^d}\abs{v(s,x)-v(s,y)}
	\abs{\nabla J_r(\tfrac{x-y}{2})}\phi(\tfrac{x+y}{2}) \,dxdy \\
    \le \norm{\nabla J}_1\frac{2}{r} \int_{\R^d} \abs{v(s,x)}(\phi
    \star \varphi_r)(x)\,dx = \norm{\nabla
      J}_1\frac{2}{r}\norm{v(s)}_{1,\phi \star \varphi_r}.
  \end{multline*}
Considering $\Term_{\text{CL}}^{2,2}$, with $\phi \in \mathfrak{N}$,
\begin{multline*}
    \frac{1}{2^d}\iint\limits_{\R^d\times\R^d}
	\abs{v(s,x)-v(s,y)}J_r(\tfrac{x-y}{2})
	\abs{\nabla \phi(\tfrac{x+y}{2})}\,dxdy \\
    \le 2C_\phi \int_{\R^d} \abs{v(s,x)}(\phi \star J_r)(x)\,dx =
    2C_\phi \norm{v(s)}_{1,\phi \star J_r}.
\end{multline*}
  By Lemma~\ref{lemma:ContMollWeightedNorm},
  \begin{equation*}
    \max\left\{\norm{v(s)}_{1,\phi \star \varphi_r},
	\norm{v(s)}_{1,\phi \star J_r}\right\}
\le \norm{v(s)}_{1,\phi}(1 + w_{1,\phi}(r)), 
  \end{equation*}
  where $w_{1,\phi}$ is defined in Lemma~\ref{lemma:PhiProp}. Hence,
  \begin{equation*}
    \abs{\Term^2_{\text{CL}}} \le 
\norm{f''}_\infty(1 + w_{1,\phi}(r))\left(\int_0^t\norm{v(s)}_{1,\phi}\,ds\right)
\left(\norm{\nabla J}_1\frac{1}{r} + C_\phi\right)\delta.
  \end{equation*}
In view of \eqref{eq:LpLocEstDetStep}, 
$\norm{v(s)}_{1,\phi} \le 
e^{\norm{f}_{\mathrm{Lip}}C_\phi s}\norm{v_0}_{1,\phi}$. 
Summarizing, 
  \begin{equation}\label{eq:entropy-vxvy-CL}
    \begin{split}
      &\frac{1}{2^d}\iint\limits_{\R^d\times \R^d}
      S_\delta(v(t,x)-v(t,y)) J_r(\tfrac{x-y}{2}) \phi(\tfrac{x+y}{2})
      \,dx\, dy \\ & \qquad\qquad
      -\frac{1}{2^d}\iint\limits_{\R^d\times\R^d}
      S_\delta(v_0(x)-v_0(y)) J_r(\tfrac{x-y}{2}) \phi(\tfrac{x+y}{2})
      \,dxdy \\ & \le
      \underbrace{C_\phi\norm{f}_{\mathrm{Lip}}}_{K}\int_0^t\frac{1}{2^d}\,
      \iint\limits_{\R^d\times\R^d}
      S_\delta(v(s,x)-v(s,y))J_r(\tfrac{x-y}{2})
      \phi(\tfrac{x+y}{2})\,dx\, dy\, ds \\ & \qquad + \int_0^t
      \underbrace{C_1 \norm{f''}_\infty\norm{v_0}_{1,\phi}
        e^{\norm{f}_{\mathrm{Lip}}C_\phi s}
        \Bigl(\frac{\delta}{r}\Bigr)}_{H(s)}\,ds,
    \end{split}
  \end{equation}
  where $C_1 = (1+w_{1,\phi}(1))(\norm{\nabla J}_1 + C_\phi) $. This inequality is of the form \eqref{eq:X-eq}. By Gr\"onwall's
  inequality \eqref{eq:gronwall},
  \begin{align*}
    & \iint\limits_{\R^d\times \R^d} S_\delta(v(t,x)-v(t,y))
    J_r(\tfrac{x-y}{2}) \phi(\tfrac{x+y}{2}) \,dxdy \\ & \qquad \le
    e^{C_\phi \norm{f}_{\mathrm{Lip}}t}\Bigg(\iint\limits_{\R^d\times
      \R^d} S_\delta(v_0(x)-v_0(y)) J_r(\tfrac{x-y}{2})
    \phi(\tfrac{x+y}{2}) \,dxdy \\ & \qquad \qquad
    \hphantom{XXXXXXXXXXXXX} + C_1\norm{f''}_\infty\norm{v_0}_{1,\phi} t
    \Bigl(\frac{\delta}{r}\Bigr)\Bigg) .
  \end{align*}
  This proves the claim \eqref{eq:claimDetStepFracBV} \medskip

  \textit{2.~Stochastic step.} Let $w(t)=\ssde(t,s)w(s)$. We will now
  derive an estimate for $w$ similar to \eqref{eq:entropy-vxvy-CL}: 
There exist constants
  $C_1$ and $C_2$, depending only on $J,\sigma,\phi$, such that
  \begin{equation}\label{eq:claimStochStepFracBV}
    \mathcal{D}^w_{r,\delta}(t) \le \mathcal{D}^w_{r,\delta}(s) + C_1 \frac{r^{2\kappa +1}}{\delta}\int_s^t 
\E{\norm{1 + \abs{w(\tau)}}_{2,\phi}^2}\,d\tau +  C_2(t-s)\delta,
  \end{equation}
  for all $0 \le r \le 1$. If $M_{\sigma} = 0$, then $C_1 = 0$.

  Since $w(t,x)-w(t,y)$ solves
$$
d (w(t,x)-w(t,y)) =\bigl(\sigma(x,w(t,x))-\sigma(y,w(t,y))\bigr)\, d
B(t),
$$
applying Ito's formula to $S_\delta(w(t,x)-w(t,y))$ yields
\begin{align*}
  &d S_\delta(w(t,x)-w(t,y)) \\ & \qquad =
  \frac12S_\delta^{\prime\prime}(w(t,x)-w(t,y))
  \bigl(\sigma(x,w(t,x))-\sigma(y,w(t,y))\bigr)^2\, dt, \\ &
  \qquad\qquad +S_\delta'(w(t,x)-w(t,y))
  \bigl(\sigma(x,w(t,x))-\sigma(y,w(t,y))\bigr)\, d B(t).
\end{align*}
Integrating against the test function
$\frac{1}{2^d}J_r(\tfrac{x-y}{2})\phi(\tfrac{x+y}{2})$, we arrive at
\begin{equation*}
  \begin{split}
    &\frac{1}{2^d}\iint\limits_{\R^d\times \R^d}
    S_\delta(w(t,x)-w(t,y)) J_r(\tfrac{x-y}{2}) \phi(\tfrac{x+y}{2})
    \,dxdy \\ & \qquad\qquad -
    \frac{1}{2^d}\iint\limits_{\R^d\times\R^d} S_\delta(w(s,x)-w(s,y))
    J_r(\tfrac{x-y}{2}) \phi(\tfrac{x+y}{2}) dxdy \\ & = \int_s^t
    \,\, \frac{1}{2^d}\iint\limits_{\R^d\times \R^d}
    \frac12S_\delta^{\prime\prime}(w(\tau,x)-w(\tau,y)) \\ & \qquad
    \qquad\qquad\quad \times
    \bigl(\sigma(x,w(\tau,x))-\sigma(y,w(\tau,y))\bigr)^2
    J_r(\tfrac{x-y}{2})\phi(\tfrac{x+y}{2}) \,dxdy\,d\tau \\ &
    \qquad + \int_s^t \,\, \frac{1}{2^d}\iint\limits_{\R^d\times \R^d}
    S_\delta'(w(\tau,x)-w(\tau,y))
    \bigl(\sigma(x,w(\tau,x))-\sigma(y,w(\tau,y))\bigr)\,dxdy\, d B(\tau)
    \\ & = \Term_{\text{SDE}}^1+\Term_{\text{SDE}}^2,
  \end{split}
\end{equation*}
where the $\Term_{\text{SDE}}^2$-term has zero expectation. Note that
\begin{equation*}
  \left(\sigma(x,u)-\sigma(y,v)\right)^2 \le 2\left(\sigma(x,u)-\sigma(x,v)\right)^2 
  + 2\left(\sigma(x,v)-\sigma(y,v)\right)^2,
\end{equation*}
for any $u,v \in \R$. We estimate the $\Term_{\text{SDE}}^1$-term as
follows:
\begin{equation*}
  \begin{split}
    E\Bigl[\abs{\Term_{\text{SDE}}^1}\Bigr] & \le 2E\Biggl[\int_s^t
    \,\, \frac{1}{2^d}\iint\limits_{\R^d\times \R^d}
    J_\delta(w(\tau,x)-w(\tau,y))
    \bigl(\sigma(x,w(\tau,x))-\sigma(y,w(\tau,x))\bigr)^2 \\ &
    \qquad\qquad \qquad\qquad\qquad\qquad \times J_r(\tfrac{x-y}{2})
    \phi(\tfrac{x+y}{2})\,dxdy\,d\tau \Biggr] \\ & 
    +2E\Biggl[ \int_s^t \,\, \frac{1}{2^d}\iint\limits_{\R^d\times
      \R^d} J_\delta(w(\tau,x)-w(\tau,y))
    \bigl(\sigma(y,w(\tau,x))-\sigma(y,w(\tau,y))\bigr)^2 \\ &
    \qquad\qquad \qquad\qquad\qquad\qquad \times J_r(\tfrac{x-y}{2})
    \phi(\tfrac{x+y}{2} )\,dxdy\,d\tau \Biggr]=: S_1+S_2.
  \end{split}
\end{equation*}
Regarding $S_1$, recall that $\abs{J_\delta} \le
\norm{J}_\infty/\delta$. By \eqref{eq:sigma-ass-xdep},
\begin{align*}
  \abs{S_1} & \le \norm{J}_\infty\frac{2}{\delta} E\Bigg[\int_s^t\, \,
    \frac{1}{2^d}\iint\limits_{\R^d\times\R^d}
    \abs{\sigma(x,w(\tau,x))-\sigma(y,w(\tau,x))}^2 \\
    &\hphantom{XXXXXXXXXXXXXXXXX} \times J_r(\tfrac{x-y}{2})
    \phi(\tfrac{x+y}{2}) \,dxdy\,d\tau\Bigg] \\
  & \le \norm{J}_\infty M_{\sigma}^2\frac{2}{\delta}
  E\Bigg[\int_s^t\,\, \frac{1}{2^d}\iint\limits_{\R^d\times\R^d}
    \abs{x-y}^{2\kappas + 1}(1 + \abs{w (\tau,x)})^2 \\
    &\hphantom{XXXXXXXXXXXXXXXXX}
     \times J_r(\tfrac{x-y}{2})\phi(\tfrac{x+y}{2}) \,dxdy\,d\tau\Bigg] \\
  & \le 2\norm{J}_\infty M_{\sigma}^2\, \frac{(2r)^{2\kappas +
      1}}{\delta} \int_s^t \E{\norm{1 + \abs{w(\tau)}}_{2,\phi \star
      J_r}^2}\,d\tau.
\end{align*}
By Lemma~\ref{lemma:ContMollWeightedNorm},
\begin{equation*}
  \norm{1 + \abs{w(\tau)}}_{2,\phi \star J_r}^2 \le \norm{1 +
    \abs{w(\tau)}}_{2,\phi}^2(1 + w_{1,\phi}(r)),
\end{equation*}
where $w_{1,\phi}$ is defined in Lemma~\ref{lemma:PhiProp}. It follows that
\begin{equation*}
  \abs{S_1} \le \underbrace{2^{2(\kappas + 1)}\norm{J}_\infty M_{\sigma}^2  (1 + w_{1,\phi}(1))}_{C_1}\int_s^t
  \E{\norm{1 + \abs{w(\tau)}}_{2,\phi}^2}\,d\tau\,\frac{r^{2\kappas
      + 1}}{\delta},
\end{equation*}
for all $0 < r \leq 1$. Consider $S_2$. Due to assumption \eqref{eq:sigma-ass},
\begin{equation*}
  J_\delta(w(\tau,x)-w(\tau,y))\left(\sigma(y,w(\tau,x))-\sigma(y,w(\tau,y))\right)^2
  \le \norm{\sigma}_{\mathrm{Lip}}^2\norm{J}_\infty\delta. 
\end{equation*}
Hence,
\begin{equation*}
  \abs{S_2} \le
  \underbrace{2\norm{\sigma}_{\mathrm{Lip}}^2\norm{J}_\infty\norm{\phi}_{L^1(\R^d)}}_{C_2}(t-s)\delta.
\end{equation*}
This proves \eqref{eq:claimStochStepFracBV} \medskip

\textit{3.~Inductive step.} Let $P_n$ be the
following claim: There exist constants $C_1,C_2,C_3$ depending
only on $J,\phi,\sigma$ such that for all $0 < r \le 1$,
\begin{multline}\label{eq:IndHypFracBV}
  \mathcal{D}^{u^n}_{r,\delta} \le
  e^{C_\phi\norm{f}_{\mathrm{Lip}}t_n}
  \Bigl(\mathcal{D}^{u^0}_{r,\delta} +C_3\norm{f''}_\infty
  \Bigl(\Dt\sum_{k=0}^{n-1}\E{\norm{u^{k}}_{1,\phi}}
	\Bigr)\frac{\delta}{r} \\
  +C_1\frac{r^{2\kappas +1}}{\delta}\int_0^{t_n} \E{\norm{1 +
      \abs{u_{\Dt}(t)}}_{2,\phi}^2}\,dt\, +C_2t_n\delta\Bigr).
\end{multline}
If $M_{\sigma} = 0$, then $C_1 = 0$. Note that $P_0$ is trivially
true. Assuming that $P_n$ is true, we want to verify 
$P_{n+1}$. Recall that $u^{n+1} = \ssde(t_{n+1},t_n)\scl(\Dt)u^n$. Let
$w^n = \scl(\Dt)u^n$ and note that $\ssde(t,t_n)w^n = u_\Dt(t)$ for
$t_n \le t <t_{n+1}$. As $1 \le e^{C_\phi\norm{f}_{\mathrm{Lip}}\Dt}$,
\eqref{eq:claimStochStepFracBV} yields
\begin{equation*}
  \mathcal{D}^{u^{n+1}}_{r,\delta} 
  \le \mathcal{D}^{w^n}_{r,\delta} 
  + e^{C_\phi\norm{f}_{\mathrm{Lip}}\Dt}
  \Bigl(C_1 \frac{r^{2\kappas + 1}}{\delta}\int_{t_n}^{t_{n+1}}
  \E{\norm{1 + \abs{u_\Dt(t)}}_{2,\phi}^2}\,dt +
  C_2\Dt\delta\Bigr). 
\end{equation*}
By \eqref{eq:claimDetStepFracBV},
\begin{equation*}
  \mathcal{D}_{r,\delta}^{w^n} \le 
	e^{C_\phi \norm{f}_{\mathrm{Lip}}\Dt}\left(\mathcal{D}_{r,\delta}^{u^n}
	+C_3 \norm{f''}_\infty\E{\norm{u^n}_{1,\phi}}\, 
	\Dt \left(\frac{\delta}{r} \right)\right).
\end{equation*}
Hence,
\begin{multline*}
  \mathcal{D}^{u^{n+1}}_{r,\delta} \le e^{C_\phi
    \norm{f}_{\mathrm{Lip}}\Dt} \Bigg(\mathcal{D}_{r,\delta}^{u^n}
  +C_3 \norm{f''}_\infty\E{\norm{u^n}_{1,\phi}}\, 
	\Dt \left(\frac{\delta}{r} \right) \\
  + C_1 \frac{r^{2\kappas +1}}{\delta}\int_{t_n}^{t_{n+1}} \E{\norm{1
      + \abs{u_\Dt(t)}}_{2,\phi}^2}\,dt + C_2\Dt\delta\Bigg),
\end{multline*}
and inserting the hypothesis $P_n$ yields $P_{n+1}$.

\medskip \textit{4.~Concluding the proof.} 
Consider \eqref{eq:IndHypFracBV}. By Corollary~\ref{cor:LpLocTimeIntPol}, there 
exists a constant $C$, independent of $\Dt$, such that 
\begin{equation*}
  \mathcal{D}_{r,\delta}^{u^n} 
  \le e^{C_\phi\norm{f}_{\mathrm{Lip}} t_n}\left(\mathcal{D}_{r,\delta}^{u^0}
  + C \, t_n\, \left( \frac{\delta}{r}+\delta 
    + \frac{r^{2\kappas+1}}{\delta} \right)\right).
\end{equation*}
Due to \eqref{eq:D-Drho}, this translates into
\begin{equation*}
  \mathcal{D}_{r}^{u^n} 
  \le  e^{C_\phi\norm{f}_{\mathrm{Lip}} t_n}\left(\mathcal{D}_{r}^{u^0}
  + C t_n \left( \frac{\delta}{r}+\delta 
    + \frac{r^{2\kappas +1}}{\delta}\right)+2\norm{\phi}_{L^1(\R^d)} \delta \right), 
  \qquad 0 \le n \le N.
\end{equation*}
We can argue via \eqref{eq:claimStochStepFracBV} to obtain
\begin{equation*}
 \mathcal{D}_{r}^{\udt}(t) 
  \le  e^{C_\phi\norm{f}_{\mathrm{Lip}} t}\left(\mathcal{D}_{r}^{u^0}
  + C t \left( \frac{\delta}{r}+\delta 
    + \frac{r^{2\kappas +1}}{\delta}\right)+2\norm{\phi}_{L^1(\R^d)} \delta \right), 
  \qquad t \in [0,T].
\end{equation*}
Note that the same holds true if we replace $\udt$ by $\vdt$, thanks to
\eqref{eq:claimDetStepFracBV}. Viewing $r > 0$ as fixed, we can choose
$\delta = r^{\kappas +1}$ to arrive at the bound
\begin{equation*}
  \mathcal{D}_{r}^{\udt}(t) \le e^{C_\phi\norm{f}_{\mathrm{Lip}} t}\mathcal{D}_{r}^{u^0} + C_T
  r^{\kappas}.
\end{equation*}
The result follows by \eqref{eq:ChVarBV}. 
In the case that $M_{\sigma} = 0$, we have
\begin{equation*}
  \mathcal{D}_{r}^{\udt}(t) 
  \le e^{C_\phi\norm{f}_{\mathrm{Lip}} t} \mathcal{D}_{r}^{u^0}
  + C_T  \left( \frac{\delta}{r}+\delta \right), 
  \qquad t\in (0,T),
\end{equation*}
and we may send $\delta \downarrow 0$ independently of $r$.
\end{proof}

In Proposition~\ref{prop:fractional-BV} the 
spatial regularity of $\udt, \vdt$ is characterized in terms 
of averaged $L^1$ space translates. In the $BV$ context, this is 
equivalently characterized by integration against the 
divergence of a smooth bounded function. 
Restricting to one dimension ($d = 1$) and $u \in C^1(\R)$, we have
\begin{align*}
	&\sup_{h > 0} \left\{\frac{1}{h}\int_\R \abs{u(x + h)-u(x)}\,dx\right\} 
	\\ & \quad
	= \int_\R \abs{u'(x)}\,dx 
  	 = \sup \left\{\int_{\R}u(x) \beta'(x) 
	 \, dx \;  : \; \beta \in C^\infty_c(\R), 
	 \, \norm{\beta}_\infty \leq 1\right\}.
\end{align*}
Fix $\kappa \in (0,1]$. The left-hand side has a natural 
generalization to the fractional $BV$ setting 
by considering $u \in L^1(\R)$ satisfying
\begin{equation}\label{eq:gen_left}
	\sup_{h > 0} \left\{\frac{1}{h^\kappa}
	\int_\R \abs{u(x + h)-u(x)}\,dx\right\} < \infty.
\end{equation} 
A possible generalization of the right-hand side reads
\begin{equation}\label{eq:gen_right}
	\sup \left\{\delta^{1-\kappa} \int_{\R}u(x) (J_\delta \star \beta)'(x) 
	\, dx \;  : \;  \delta > 0, \, \norm{\beta}_\infty \leq 1\right\} < \infty,
\end{equation}
where $\seq{J_\delta}_{\delta > 0}$ is a suitable family of symmetric mollifiers. 
Loosely speaking, the next lemma shows that \eqref{eq:gen_right} 
may be bounded in terms of \eqref{eq:gen_left}. The lemma plays a 
key role in obtaining the optimal $L^1$ time continuity 
estimates in Proposition~\ref{prop:time-cont}.

\begin{lemma}\label{lemma:FracBVBoundOnMoll}
Let $\rho \in C^\infty_c((0,1))$ satisfy $\int_0^1 \rho(r)\,dr = 1$ and $\rho \geq 0$. 
For $x \in \R^d$ define
\begin{displaymath}
U(x) = \frac{1}{\alpha(d)M_d }\left(1-\int_0^{\abs{x}} \rho(r)\,dr \right), 
\qquad V(x) = \frac{1}{d\alpha(d)M_{d-1}}\rho(\abs{x}),
\end{displaymath}
where $M_n = \int_0^\infty r^n\rho(r)\,dr$, $n \geq 0$ and 
$\alpha(d)$ denotes the volume of the unit ball in $\R^d$. Then $U,V$ are symmetric 
mollifiers on $\R^d$ with support in $B(0,1)$. For $\phi \in \mathfrak{N}$, 
$u \in L^1(\R^d,\phi)$, and $\delta > 0$, define
\begin{displaymath}
\mathcal{V}_\delta(u) = \iint_{\R^d \times \R^d} 
\abs{u(x+z)-u(x-z)}V_\delta(z)\phi(x)\,dzdx,
\end{displaymath}
where $V_\delta(z) = \delta^{-d}V(\delta^{-1}z)$. 
Similarly, for $\beta \in \nobreak L^\infty(\R^d)$ let
\begin{displaymath}
\mathcal{U}_\delta^i(u,\beta) = \int_{\R^d} u(x)
\partial_{x_i}(U_\delta \star \beta)(x)\phi(x)\,dx, 
\qquad 1 \leq i \leq d,
\end{displaymath}
where $U_\delta(z) = \delta^{-d}U(\delta^{-1}z)$. Then
\begin{displaymath}
\abs{\mathcal{U}_\delta^i(u,\beta)} \leq 
\frac{dM_{d-1}}{2M_d}\left(\frac{1}{\delta}\mathcal{V}_\delta(u) 
+ 2\norm{u}_{1,\phi}\frac{w_{1,\phi}(\delta)}{\delta}\right)\norm{\beta}_{L^\infty},
\end{displaymath}
for each $1 \leq i \leq d$, where $w_{1,\phi}$ is 
defined in Lemma~\ref{lemma:PhiProp}. 
 \end{lemma}

 \begin{remark}\label{remark:spatToBV}
We note that Lemma \ref{lemma:FracBVBoundOnMoll} covers 
the $BV$ case. If there is a constant $C \geq 0$ 
such that $\mathcal{V}_\delta(u) \leq C \delta$ 
(the $BV$ case), then
\begin{displaymath}
\int_{\R^d} u(\nabla \cdot \beta)\phi \,dx = \lim_{\delta \downarrow 0} 
\sum_{i = 1}^d \mathcal{U}_\delta^i(u,\beta^i) 
\leq \frac{d^2M_{d-1}}{2M_d}\left(C + 2C_\phi\norm{u}_{1,\phi}\right),
\end{displaymath}
for any $\beta = (\beta^1, \dots,\beta^d) \in C^1_c(\R^d;\R^d)$ 
satisfying $\norm{\beta}_\infty \leq 1$. It follows that 
\begin{align*}
	\int_{\R^d} \abs{\nabla u} \phi\,dx &\leq 
	\sup_{\abs{\beta} \leq 1} \int_{\R^d} (\nabla u \cdot \beta)\phi \,dx \\
	&= \sup_{\abs{\beta} \leq 1} \int_{\R^d} u(\nabla \cdot \beta)\phi + u(\beta \cdot \nabla \phi)\,dx \\
	& \leq \frac{d^2M_{d-1}}{2M_d}\left(C + 2C_\phi\norm{u}_{1,\phi}\right) 
	+ C_\phi \norm{u}_{1,\phi},
\end{align*}
and so $\abs{\nabla u}$ is a finite measure with respect to $\phi\, dx$.
\end{remark}

\begin{proof}
Let us first show that $U$ is a symmetric mollifier. It is clearly symmetric, furthermore 
it is smooth since $\seq{0} \notin \mathrm{cl}(\mathrm{supp}(\rho))$. 
Change to polar coordinates and integrate by parts to obtain
\begin{align*}
\int_{\R^d} \left(1-\int_0^{\abs{x}} \rho(\sigma)d\sigma\right)\,dx 
&= \alpha(d)\int_0^\infty dr^{d-1}\left(1-\int_0^r \rho(\sigma)d\sigma\right)\,dr \\
&= \alpha(d)\int_0^\infty r^d\rho(r)\,dr = \alpha(d)M_d. 
\end{align*}
Similarly for $V$,   
\begin{displaymath}
\int_{\R^d}\rho(\abs{x})\,dx 
= d\alpha(d)\int_0^\infty r^{d-1}\rho(r)dr 
= d \alpha(d)M_{d-1}.
\end{displaymath}
Note that 
\begin{displaymath}
\mathcal{U}_\delta^i(u,\beta) = \iint_{\R^d \times \R^d} 
u(x)\partial_{x_i}U_\delta(x-y)\beta(y)\phi(x)\,dydx.
\end{displaymath}
Next, we differentiate to obtain
\begin{displaymath}
\partial_{x_i}U_\delta(x) = -\frac{1}{\alpha(d)M_d}\frac{1}{\delta^d}
\rho\left(\frac{\abs{x}}{\delta}\right) \frac{1}{\delta}\sgn{x_i} = 
-\frac{dM_{d-1}}{M_d}V_\delta(x)\frac{1}{\delta}\sgn{x_i}.
\end{displaymath}
Hence
\begin{displaymath}
\mathcal{U}_\delta^i(u,\beta) = -\frac{dM_{d-1}}{M_d}
\frac{1}{\delta}\iint_{\R^d \times \R^d} u(x)V_\delta(x-y)\sgn{x_i-y_i}\beta(y)\,dydx.
\end{displaymath}
This integral may be reformulated according to
\begin{align*}
&\iint_{\R^d \times \R^d}u(x)V_\delta(x-y)\sgn{x_i-y_i}\beta(y)\phi(x)\,dydx \\
&= \frac{1}{2}\iint_{\R^d \times \R^d}u(x)V_\delta(x-y)\sgn{x_i-y_i}\beta(y)\phi(x)\,dydx \\
& \quad -\frac{1}{2}\iint_{\R^d \times \R^d}u(x)V_\delta(y-x)\sgn{y_i-x_i}\beta(y)\phi(x)\,dydx \\
&= \frac{1}{2}\iint_{\R^d \times \R^d}(u(y-z)\phi(y-z)-u(y+z)\phi(y+z))V_\delta(z)\sgn{z_i}\beta(y)\,dzdy,
\end{align*}
where we made the substitution $x = y-z$ and $x = y + z$ respectively.
Since
\begin{multline*}
u(y-z)\phi(y-z)-u(y+z)\phi(y+z) = (u(y-z)-u(y+z))\phi(y) \\
				  + u(y-z)(\phi(y-z)-\phi(y))
				  -u(y+z)(\phi(y+z)-\phi(y)),
\end{multline*}
it follows that
\begin{align*}
\mathcal{U}_\delta^i(u,\beta) &= \frac{dM_{d-1}}{2M_d}
\frac{1}{\delta}\iint_{\R^d \times \R^d}(u(y+z)-u(y-z))V_\delta(z)\sgn{z_i}\beta(y)\phi(y)\,dzdy \\
	       & \quad+\frac{dM_{d-1}}{2M_d}\frac{1}{\delta}\iint_{\R^d \times \R^d}
	       u(y+z)(\phi(y+z)-\phi(y))V_\delta(z)\sgn{z_i}\beta(y)\,dzdy \\
	       & \quad+\frac{dM_{d-1}}{2M_d}\frac{1}{\delta}\iint_{\R^d \times \R^d}
	       u(y-z)(\phi(y)-\phi(y-z))V_\delta(z)\sgn{z_i}\beta(y)\,dzdy \\
	       &=:\Zerm_\delta^1 + \Zerm_\delta^2 + \Zerm_\delta^3.
\end{align*}
Clearly
\begin{displaymath}
\abs{\Zerm_\delta^1} \leq \frac{dM_{d-1}}{2M_d}\frac{1}{\delta}\underbrace{\iint_{\R^d \times \R^d} 
\abs{u(y+z)-u(y-z)}V_\delta(z)\phi(y)\,dzdy}_{\mathcal{V}_\delta(u)} \norm{\beta}_{L^\infty}.
\end{displaymath}
Consider $\Zerm_\delta^2$, the term $\Zerm_\delta^3$ is treated similarly. 
By Lemma~\ref{lemma:PhiProp} 
\begin{displaymath}
\abs{\phi(y + z)-\phi(y)} \leq w_{1,\phi}(\abs{z})\phi(y+z).
\end{displaymath}
Hence, by Young's inequality for convolutions,
\begin{displaymath}
\abs{\Zerm_\delta^2} \leq \frac{dM_{d-1}}{2M_d}
\frac{w_{1,\phi}(\delta)}{\delta}\int_{\R^d}(\abs{u\phi} \star 
V_\delta)(y)\,dy\norm{\beta}_{L^\infty} \leq 
\frac{dM_{d-1}}{2M_d}\frac{w_{1,\phi}(\delta)}{\delta} \norm{u}_{1,\phi}.
\end{displaymath}
This concludes the proof of the lemma.
\end{proof}

Next, we consider the time continuity of the splitting 
approximations. Recall that the interpolants
$\udt,\vdt$ are discontinuous at $t_n = n\Dt$. Hence, the result must 
somehow quantify the size of the jumps as $\Dt \downarrow 0$.
The idea of the proof is to ``transfer \textit{\`{a} la} Kru{\v{z}}kov"  
spatial regularity to temporal continuity \cite{Kruzkov:1969th,Kruzkov:1970kx}. 
Given a bounded variation bound, or some 
spatial $L^1$ modulus of continuity, this approach 
has been applied to miscellaneous splitting methods 
for deterministic problems, cf.~\cite{Holden:2010fk} (and references therein). 
At variance with \cite{Holden:2010fk}, we quantify spatial 
regularity differently, namely in terms of averaged (weighted) $L^1$ translates. 
Combined with Lemma \ref{lemma:FracBVBoundOnMoll}, we deduce 
$L^1$ time continuity estimates that recover
the optimal estimates in the $BV_x$ case ($\kappa=1$).

\begin{proposition}[$L^1$ time continuity]\label{prop:time-cont}
  Assume that \eqref{eq:fluxLip-ass}, \eqref{eq:fluxDoubleDerBounded-ass},
  \eqref{eq:sigma-ass}, and \eqref{eq:sigma-ass-xdep} 
hold. Fix $\phi \in \mathfrak{N}$, and let $u^0 \in
  L^2(\Omega,\fcal_0,P;L^2(\R^d,\phi))$ satisfy
  \begin{equation}\label{eq:AssumptionFracBVu0}
    \Eb{\;\,\iint\limits_{\R^d\times \R^d}
	\abs{u^0(x+z)-u^0(x-z))}
      J_r(z)\phi(x) \,dx\, dz} = \ocal(r^\kappao), 
  \end{equation}
  for any symmetric mollifier $J$ and some $0 < \kappao \leq 1$. Set 
  \begin{displaymath}
   \kappa := 
		\begin{cases}
         \min\seq{\kappao, \kappas}
			 &\mbox{ if $\sigma = \sigma(x,u)$,} \\
	      \kappao &\mbox{ if $\sigma = \sigma(u)$}.
     \end{cases}
  \end{displaymath}
Let $\udt$ and $\vdt$ be defined in 
\eqref{eq:udt} and \eqref{eq:vdt}, respectively. 
Then:
\begin{itemize}
  \item[(i)] Suppose $0 < \tau_1 < \tau_2 \le T$ satisfy $\tau_1\in
    (t_k,t_{k+1}]$ and $\tau_2\in (t_l,t_{l+1}]$. Then there exists a
    finite constant $C_{T,\phi}$, independent of $\Dt$, such that
    \begin{equation*}
      \E{\int_{\R^d}\abs{\udt(\tau_2,x)-\udt(\tau_1,x)}\phi(x)\,dx}
      \le C_{T,\phi}\Bigl(\abs{(l-k)\Dt}^\kappa
      + \sqrt{\tau_2-\tau_1}\Bigr). 
    \end{equation*}
  \item[(ii)] Suppose $0 \le \tau_1 \le \tau_2 < T$ satisfy
    $\tau_1\in [t_k,t_{k+1})$ and $\tau_2\in [t_l,t_{l+1})$. Then
    there exists a finite constant $C_{T,\phi}$, independent of $\Dt$,
    such that
    \begin{equation*}
      \E{\int_{\R^d}\abs{\vdt(\tau_2,x)-\vdt(\tau_1,x)}\phi(x)\,dx}
      \le C_{T,\phi}\Bigl(\sqrt{(l-k)\Dt} +
      \abs{\tau_2-\tau_1}^\kappa\Bigr).
    \end{equation*}
  \end{itemize}
\end{proposition}

\begin{proof}[Proof of Proposition~\ref{prop:time-cont}.]
We shall first quantify weak continuity in the mean 
of $t\mapsto \udt(t)$, $t\mapsto \vdt(t)$, and then turn this into 
fractional $L^1$ time continuity in the mean.
The reason for first exhibiting a weak estimate is that 
the splitting steps do not produce functions that are 
Lipschitz continuous in time, thereby preventing a 
direct ``inductive argument", see \cite{Kruzkov:1969th}.

\medskip
\textit{1. Weak estimate.} Let $t_n = n \Dt$. 
Suppose $0 < \tau_1 \le \tau_2 \le T$ satisfies 
$\tau_1\in (t_k,t_{k+1}]$ and $\tau_2\in (t_l,t_{l+1}]$. 
Suppose $\beta$ belongs to $L^\infty(\Omega \times \R^d,\fcal \otimes \Borel{\R^d},dP \otimes dx)$ 
and let $\beta_\delta = \beta \star U_\delta$, where $U_\delta$ is 
defined in Lemma~\ref{lemma:FracBVBoundOnMoll}. We claim 
that there is a constant $C > 0$, independent of $\Dt$, such that 
\begin{equation}\label{eq:weak-time}
	\begin{split}
		& \E{\int_{\R^d} 
		\Bigl(\udt(\tau_2,x)-\udt(\tau_1,x)\Bigr)
		(\beta_\delta\phi)(x)\, dx}
		\\ & \hphantom{XXXXXXXXXXXX} \le 
		C\left(\delta^{\kappa-1}(l-k)\Dt
		+ \sqrt{\tau_2-\tau_1}\right)\norm{\beta}_{L^\infty}.
	\end{split}
\end{equation}
Consider the case $l \geq k+1$. We continue as follows:
\begin{align*}
	\Term &= \E{\int_{\R^d}   \Bigl(\udt(\tau_2,x)-\udt(\tau_1,x)\Bigr)(\beta_\delta\phi)(x)\, dx}
	\\ & 
	= \E{\int_{\R^d}   \Bigl(\udt(\tau_2,x)-\udt((t_l)+,x)\Bigr)(\beta_\delta\phi)(x)\, dx} 
	\\ & \qquad
	+\E{\int_{\R^d}   \Bigl(\udt(t_{k+1},x) - \udt(\tau_1,x)\Bigr)(\beta_\delta\phi)(x)\, dx}
	\\ & \qquad
	+ \E{\int_{\R^d}   \Bigl(\udt((t_l)+,x)-\udt(t_{k+1},x)\Bigr)(\beta_\delta\phi)(x)\, dx}
	\\ & \quad =: \Term_1+\Term_2+\Term_3.
\end{align*}
Recall that $\udt((t_n)+) = \vdt((t_{n+1})-) = \scl(\Dt)u^n$. Regarding the last term,
\begin{multline*}
 \udt((t_l)+,x)-\udt(t_{k+1},x) = 
\vdt((t_{l+1})-,x)-\vdt(t_l,x) \\
			  + \sum_{n = k+1}^{l-1} \udt(t_{n+1},x)-\udt(t_{n},x),
\end{multline*}
where the sum is empty for the case $l = k+1$. Furthermore, we note that 
\begin{multline*}
 \udt(t_{n+1},x)-\udt(t_{n},x) = (\udt(t_{n+1},x)-\udt((t_n)+,x)) \\
			+ (\vdt((t_{n+1})-,x)-\vdt(t_n,x)).
\end{multline*}
This yields
\begin{align*}
	\Term_3&=\E{\sum_{n={k+1}}^{l-1} 
	\int_{\R^d}  
	\Bigl(\udt(t_{n+1},x)-\udt((t_n)+,x)\Bigr)(\beta_\delta\phi)(x)\,dx}
	\\ & \quad  
	+ \E{\sum_{n={k+1}}^{l}\int_{\R^d} 
	\Bigl(\vdt((t_{n+1})-,x)-\vdt(t_n,x)\Bigr) (\beta_\delta\phi)(x)\,dx} 
	\\ &
	=\E{\int_{\R^d}\Bigl(\int_{t_{k+1}}^{t_l}
	\sigma(x,\udt(t,x))\,dB(t)\Bigr)(\beta_\delta\phi)(x)\,dx} \\
	& \quad
	+\E{\sum_{n=k+1}^{l} \int_{\R^d}  
	\Bigl(\scl(\Dt)u^n(x)-u^n(x)\Bigr)(\beta_\delta\phi)(x)\,dx}.
\end{align*}
It follows that $\Term = \Term^{\text{CL}} + \Term^{\text{SDE}}$, where
\begin{align*}
	\Term^{\text{CL}} &:=  
	E\left[\sum_{n=k+1}^{l} \int_{\R^d}  
	\Bigl(\scl(\Dt)u^n(x)-u^n(x)\Bigr)(\beta_\delta\phi)(x)\,dx\right],
	\\ \Term^{\text{SDE}} &:= E\left[\int_{\R^d}\int_{\tau_1}^{\tau_2}
	\sigma(x,\udt(t,x))\,dB(t)(\beta_\delta\phi)(x)\,dx\right].
\end{align*}
Note that this holds true for $k = l$ as $\Term^{\text{CL}} = 0$ in 
this case. As $\vdt(t,x)$ is a weak solution of the conservation 
law \eqref{eq:CL} on $[t_n,t_{n+1})$ 
\begin{equation*}
	\begin{split}
		&\abs{\int_{\R^d} 
		\bigl(\vdt((t_{n+1})-,x)-\vdt(t_n,x)\bigr)(\beta_\delta\phi)(x)\,dx} \\
		&\hphantom{XXXXXXXXXXX}
		\le \abs{\int_{t_n}^{t_{n+1}}\int_{\R^d} f(\vdt(r,x)) 
		\cdot \nabla (\beta_\delta\phi)(x) \,dxdr} \\
		&\hphantom{XXXXXXXXXXX}
		\le \abs{\int_{t_n}^{t_{n+1}}\int_{\R^d} f(\vdt(r,x)) 
		\cdot \nabla\beta_\delta(x)\,\phi(x) \,dxdr} \\
		&\hphantom{XXXXXXXXXXX}
		\quad+ \abs{\int_{t_n}^{t_{n+1}}\int_{\R^d} f(\vdt(r,x)) 
		\cdot (\beta_\delta(x)\nabla\phi(x)) \,dxdr} \\
		&\hphantom{XXXXXXXXXXX}
		=:\Zerm_\delta^1 + \Zerm_\delta^2.
	\end{split}
\end{equation*}
By Proposition~\ref{prop:fractional-BV}, there exists a constant $C > 0$ such that 
\begin{displaymath}
 \E{\iint_{\R^d \times \R^d} \abs{\vdt(r,x+z)-\vdt(r,x-z)}
 V_\delta(z)\phi(x)\,dzdx} \leq C\delta^\kappa.
\end{displaymath}
Consequently, taking expectations 
in Lemma~\ref{lemma:FracBVBoundOnMoll} yields 
\begin{multline*}
  \E{\Zerm_\delta^1} \leq \frac{d^2M_{d-1}}{2M_d}
  \norm{f}_\mathrm{Lip}\bigg(\Dt C\delta^{\kappa-1} \\
  + 2\E{\int_{t_n}^{t_{n+1}}\norm{\vdt(r)}_{1,\phi}\,dr}
  \delta^{-1}w_{1,\phi}(2\delta)\bigg)\norm{\beta}_{L^\infty}.
 \end{multline*}
As $\phi \in \mathfrak{N}$,
\begin{equation*}
 \Zerm_\delta^2 \le \norm{f}_{\mathrm{Lip}}C_\phi\E{\int_{t_{k+1}}^{t_{l+1}} 
\norm{\vdt(t)}_{1,\phi}\,dt}\norm{\beta}_{L^\infty}.
\end{equation*}
Summarizing, there exists a constant $C$ such that
\begin{equation*}
 \abs{\Term^{\text{CL}}} \le C\delta^{\kappa-1}(l-k)\Dt\norm{\beta}_{L^\infty},
\end{equation*}
for all $0 < \delta \le 1$.

By \eqref{eq:HolderFFMSpace}, Jensen's 
inequality, and the It\^o isometry,
\begin{align*}
 \abs{\Term^{\text{SDE}}}
& \le \norm{\beta}_{L^\infty}
\int_{\R^d}E\Bigg[\abs{\int_{\tau_1}^{\tau_2}
\sigma(x,\udt(t,x))\,dB(t)}\Bigg]\phi(x)\,dx \\
& \le \norm{\beta}_{L^\infty}\norm{\phi}_{L^1(\R^d)}^{1/2}\left(\int_{\R^d}
E\Bigg[\abs{\int_{\tau_1}^{\tau_2}\sigma(x,\udt(t,x))\,dB(t)}\Bigg]^2\phi(x)\,dx\right)^{1/2} \\
& \le \norm{\beta}_{L^\infty}\norm{\phi}_{L^1(\R^d)}^{1/2}\left(\int_{\R^d}
E\Bigg[\int_{\tau_1}^{\tau_2}\sigma^2(x,\udt(t,x))\,dt\Bigg]
\phi(x)\,dx\right)^{1/2} \\
& = \norm{\beta}_{L^\infty}\norm{\phi}_{L^1(\R^d)}^{1/2}
\left(\int_{\tau_1}^{\tau_2}\E{\norm{\sigma(\cdot,\udt(t,\cdot))}_{2,\phi}^2}\,dt\right)^{1/2} \\
& \leq C\norm{\beta}_{L^\infty}\norm{\phi}_{L^1(\R^d)}^{1/2}\sqrt{\tau_2-\tau_1},
\end{align*}
since, in view of \eqref{eq:sigma-ass} and Corollary~\ref{cor:LpLocTimeIntPol}, $\E{\norm{\sigma(\cdot,\udt(t,\cdot))}^2_{2,\phi}}^{1/2} \leq C$ for some constant $C$ independent of $t \in [0,T]$. Summarizing, the 
above estimates imply the existence of a constant $C$, independent of $\Dt,\delta$ and $\beta$, such that
\begin{align*}
	& \abs{\Term} 
	\le C\left(\delta^{\kappa-1}(k-l)\Dt
	+ \sqrt{\tau_2-\tau_1}\right)\norm{\beta}_{L^\infty},
\end{align*}
which yields \eqref{eq:weak-time}. 

Let us consider $\vdt$. Suppose 
$0 \le \tau_1 \le \tau_2 < T$, with 
$\tau_1 \in [t_k,t_{k+1})$, $\tau_2 \in [t_l,t_{l+1})$.
We claim there is a constant $C > 0$, independent 
of $\Dt,\delta$ and $\beta$, such that 
\begin{equation}\label{eq:weak-timevDt}
	\begin{split}
		& \E{\int_{\R^d} 
		\Bigl(\vdt(\tau_2,x)-\vdt(\tau_1,x)\Bigr)
		(\beta_\delta\phi)(x)\, dx}
		\\ & \hphantom{XXXXXXXXXXX} \le 
		C\left(\delta^{\kappa-1}\abs{\tau_2 -\tau_1}
		+ \sqrt{(l-k)\Dt}\right)\norm{\beta}_{L^\infty}.
	\end{split}
\end{equation}
To prove this claim, note that
\begin{align*}
 \vdt(\tau_2,x)-\vdt(\tau_1,x) &= \vdt(\tau_2,x)-\vdt(t_l,x) \\ 
& \quad + \sum_{n=k+1}^{l} \vdt(t_n,x)-\vdt((t_n)-,x) \\
& \quad + \sum_{n=k+1}^{l-1} \vdt((t_{n+1})-,x)-\vdt(t_n,x) \\
& \quad +\vdt((t_{k+1})-,x)-\vdt(\tau_1,x),
\end{align*}
and so
\begin{equation*}
 \E{\int_{\R^d}  \Bigl(\vdt(\tau_2,x)-\vdt(\tau_1,x)\Bigr)(\beta_\delta\phi)(x)\, dx} 
 = \Term^{\text{CL}} + \Term^{\text{SDE}},
\end{equation*}
where 
\begin{align*}
 \Term^{\text{CL}} &:=  \E{\int_{\R^d}  \Bigl(\scl(\tau_2-t_l)u^l(x)-u^l(x)\Bigr)(\beta_\delta\phi)(x)\, dx} \\
 &\qquad+\sum_{n=k+1}^{l-1}\E{\int_{\R^d}  \Bigl(\scl(\Dt)u^n(x)-u^n(x)\Bigr)(\beta_\delta\phi)(x)\, dx} \\
&\qquad + \E{\int_{\R^d} \Bigl(\scl(\Dt)u^k-\scl(\tau_1-t_k)u^k\Bigr)(\beta_\delta\phi)(x)\, dx}, \\
 \Term^{\text{SDE}} &:= \sum_{n=k+1}^{l}\E{\int_{\R^d}  \Bigl(\int_{t_{n-1}}^{t_n}
\sigma(x,\udt(t,x))\,dB(t)\Bigr)(\beta_\delta\phi)(x)\, dx} \\
&=\E{\int_{\R^d} \Bigl(\int_{t_k}^{t_l}\sigma(x,\udt(t,x))\,dB(t)\Bigr)(\beta_\delta\phi)(x)\, dx}.
\end{align*}
Combining the above estimates 
yields \eqref{eq:weak-timevDt}.
 
\medskip
\textit{2. Strong estimate.} Let $d(x)= \udt(\tau_2)-\udt(\tau_1)$, 
$\beta(x) = \sgn{d(x)}$. By the triangle inequality,
\begin{align*}
& \E{\,\, \int_{\R^d} 
\abs{\udt(\tau_2,x)-\udt(\tau_1,x)}\phi(x)\, dx}
\\ & \quad
	\le \Bigl|E\Bigl[\,  
\int_{\R^d} \beta_\delta(x)d(x)
\phi(x)\, dx\Bigr]\Bigr|
	+ E\Bigl[\,\,  \int_{\R^d} \abs{\abs{d(x)}-\beta_\delta(x)d(x)}\phi(x)\, dx\Bigr] 
	\\ & \quad
	=: \Term_1 + \Term_2.
\end{align*}
By \eqref{eq:weak-time},
\begin{equation*}
  \Term_1 = \ocal \left(\delta^{\kappa-1}(l-k)\Dt + \sqrt{\tau_2-\tau_1}\right).
\end{equation*}
Consider $\Term_2$. Following, e.g.~\cite[Lemma 1]{Kruzkov:1970kx},
\begin{align*}
  \abs{\abs{d(x)}-\beta_\delta(x)d(x)}
  &\le \int_{\R^d} \abs{\abs{d(x)}-d(x)\sgn{d(y)}}V_{\delta}(x-y)\,dy \\
  &\le 2\int_{\R^d} \abs{d(x)-d(y)}V_{\delta}(x-y)\,dy. 
\end{align*}
Upon adding and subtracting identical terms 
and changing variables 
$2\tilde x = x+y$, $2z = x-y$, 
it follows (after relabeling $\tilde x$ by $x$)
\begin{align*}
  \int_{\R^d}\abs{\abs{d(x)}
	-\beta_\delta(x)d(x)}\phi(x)\,dx 
  &\le 2\iint_{\R^d \times \R^d}
  \abs{d(x+z)-d(x-z)}\\
  &\hphantom{le 2\iint_{\R^d \times \R^d}}\quad
  \times V_{\delta/2}(z)\abs{\phi(x+z)-\phi(x)}\,dzdx  \\
  & \quad + 2\iint_{\R^d \times \R^d}
  \abs{d(x+z)-d(x-z)}V_{\delta/2}(z)\phi(x)\,dzdx \\
  &=: \Term_2^1 + \Term_2^2.
\end{align*}
Consider $\Term_2^1$. By Lemma~\ref{lemma:PhiProp},
\begin{equation*}
 \abs{\phi(x+z)-\phi(x)} \le w_{1,\phi}(\abs{z})\phi(x).
\end{equation*}
Hence, by the symmetry of $V$ 
and the triangle inequality,
\begin{align*}
  \abs{\Term_2^1} 
  &\le 4\iint_{\R^d \times \R^d}
  \abs{d(x-z)}V_{\delta/2}(z)
	w_{1,\phi}(\abs{z})\phi(x)\,dzdx 
  \\ &\le 4w_{1,\phi}(\delta/2)\iint_{\R^d \times \R^d}
  \abs{d(y)}V_{\delta/2}(x-y)\phi(x)\,dydx 
  \\ &\le 2w_{1,\phi}(\delta)\norm{\udt(\tau_2)-\udt(\tau_1)}_{1,\phi \star V_{\delta/2}}.
\end{align*}
By Lemma~\ref{lemma:ContMollWeightedNorm} 
and Corollary~\ref{cor:LpLocTimeIntPol}, $\E{\abs{\Term_2^1}} = \ocal(\delta)$. 
By Proposition~\ref{prop:fractional-BV}, it follows in view 
of assumption \eqref{eq:AssumptionFracBVu0} 
that $\E{\abs{\Term_2^2}} = \ocal(\delta^\kappa)$. Consequently,
\begin{equation*}
 \Term_1+\Term_2= \ocal \left(\delta^{\kappa-1}(l-k)\Dt + \sqrt{\tau_2-\tau_1} +\delta^\kappa\right).
\end{equation*}
Choosing $\delta=((l-k)\Dt)$ concludes the proof of (i). 
The result (ii) follows analoguously due to \eqref{eq:weak-timevDt}.
\end{proof}

\section{Convergence}\label{sec:conv}
Equipped with $\Dt$-uniform a priori estimates, we are 
now prepared to study the limiting behavior 
of $\udt, \vdt$ as $\Dt\downarrow 0$. 
As discussed in the introduction, we will apply the 
framework of Young measures. We refer to the 
appendix (Section~\ref{seq:AppendixYoung}) for some 
background material on Young measures and weak compactness.

We start by establishing an approximate entropy 
inequality for the operator splitting solutions.
\begin{lemma}\label{lemma:EntOpSplit}
 Suppose $u^0 \in L^2(\Omega,\fcal_0,P;L^2(\R^d,\phi))$, $\phi \in \mathfrak{N}$. 
 Let $\udt$ and $\vdt$ be defined by \eqref{eq:udt} and \eqref{eq:vdt}, respectively. 
 For any $(S,Q) \in \mathscr{E}$, any $V \in \Sm$, and 
 any non-negative $\test \in C^\infty_c([0,T) \times \R^d)$,
 \begin{equation}\label{eq:EntOpSplit}
   \begin{split}
	&0\le \E{\int_{\R^d} S(u^0(x)-V)\test(0,x)\,dx} \\
	&+\E{\iint_{\Pi_T}S(\udt(t,x)-V)\partial_t \test(t,x) + Q(\vdt(t,x),V) \cdot \nabla \test(t,x)\,\,dxdt} \\
	&-\E{\iint_{\Pi_T} S''(\udt(t,x)-V)D_tV\sigma(x,\udt(t,x))\test(t,x)\,dxdt }\\
	&+\E{\frac{1}{2} \iint_{\Pi_T} S''(\udt(t,x)-V)\sigma^2(x,\udt(t,x))\test(t,x)\,dxdt}\\
	&+\sum_{n = 0}^{N-1}\E{\int_{t_n}^{t_{n+1}}\int_{\R^d} 
	\left(S(\vdt(t,x)-V)-S(\vdt((t_{n+1})-,x)-V\right)\partial_t\test(t,x)\,dxdt}.
   \end{split}
\end{equation}
\end{lemma}

\begin{proof}
Let us for the moment assume that $u^0 \in L^p(\Omega,\fcal_0,P;L^p(\R^d,\phi))$ 
for all $2 \leq p < \infty$. By definition, $\vdt$ satisfies
\begin{align*}
	&\int_{\R^d} S(u^n(x)-V)\test(t_n,x)\,dx- \int_{\R^d} S(\vdt((t_{n+1})-,x)-V)\test(t_{n+1},x)\,dx \\
	& \qquad + \int_{t_n}^{t_{n+1}} \int_{\R^d}  S(\vdt(t,x)-V)\partial_t\test(t,x)\,dx\,dt 
	\\ & \qquad 
	+ \int_{t_n}^{t_{n+1}} \int_{\R^d} Q(\vdt(t,x),V) \cdot \nabla \test(t,x)\, dx\,dt	
	\ge 0.
\end{align*}
For fixed $x \in \R^d$, apply Theorem~\ref{theorem:AntIto} with 
$F(\zeta,\lambda,t) = S(\zeta-\lambda)\varphi(t,x)$ and
\begin{displaymath}
 \underbrace{\udt(t,x)}_{X(t)} = \underbrace{\udt((t_n)+,x)}_{X_0} 
 + \int_{t_n}^t \underbrace{\sigma(x,\udt(s,x))}_{u(s)}\,dB(s).
\end{displaymath}
This yields, after integrating in space,
\begin{align*}
\int_{\R^d} S(u^{n+1}(x)-V)&\test(t_{n+1},x)\,dx = \int_{\R^d}S(\udt((t_n)+)-V)\test(t_n,x)\,dx \\
&+ \int_{\R^d}\int_{t_n}^{t_{n+1}}S(\udt(t,x)-V)\partial_t \test(t,x)\,dt\,dx \\
&+ \int_{\R^d}\int_{t_n}^{t_{n+1}}S'(\udt(t,x)-V)\sigma(x,\udt(t,x))\test(t,x)\,dB(t)\,dx \\
&- \int_{\R^d}\int_{t_n}^{t_{n+1}} S''(\udt(t,x)-V)D_tV\sigma(x,\udt(t,x))\test(t,x)\,dt\,dx \\
&+ \frac{1}{2} \int_{\R^d}\int_{t_n}^{t_{n+1}} S''(\udt(t,x)-V)\sigma^2(x,\udt(t,x))\test(t,x)\,dt\,dx,
\end{align*}
where the stochastic integral is a Skorohod integral. Note that 
\begin{multline*}
 \int_{\R^d}S(\udt((t_n)+,x)-V)\test(t_n,x)\,dx-\int_{\R^d} S(\vdt((t_{n+1})-,x)-V)\test(t_{n+1},x)\,dx \\
  = -\int_{\R^d}\int_{t_n}^{t_{n+1}}S(\vdt((t_{n+1})-,x)-V)\partial_t\test(t,x)\,dtdx.
\end{multline*}
 Adding the two equations and taking expectations we attain
 \begin{equation*}
  \begin{split}
	&\E{\int_{\R^d} S(u^n(x)-V)\test(t_n,x)\,dx}-\E{\int_{\R^d} S(u^{n+1}(x)-V)\test(t_{n+1},x)\,dx }\\
	&\quad +\E{\int_{t_n}^{t_{n+1}}\int_{\R^d} 
	\left(S(\vdt(t,x)-V)-S(\vdt((t_{n+1})-,x)-V)\right)\partial_t\test(t,x)\,dx\,dt} \\
	&\quad +\E{\int_{\R^d}\int_{t_n}^{t_{n+1}}S(\udt(t,x)-V)\partial_t \test(t,x)\,dt\,dx} \\
	&\quad +\E{\int_{t_n}^{t_{n+1}} \int_{\R^d} Q(\vdt(t,x),V) \cdot \nabla \test(t,x)\, dx\,dt} \\
	&\quad -\E{\int_{\R^d}\int_{t_n}^{t_{n+1}} S''(\udt(t,x)-V)D_tV\sigma(x,\udt(t,x))\test(t,x)\,dt\,dx }\\
	&\quad +\E{\frac{1}{2} \int_{\R^d}\int_{t_n}^{t_{n+1}} S''(\udt(t,x)-V)\sigma^2(x,\udt(t,x))\test(t,x)\,dt\,dx}
	\ge 0,
  \end{split}
 \end{equation*}
 where we applied the fact that the Skorohod integral has zero expectation. 
 Next we sum over $n = 0,1,\dots,N-1$. This yields \eqref{eq:EntOpSplit}. 
 The result follows for general $u^0 \in L^2(\Omega,\fcal_0,P;L^2(\R^d,\phi))$ 
 by approximation.
\end{proof}

\begin{theorem}\label{thm:entropy}
Suppose \eqref{eq:fluxLip-ass}, \eqref{eq:fluxDoubleDerBounded-ass}, 
\eqref{eq:sigma-ass}, and \eqref{eq:sigma-ass-xdep} hold. Let $\phi \in \mathfrak{N}$ and $2 \le p < \infty$. 
Suppose $u^0 \in L^p(\Omega,\fcal_0,P;L^p(\R^d,\phi))$ 
satisfies \eqref{eq:AssumptionFracBVu0}. 
Let $\udt$ and $\vdt$ be defined by \eqref{eq:udt} and \eqref{eq:vdt}, respectively. 
Then there exists a subsequence $\seq{\Dt_j}$ and a predictable 
$u \in L^p([0,T] \times \Omega;L^p(\R^d \times [0,1],\phi))$ such 
that both $\udtj \rightarrow u$ and $\vdtj \rightarrow u$ in the following sense:
For any Carath\'eodory function $\Psi:\R \times \Pi_T \times \Omega \rightarrow \R$ 
such that $\Psi(\udtj,\cdot) \rightharpoonup \overline{\Psi}$ (respectively 
$\Psi(\vdtj,\cdot) \rightharpoonup \overline{\Psi}$) in 
$L^1(\Pi_T \times \Omega,\phi\,dx \otimes dt \otimes dP)$,
\begin{equation}\label{eq:ReprOfNonLimit}
\overline{\Psi}(t,x,\omega) 
= \int_0^1 \Psi(u(t,x,\alpha,\omega),t,x,\omega)
\,d\alpha. 
\end{equation}
The process $\tilde{u} = \int_0^1 u \,d\alpha$ is an entropy solution in the sense of 
Definition \ref{def:stoch-es} with initial condition $u^0$.
\end{theorem}
\begin{proof}

\emph{1.~Existence of limits.} 
Let us investigate the limit behavior of $\udt$, noting 
that the same considerations apply to $\vdt$. 
We argue as in \cite[Theorem~4.1, Step~1]{KarlsenStorroesten2014} (see 
also \cite[\S~A.3.3]{Bauzet:2012kx}). We apply 
Theorem~\ref{theorem:YoungMeasureLimitOfComposedFunc} 
to $\seq{\udt}$ on the measure space
\begin{equation*}
(X,\mathscr{A},\mu) = (\Omega \times \Pi_T,\pcal 
\otimes \Borel{\R^d},dP \otimes dt \otimes \phi\,dx).
\end{equation*}
By Corollary~\ref{cor:LpLocTimeIntPol}, 
\begin{equation*}
\sup_{\Dt > 0}\seq{\E{\iint_{\Pi_T} \abs{\udt}^2\phi(x)\,dxdt}} < \infty.
\end{equation*}
Hence there exists a Young measure $\nu = \nu_{t,x,\omega}$ such that 
for any Carath\'eodory function $\Psi$ satisfying 
$\Psi(\udtj,\cdot) \rightharpoonup \overline{\Psi}$ in 
$L^1(\Pi_T \times \Omega,\phi\,dx \otimes dt \otimes dP)$, it follows that 
\begin{equation*}
\overline{\Psi}(t,x,\omega) = \int_\R 
	\Psi(\xi,t,x,\omega)\,d\nu_{t,x,\omega}(\xi). 
\end{equation*}
Define \cite{Eymard:2000fr,Panov:1996aa}
\begin{equation*}
u(t,x,\alpha,\omega) := 
	\inf \seq{\xi \in \R\,:\,\nu_{t,x,\omega}((-\infty,\xi]) > \alpha}.
\end{equation*}
The representation \eqref{eq:ReprOfNonLimit} follows from 
the fact that $\Lebesgue \circ u^{-1}(t,x,\cdot,\omega)
= \nu_{t,x,\omega}$, where $\Lebesgue$ denotes 
the Lebesgue measure on $[0,1]$. 
For predictability and the fact that 
$u \in L^p([0,T] \times \Omega;L^p(\R^d \times [0,1],\phi))$, see 
\cite[Theorem~4.1]{KarlsenStorroesten2014} 
and also \cite[\S~A.3.3]{Bauzet:2012kx}, \cite[\S~3]{Panov:1996aa}.

  \emph{2.~Independence of interpolation.} Denote by $v$ the limit 
of $\seq{\vdt}$, see Step 1. 
We want to show that $v = u$. 
By \cite[Lemma~6.3]{Pedregal1997}, this holds true if
\begin{equation}\label{eq:udtandvdtDist}
 \Term(\Dt):= \E{\iint_{\Pi_T}
	\abs{\udt(t,x)-\vdt(t,x)}\phi(x)\,dtdx} 
	\rightarrow 0 \mbox{ as $\Dt \downarrow 0$.}
\end{equation}
To see this, observe that 
\begin{align*}
   \Term &\le \E{\sum_{n = 0}^{N-1}
	\int_{t_n}^{t_{n+1}}	\int_{\R^d}
	\abs{\udt(t,x)-\udt((t_n)+,x)}
	\phi(x)\,dtdx} 
	\\ & \qquad +\E{\sum_{n = 0}^{N-1}
	\int_{t_n}^{t_{n+1}}\int_{\R^d}
	\abs{\vdt((t_{n+1})-,x)-\vdt(t,x)}\phi(x)\,dtdx} 
	\\ & =: \Term_1 + \Term_2.
\end{align*}
  By Proposition~\ref{prop:time-cont}~(i), 
 \begin{equation*}
 	\Term_1 \le C_{T,\phi}\sum_{n = 0}^{N-1}
	\int_{t_n}^{t_{n+1}}\sqrt{t-t_n}\,dt
	= \frac{2}{3}C_{T,\phi}T\sqrt{\Dt}.
  \end{equation*}
  By Proposition~\ref{prop:time-cont}~(ii),
  \begin{equation*}
  \Term_2 \le C_{T,\phi}\sum_{n = 0}^{N-1}
	\int_{t_n}^{t_{n+1}}(t_{n+1}-t)^\kappa\,dt 
	\le C_{T,\phi}T\Dt^\kappa,
  \end{equation*}
where $\kappa$ is defined in Proposition~\ref{prop:time-cont}. This proves \eqref{eq:udtandvdtDist}.

\emph{3. Entropy inequality.} We need to prove that $u$ is 
a Young measure-valued entropy solution in the sense of \cite[Definition~2.2]{KarlsenStorroesten2014}. 
The result then follows from \cite[Theorem~5.1]{KarlsenStorroesten2014}. Let $S,V,\test$ 
be as in Lemma~\ref{lemma:EntOpSplit} and define 
 \begin{equation*}
  \Term_\Dt := \sum_{n = 0}^{N-1}\E{\int_{t_n}^{t_{n+1}}\int_{\R^d} \big(S(\vdt(t,x)-V)-
S(\vdt((t_{n+1})-,x)-V)\big)\partial_t\test\,dxdt}.
\end{equation*}
We want to show that $\Term_\Dt \rightarrow 0$ as $\Dt \downarrow 0$. Recall the definition of 
the weighted $L^\infty$-norm \eqref{eq:DefWInfNorm}. By Proposition~\ref{prop:time-cont},
 \begin{align*}
  \abs{\Term_\Dt} 
    &\le \norm{S}_{\mathrm{Lip}}\sup_{t \in [0,T]}
	\left\{\norm{\partial_t\test}_{\infty,\phi^{-1}}
	\right\} 
	\\ & \qquad \quad \times
	\sum_{n = 0}^{N-1}\E{\int_{t_n}^{t_{n+1}}
	\int_{\R^d}\abs{\vdt(t,x)
	-\vdt((t_{n+1})-,x)}\phi(x)\,dxdt} \\
    &\le \norm{S}_{\mathrm{Lip}}
	\sup_{t \in [0,T]}
	\left\{\norm{\partial_t\test}_{\infty,\phi^{-1}}
	\right\} C_{T,\phi}T\Dt^\kappa,
 \end{align*}
 as in the proof of Step~2. Concerning the remaining 
 terms in Lemma~\ref{lemma:EntOpSplit}, the limit $\Dt \downarrow 0$ 
 is treated exactly as in \cite[Proof of Theorem~4.1, Step~2]{KarlsenStorroesten2014}. 
 It follows that $u$ is a Young measure-valued entropy solution.
\end{proof}

\section{Error estimate}\label{sec:error}
\FloatBarrier
We now restrict our attention to the case 
\begin{equation*}\label{eq:sigma-ass-xindep}
 \sigma(x,u) = \sigma(u), \qquad \sigma \in L^\infty.
 \tag{$\mathcal{A}_{\sigma,2}$}
\end{equation*}
As mentioned in the introduction, for homogeneous noise 
functions $\sigma = \sigma(u)$, whenever $\E{\norm{\nabla u_0}_{1,\phi}} < \infty$, the 
entropy solution $u$ 
to \eqref{eq:SBL} satisfies a spatial $BV$ estimate of the form
\begin{equation}\label{eq:BVMeanBound}
	\E{\int_{\R^d} \abs{\nabla u(t,x)} \phi(x)\,dx } 
	\leq C, \qquad (0 \leq t \leq T),
\end{equation}
for some finite constant $C$ (depending on $u_0,f,\phi,\sigma,T$). 
Here $\nabla u(t,\cdot)$ is a (locally finite) measure and $\phi \in \mathfrak{N}$. 
This can be seen as a consequence of the fractional space 
translation estimate \eqref{eq:EntSolSpatReg} and Remark~\ref{remark:spatToBV}. 
A direct verification of \eqref{eq:BVMeanBound} can 
also be found in \cite[Theorem~2.1]{Chen:2011fk} (when $\phi \equiv 1$). 
The same estimate is available for the 
operator splitting solution, cf.~Proposition~\ref{prop:fractional-BV}.

For the error estimate, we consider yet another time interpolation $\eta_\Dt$ of the 
operator splitting $\seq{u^n}_{n=0}^N$. Inspired by \cite{Langseth:1996jw}, let
\begin{equation}\label{eq:etaDt}
  \eta_\Dt(t) :=
  \underbrace{(\ssde(t,t_n)-\mathcal{I})
\scl(\Dt)u^n}_{\udt(t)-\scl(\Dt)u^n}
  + \underbrace{\scl(t-t_n)u^n}_{\vdt(t)}, 
\qquad t \in [t_n,t_{n+1}].   
\end{equation}
A graphical representation of the interpolation $\eta_\Dt$ is given in Figure~\ref{fig:etaDtRepr}.
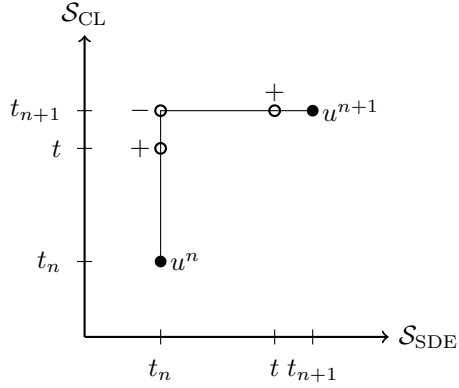
\begin{figure}[h]
\begin{tikzpicture}

\draw[thick, ->] (0,0) -- (4,0) node[right]{$\ssde$};
\draw[thick, ->] (0,0) -- (0,4) node[above]{$\scl$};

\filldraw (1,1) circle (2pt) node[right]{$u^{n}$};
\filldraw (3,3) circle (2pt) node[right]{$u^{n+1}$};

\draw (-0.1,1) -- (0.1,1) node[left=8pt]{$t_n$};
\draw (-0.1,2.5) -- (0.1,2.5) node[left=8pt]{$t$};
\draw (-0.1,3) -- (0.1,3) node[left=8pt]{$t_{n+1}$};

\draw (1,-0.1) -- (1,0.1) node[below=8pt]{$t_n$};
\draw (2.5,-0.1) -- (2.5,0.1) node[below=8pt]{$t$};
\draw (3,-0.1) -- (3,0.1) node[below=8pt]{$t_{n+1}$};

\draw (1,1) -- (1,3) -- (3,3);

\draw [thick] (1,2.5) circle (2pt) node[left]{$+$};
\draw [thick] (2.5,3) circle (2pt) node[above]{$+$};
\draw [thick](1,3) circle (2pt) node[left]{$-$};

\end{tikzpicture}
\caption{A graphical representation of $\eta_\Dt$. 
The value of $\eta_\Dt(t)$ corresponds to 
summing (with signs) the values taken at the unfilled dots.}
\label{fig:etaDtRepr}
\end{figure}

\begin{theorem}\label{theorem:ConRate}
Fix $\phi \in \mathfrak{N}$. 
Suppose \eqref{eq:fluxLip-ass}, \eqref{eq:fluxDoubleDerBounded-ass}, 
\eqref{eq:sigma-ass}, and \eqref{eq:sigma-ass-xindep} are satisfied. 
Suppose also that $u^0,u_0 \in L^2(\Omega,\fcal_0,P;L^2(\R^d,\phi))$ 
satisfy \eqref{eq:AssumptionFracBVu0} with $\kappao = 1$. 
Let $u$ be the entropy solution of \eqref{eq:SBL}--\eqref{eq:SBL-data} 
according to Definition~\ref{def:stoch-es} with initial condition $u_0$, and 
let $\eta_\Dt$ be defined by \eqref{eq:etaDt}. Then there exists a 
constant $C$, independent of $\Dt$ but dependent on 
$\sigma,f,T,\phi,u_0,u^0$, such that
$$
\E{\norm{u(t)-\eta_\Dt(t)}_{1,\phi}} 
\le e^{C_\phi\norm{f}_\mathrm{Lip}t}
\left(\E{\norm{u_0-u^0}_{1,\phi}} 
+ C\Dt^{\frac{1}{3}}\right),
\quad t \in [0,T].
$$
\end{theorem}

The proof is split into several parts, the results of which are gathered towards
the end of the section. To help motivate the upcoming technical 
arguments, let us outline a ``high-level" overview of the main idea, 
assuming that all relevant functions are 
smooth in $x$ and the spatial dimension is $d = 1$. 
 
The function $\eta_\Dt$ defined in \eqref{eq:etaDt} ought to satisfy an 
``approximate" entropy inequality. Formally, we have
\begin{equation}\label{eq:SmoothEquApprox}
	d\eta_\Dt + \partial_x f(\vdt) \, dt = \sigma(\udt)\, dB,
\end{equation}
indicating that the error terms can be 
expressed as perturbations of 
the coefficients $f,\sigma$. 
Let $u$ be a smooth (in $x$) solution of \eqref{eq:SBL}. 
By \eqref{eq:SmoothEquApprox},
\begin{displaymath}
 d(\eta_\Dt-u) = -\partial_x(f(\vdt)-f(u))\,dt + (\sigma(\udt)-\sigma(u))\,dB,
\end{displaymath}
and thus the It\^o formula gives
\begin{multline*}
 dS(\eta_\Dt-u) = -S'(\eta_\Dt-u)\partial_x(f(\vdt)-f(u))\,dt 
+S'(\eta_\Dt-u)(\sigma(\udt)-\sigma(u))\,dB \\
+\frac{1}{2}S''(\eta_\Dt-u)(\sigma(\udt)-\sigma(u))^2\,dt,
\end{multline*}
for any $S \in C^2(\R)$. Upon adding and 
subtracting identical terms and 
taking expectations, we arrive at
\begin{align*}
 \E{dS(\eta_\Dt-u)} =& -\E{S'(\eta_\Dt-u)\partial_x(f(\eta_\Dt)-f(u))\,dt} \\
&+ \frac{1}{2}\E{S''(\eta_\Dt-u)(\sigma(\eta_\Dt)-\sigma(u))^2\,dt} \\
&+ \E{S'(\eta_\Dt-u)
\partial_x(f(\eta_\Dt)-f(\vdt))\,dt}\\
&+ \E{S''(\eta_\Dt-u)\left(\int_{\eta_\Dt}^{\udt}(\sigma(z)-\sigma(u))
\sigma'(z)\,dz\right)\,dt}.
\end{align*}
The first two terms vanish as 
$S\rightarrow \abs{\cdot}$. 
Note that these 
terms also appear in the uniqueness 
argument, when two exact solutions are compared. 
Accordingly, they should not be thought 
of as error terms originating from the splitting 
procedure. The last two terms, however, are 
genuine error terms associated with the 
operator splitting and the interpolation $\eta_\Dt$. 
All of the above terms may be recognized 
in the forthcoming Lemma~\ref{lemma:DoublingArg}. 
The above simplified 
representation provides
intuition on how to estimate 
these error terms. This is in particular 
the case concerning the third 
term on the right-hand side. 
To this end, note that 
\begin{displaymath}
 \eta_\Dt - \vdt = \udt -\scl(\Dt)u^n 
= \int_{t_n}^t \sigma(\udt(s))\,dB(s),
\end{displaymath}
for $t_n \leq t < t_{n+1}$. Consequently,
\begin{equation}\label{eq:FluxDiffEx}
	\begin{split}
		\partial_x(f(\eta_\Dt)-f(\vdt)) 
		& = \left(f'(\eta_\Dt)-f'(\vdt)\right)\partial_x\vdt 
		\\ & \qquad\quad 
		+ f'(\eta_\Dt)\int_{t_n}^t \partial_x\sigma(\udt(s))\,dB(s).
	\end{split}
\end{equation}
Furthermore,
\begin{multline*}
	\E{\abs{\left(f'(\eta_\Dt)-f'(\vdt)\right)
	\partial_x\vdt}}
	\\ \leq \norm{f'}_\mathrm{Lip}
	\E{\E{\abs{\int_{t_n}^t \sigma(\udt(s))\,dB(s)}\,
	\Big|\,\fcal_{t_n}}\abs{\partial_x\vdt}},
\end{multline*}
which provides a way to estimate the term 
since $\vdt(t)\in BV$ 
and $\sigma\in L^\infty$.

Due to the lack of regularity we will work 
with an approximation of $\eta_\Dt$. 
Given $\seq{w^n = w^n(x)}_{n = 0}^{N-1}$, we set
\begin{equation}\label{eq:psiwDef}
\psi(t) := (\ssde(t,t_n)-\mathcal{I})w^n, 
\qquad t \in [t_n,t_{n+1}), 
\end{equation}
and $\etat := \psi + \vdt$. 
Note that $\edt = \psi + \vdt$ 
whenever $w^n = \scl(\Dt)u^n$ for $n = 0,\dots,N-1$. However, due to 
the lack of differentiability of $\scl(\Dt)u^n$, we 
will work with a sequence $\seq{w^n_k}_{k \geq 1}$ of smooth 
functions satisfying $w^n_k \rightarrow \scl(\Dt)u^n$ in 
$L^1(\Omega;L^1(\R^d,\phi))$ as $k \rightarrow \infty$. 
To simplify notation we suppress 
the dependence on $k$ and write $w^n = w^n_k$.

\begin{proposition}\label{proposition:SumIsEntropySol}
Suppose \eqref{eq:fluxLip-ass}, 
\eqref{eq:sigma-ass}, and \eqref{eq:sigma-ass-xindep} are 
satisfied. Let $\etat = \psi + \vdt$, where 
$\psi$ and $\vdt$ are defined in \eqref{eq:psiwDef} 
and \eqref{eq:vdt}, respectively. 
Then, for all nonnegative 
$\phi \in C^\infty_c([t_n,t_{n+1}]\times \R^d)$, any $V \in \D^{1,2}$, 
and all entropy/entropy-flux pairs $(S,Q) \in \mathscr{E}$,
\begin{align*}
    &\E{\int_{\R^d} S(\etat(t_n,x)-V)\phi(t_n,x)\,dx} \\
    &\qquad- \E{\int_{\R^d} 
	S(\etat((t_{n+1})-,x)-V)\phi(t_{n+1},x)\,dx} \\
    &\qquad+ \E{\iint_{\Pi_n} S(\etat-V)\partial_t\phi +
      Q(\etat,V)\cdot \nabla \phi\,dxdt} \\ 
    &\qquad + \E{\iint_{\Pi_n} \int_V^\etat S'(z-V)	
	\left(f'(z-\psi)-f'(z)\right)\,dz 
	\cdot \nabla\phi\,dxdt} \\
    &\qquad + \E{\iint_{\Pi_n} 
	\int_V^\etat S''(z-V)f'(z-\psi)\,dz 
	\cdot \nabla\psi \,\phi\,dxdt} \\
    &\qquad - \E{\iint_{\Pi_n}S''(\etat - V)D_tV 	
	\sigma(\psi + w^n)\phi\,dxdt} \\
    &\qquad + \frac{1}{2}\E{\iint_{\Pi_n}S''(\etat -V)	
	\sigma^2(\psi +w^n)\phi\,dxdt} \ge 0,
  \end{align*}
  where $\Pi_n = [t_n,t_{n+1}] \times \R$.
\end{proposition}
The proof of Proposition \ref{proposition:SumIsEntropySol} is deferred
to Section~\ref{sec:ProofofSumIsEntropy}. To ensure that the relevant 
quantities are Malliavin differentiable, we replace the entropy 
solution $u$ by the viscous approximation $\ue$, which solves
\begin{equation*}
  d\ue + \nabla \cdot f(\ue)dt = \sigma(x,\ue)dB(t) +
  \varepsilon\Delta\ue dt, \qquad \ue(0) = u_0, 
\end{equation*}
and then send $\varepsilon \downarrow 0$ at a later stage. Let us 
recall that $\seq{D_r\ue(t)}_{t > r}$ is a predictable weak solution to the linear problem
\begin{equation*}
 dw + \nabla \cdot (f'(\ue)w)dt = \sigma'(x,\ue)wdB(t) +
  \varepsilon\Delta w dt, \qquad w(r) = \sigma(\ue(r)),
\end{equation*}
for almost all $r \in [0,T]$, cf.~\cite[\S~3]{KarlsenStorroesten2014}. Furthermore,
\begin{displaymath}
 \esssup_{r \in [0,T]}\seq{\sup_{t \in [0,T]}\E{\norm{D_r\ue(t)}_{2,\phi}^2}} < \infty.
\end{displaymath}
As a consequence of \cite[Theorem~5.1]{KarlsenStorroesten2014} 
and \cite[Proposition~6.12]{Pedregal1997}, we have 
$\ue \rightarrow u$ in $L^1([0,T] \times \Omega;L^1(\R^d,\phi))$ 
as $\varepsilon \downarrow 0$. In fact, under the 
assumptions of Theorem~\ref{theorem:ConRate}, $\ue \rightarrow u$ with rate 
$1/2$ \cite[Theorem~5.2]{Chen:2011fk}.

We may now proceed with the 
doubling-of-the-variables argument.

\begin{lemma}\label{lemma:DoublingArg}
Fix $\phi \in \mathfrak{N}$. Let 
$\ue = \ue(s,y)$ be the viscous approximation of \eqref{eq:SBL}. 
Take $w(t,x) = w^n(x)$ for $t \in [t_n,t_{n+1})$, and let 
$\psi = \psi(t,x)$, $\vdt = \vdt(t,x)$, and $\etat = \etat(t,x)$ 
be defined in Proposition~\ref{proposition:SumIsEntropySol}. 
Let $t_0 \in [0,T)$, and pick $\gamma,r_0,r > 0$ such 
that $t_0 \leq T-2(\gamma + r_0)$. Define
 \begin{equation*}
   \xi_\gamma(t) = 1-\int_0^tJ_\gamma^+(s-t_0)\,ds.
 \end{equation*}
 Furthermore, let
 \begin{equation*}
 \test(t,x,s,y) = \frac{1}{2^d}\phi\left(\frac{x+y}{2}\right)
 J_r\left(\frac{x-y}{2}\right)
J_{r_0}^+(s-t)\xi_\gamma(t),
 \end{equation*}
and $S_\delta$ be defined in \eqref{eq:SdeltaDef}. Then
\begin{equation}\label{eq:DoublingIneqErrEst}
   L - R  + F + \Term_1 
+ \Term_2 + \Term_3 + \Term_4 
+ \Term_5 + \Term_6 \ge 0,
\end{equation}
 where
 \begin{align*}
   L &= \E{\iint_{\Pi_T}\int_{\R^d} 
	S_\delta(\etat(0,x)-\ue(s,y))
	\test(0,x,s,y)\,dxdsdy},\\
   R &= -\E{\iiiint_{\Pi_T^2} 
	S_\delta(\etat-\ue)(\partial_t
	+ \partial_s)\test\,dX},\\
   F &= \E{\iiiint_{\Pi_T^2} Q(\ue,\etat)
	\cdot \nabla_y\test +
     Q(\etat,\ue)\cdot \nabla_x\test\,dX}, \\ 
   \Term_1 &= \frac{1}{2}\E{\iiiint_{\Pi_T^2}
     S_\delta''(\ue-\etat)
	(\sigma(\ue)-\sigma(\etat))^2 \test \,dX},
   \\ 
   \Term_2 &= \E{\iiiint_{\Pi_T^2}
     S_\delta''(\ue-\etat)(\sigma(\ue)-D_t\ue)
	\sigma(\psi + w)\test
     \,dX}, \\ 
   \Term_3 &= \E{\iiiint_{\Pi_T^2}
     S_\delta''(\ue-\etat)\left(\int_\etat^{\psi+w}
	(\sigma(z)-\sigma(\ue))\sigma'(z)\,dz\right) \test
     \,dX}, \\ 
   \Term_4 &= \E{\iiiint_{\Pi_T^2} \int_\ue^\etat S_\delta'(z-\ue)
	\left(f'(z-\psi)-f'(z)\right)\,dz \cdot \nabla_x\test\,dX} \\ 
	   &\quad +\E{\iiiint_{\Pi_T^2}\int_\ue^\etat 
	S_\delta''(z-\ue)f'(z-\psi)\,dz \cdot \nabla_x\psi \test\,dX},\\ 
   \Term_5 &= \varepsilon\E{\iiiint_{\Pi_T^2}
	S_\delta(\ue-\etat)\Delta_y\test\,dX}, \\
   \Term_6 &= \sum_{n=0}^{N-1}E\Biggl[\iint_{\Pi_T}\int_{\R^d}
   \Big(S_\delta(\etat((t_{n+1}),x)-\ue(s,y)) \\ 
   &\hphantom{= -\sum_{n=0}^{N-1}E\Bigl[\iint_{\Pi_T}\int_{\R^d}\Big(}
   -S_\delta(\etat((t_{n+1})-,x)-\ue(s,y))\Big)
	\test(t_{n+1},x,s,y)\,dxdsdy\Biggr],
 \end{align*}
 where $dX = dxdtdsdy$.
\end{lemma}
\begin{proof}
Let us first assume $\phi \in C^\infty_c(\R^d)$, as the 
result for $\phi\in \mathfrak{N}$ then follows from an 
approximation argument. 
After a standard application of It\^o's formula to 
$\ue(s,y) \mapsto S_\delta(\ue(s,y)-\etat(t,x))
\test(s)$ for $s \geq t$, we arrive at 
 \begin{align*}
   &\E{\iiiint_{\Pi_T^2} 
	S_\delta(\ue-\etat)\partial_s\test +
     Q(\ue,\etat)\cdot \nabla_y\test \,dX} \\
   &\quad +\frac{1}{2}\E{\iiiint_{\Pi_T^2} 
	S_\delta''(\ue-\etat)\sigma^2(\ue) \test \,dX} 
   +\varepsilon\E{\iiiint_{\Pi_T^2}
	S_\delta(\ue-\etat)\Delta_y\test\,dX}\ge 0,
 \end{align*}
  cf.~\cite[Lemma~5.3]{KarlsenStorroesten2014}. 
	Take $V = \ue(s,y)$ in
  Proposition~\ref{proposition:SumIsEntropySol}, 
	integrate in $(s,y) \in \Pi_T$, and 
	sum over $n = 0,\dots,N-1$.
The outcome is 
 \begin{align*}
  &\E{\iint_{\Pi_T}\int_{\R^d} 
	S_\delta(\etat(0,x)-\ue(s,y))
	\test(0,x,s,y)\,dxdsdy} \\
  &\qquad+ \E{\iiiint_{\Pi_T^2} S_\delta(\etat-\ue)
	\partial_t\test + Q(\etat,\ue)
	\cdot \nabla_x\test\,dX} \\
  &\qquad+ \E{\iiiint_{\Pi_T^2} \int_\ue^\etat S_\delta'(z-\ue)
	\left(f'(z-\psi)-f'(z)\right)\,dz 
	\cdot \nabla_x\test\,dX} \\
  &\qquad+ \E{\iiiint_{\Pi_T^2}
	\int_\ue^\etat S_\delta''(z-\ue)f'(z-\psi)\,dz 
	\cdot \nabla_x\psi \test\,dX} \\
  &\qquad- \E{\iiiint_{\Pi_T^2}S_\delta''(\etat - \ue)
	D_t\ue \sigma(\psi + w)\test\,dX} \\
  &\qquad+ \frac{1}{2}
	\E{\iiiint_{\Pi_T^2}S_\delta''(\etat
    -\ue)\sigma^2(\psi + w)\test\,dX} \\ 
  &\qquad+ \sum_{n=0}^{N-1}
	E\Biggl[\iint_{\Pi_T}\int_{\R^d}
  \Big(S_\delta(\etat((t_{n+1}),x)-\ue(s,y)) \\ 
  &\qquad \hphantom{- \sum_{n=0}^{N-1}
	E\Big[\iint_{\Pi_T}\int_{\R^d}}
  -S_\delta(\etat((t_{n+1})-,x)
	-\ue(s,y))\Big)\test(t_{n+1},x,s,y)\,dxdsdy\Biggr] 
  \ge 0.
 \end{align*}
The lemma follows upon adding the two previous 
inequalities, noting that  
 \begin{align*}
   \frac{1}{2}\sigma^2(\ue)& - D_t\ue \sigma(\psi + w) 
	+ \frac{1}{2}\sigma^2(\psi + w) \\
   &= \frac{1}{2}(\sigma(\psi+w)-\sigma(\ue))^2 
	+ (\sigma(\ue)-D_t\ue)\sigma(\psi+w) \\
   &= \frac{1}{2}(\sigma(\etat)-\sigma(\ue))^2 
	+ \int_\etat^{\psi + w}(\sigma(z)-\sigma(\ue))	\sigma'(z)\,dz \\
   &\hphantom{XXXXXXXXXXXXXXXxx} 
	+ (\sigma(\ue)-D_t\ue)\sigma(\psi+w).
 \end{align*}
\end{proof}
In the following we estimate the terms 
appearing in Lemma~\ref{lemma:DoublingArg}. 
The underlying assumptions are the ones made in 
Theorem~\ref{theorem:ConRate}. We let $C$ denote a 
generic constant, meaning that it is 
independent of the ``small'' parameters 
$\Dt,r,r_0,\gamma,\varepsilon,\delta$. 
Furthermore, given a term $\Term$, we write $\Term = \ocal(g(\Dt, \dots,\delta))$ 
whenever $\abs{\Term} \leq Cg(\Dt, \dots,\delta)$ for 
some nonnegative function $g$. 

\begin{estimate}\label{est:L}
Let $L$ be defined in 
Lemma~\ref{lemma:DoublingArg}. Then
  \begin{equation*}
    \limsup_{r_0 \downarrow 0}L 
	\leq \E{\norm{u_0-u^0}_{1,\phi}} 
	+ \ocal\left(\delta + r\right). 
  \end{equation*}
\end{estimate}

\begin{proof}
  By \eqref{eq:Sdelta-est},
  \begin{equation*}
    \big|S_\delta(\etat(0,x)-\ue(s,y)) 
	- \abs{\etat(0,x)-\ue(s,y)}\big| \le \delta.
  \end{equation*}
  By the reverse triangle inequality
  \begin{align*}
    \big|\abs{\etat(0,x)-\ue(s,y)}
	-\abs{\etat(0,x)-u_0(y)}\big| & \le
    \abs{\ue(s,y)-u_0(y)}, \\ 
    \big|\abs{\etat(0,x)-u_0(y)}
	-\abs{\etat(0,x)-u_0(x)}\big| & \le
    \abs{u_0(y)-u_0(x)}.
  \end{align*}
  Hence, after adding and subtracting 
	identical terms, noting 
	that $\etat(0) = u^0$, it
  follows by the triangle inequality that
  \begin{multline*}
    \big|S_\delta(\etat(0,x)-\ue(s,y)) 
		- \abs{u^0(x)-u_0(x)}\big| \\
    \le \delta + \abs{\ue(s,y)-u_0(y)} +
    \abs{u_0(y)-u_0(x)}.
  \end{multline*}
  By \eqref{eq:symmetricMollExpr},
  \begin{align*}
    &\abs{L - \E{\norm{u^0-u_0}_{1,\phi \star J_r}}} \\
    & \qquad \le \delta \norm{\phi}_{L^1(\R^d)} +
    \underbrace{\int_0^T
	\E{\norm{\ue(s)-u_0}_{1,\phi \star
          J_r}}J_{r_0}^+(s)\,ds}_{\Zerm_1} \\ 
    & \qquad \qquad 
		+ \underbrace{\E{\frac{1}{2^d}\iint_{\R^d \times
          \R^d}\abs{u_0(y)-u_0(x)}
        \phi\left(\frac{x+y}{2}\right)
		J_r\left(\frac{x-y}{2}\right)\,dxdy}}_{\Zerm_2}.
  \end{align*}
  Thanks to \cite[Lemma~2.3]{KarlsenStorroesten2014}, 
	$\Zerm_1 \to 0$ as $r_0 \to 0$. 
	Regarding $\Zerm_2$ we apply \eqref{eq:ChVarBV}. 
	As $u_0$ satisfies \eqref{eq:AssumptionFracBVu0} 
	with $\kappao = 1$,
  \begin{equation*}
    \Zerm_2 = \E{\iint_{\R^d \times
        \R^d}\abs{u_0(x+z)-u_0(x-z)}
			\phi(x)J_r(z)\,dxdz} =
    \ocal(r).
  \end{equation*}
  Finally, we apply Lemma~\ref{lemma:ContMollWeightedNorm} to conclude
  that
  \begin{equation*}
    \abs{\E{\norm{u^0-u_0}_{1,\phi \star J_r}
	-\norm{u^0-u_0}_{1,\phi}}} = \ocal(r).
  \end{equation*}
\end{proof}

\begin{estimate}\label{est:R}
Let $R$ be defined in 
Lemma~\ref{lemma:DoublingArg}. Then
  \begin{equation*}
    \liminf_{\varepsilon,r_0 \downarrow 0}R 
		\geq \E{\int_0^T
      \norm{\etat(t)-u(t)}_{1,\phi}
		J_\gamma^+(t-t_0)\,dt} 
		+ \ocal\left(\delta + r\right).
  \end{equation*}
\end{estimate}

\begin{proof}
  It is easy to check that
  \begin{equation*}
   \begin{split}
    R &= E\bigg[\iiiint_{\Pi_T^2} 
		S_\delta(\etat(t,x)-\ue(s,y))
      \frac{1}{2^d}\phi\left(\frac{x+y}{2}\right) \\
      &\hphantom{XXXXXXXXXXXXXX}
      \times J_r\left(\frac{x-y}{2}\right) 
		J_{r_0}^+(s-t)
		J_\gamma^+(t-t_0)\,dX\bigg]. 
   \end{split}
  \end{equation*}
  Moreover, adding and subtracting 
identical terms, we obtain
  \begin{multline*}
    \abs{S_\delta(\etat(t,x)-\ue(s,y))
		- \abs{\etat(t,x)-\ue(t,x)}} \le  \\
    \delta + \abs{\ue(s,y)-\ue(t,y)} 
		+ \abs{\ue(t,y)-\ue(t,x)},
  \end{multline*}
  and so
  \begin{align*}
    &\Bigl|R - \E{\int_0^T 
	\norm{\etat(t)-\ue(t)}_{1,\phi \star
        J_r}J_\gamma^+(t-t_0)\,dt} \Bigr|
    \\ &\le \delta \norm{\phi}_{L^1(\R^d)} +
    \underbrace{\E{\iint_{[0,T]^2} 
		\norm{\ue(s)-\ue(t)}_{\phi \star J_r}
		J_{r_0}^+(s-t)
		J_\gamma^+(t-t_0)\,dsdt}}_{\Zerm_1} \\
    &+\underbrace{\E{\iint_{\Pi_T}
		\int_{\R^d}\abs{\ue(t,y)
          -\ue(t,x)}\frac{1}{2^d}\phi
			\left(\frac{x+y}{2}\right)
        J_r\left(\frac{x-y}{2}\right)
		J_\gamma^+(t-t_0)\,dxdydt}}_{\Zerm_2}.
  \end{align*}
Owing to Lemma~\ref{lemma:DoubLebCon}, 
$\lim_{r_0 \downarrow 0}
\Zerm_1 = 0$. Next, we utilize the 
strong convergence
$\ue \rightarrow u$ in $L^1([0,T] 
\times \Omega;L^1(\R^d,\phi))$ 
and \eqref{eq:ChVarBV} to conclude that
  \begin{equation*}
    \lim_{\varepsilon,r_0 \downarrow 0}\Zerm_2 =
    \int_0^T\E{\iint_{\R^d \times
        \R^d}\abs{u(t,x+z)-u(t,x-z)}
			\phi(x)J_r(z)\,dxdz}
			\,J_\gamma^+(t-t_0)dt. 
  \end{equation*}
It follows from \cite[Proposition~5.2]{KarlsenStorroesten2014} 
and the assumption \eqref{eq:AssumptionFracBVu0} with $\kappao = 1$ 
that $\abs{\lim_{\varepsilon,r_0 \downarrow 0}
\Zerm_2} = \ocal(r)$. 
The claim is now a consequence of 
Lemma~\ref{lemma:ContMollWeightedNorm}.
\end{proof}

\begin{estimate}\label{est:F}
Let $F$ be defined in 
Lemma~\ref{lemma:DoublingArg}. Then
 \begin{equation*}
  \limsup_{\varepsilon,r_0 \downarrow 0} F \le
    C_\phi\norm{f}_{\mathrm{Lip}}
		\E{\int_0^T\norm{u(t)
		-\etat(t)}_{1,\phi}\xi_\gamma(t)\,dt}
    + \ocal\left(\delta\left(1 
		+ \frac{1}{r}\right) + r\right).
 \end{equation*}
\end{estimate}

\begin{proof}
 Observe that 
 \begin{equation}\label{eq:Fsplit}
  F = F_1 + F_2 + F_3,
 \end{equation}
 where
\begin{align*}
  F_1 &:= \E{\iiiint_{\Pi_T^2} 
S_\delta'(\ue-\etat)(f(\ue)-f(\etat))
(\nabla_x + \nabla_y)\test \,dX}, \\
  F_2 &:= -\E{\iiiint_{\Pi_T^2}  
\int_\ue^\etat S_\delta''(z-\ue)
(f(z)-f(\ue))\,dz \cdot \nabla_x\test\,dX}, \\
  F_3 &:= -\E{\iiiint_{\Pi_T^2}  
\int_\etat^\ue 
S_\delta''(z-\etat)
(f(z)-f(\etat))\,dz \cdot \nabla_y\test\,dX}.
\end{align*} 

 The decomposition \eqref{eq:Fsplit} 
follows from the identities
\begin{align*}
  Q_\delta(\ue,\etat) &= 
S_\delta'(\ue-\etat)(f(\ue)-f(\etat)) 
- \int_\etat^\ue 
S_\delta''(z-\etat)(f(z)-f(\etat))\,dz, \\
  Q_\delta(\etat,\ue) &= 
S_\delta'(\etat-\ue)(f(\etat)-f(\ue)) 
- \int_\ue^\etat 
S_\delta''(z-\ue)(f(z)-f(\ue))\,dz,
\end{align*}
derived using integration by parts. 

 Next, we claim that 
 \begin{equation}\label{eq:F2F3bound}
  \abs{F_2} + \abs{F_3}
  = \ocal\left(\delta
\left(1 + \frac{1}{r}\right)\right).
 \end{equation}
 We consider $F_2$; the $F_3$ term 
is estimated likewise. Note that  
 \begin{equation*}
  \abs{\int_\ue^\etat 
S_\delta''(z-\ue)(f(z)-f(\ue))\,dz} \le 
\norm{f}_\mathrm{Lip}\delta.
 \end{equation*}
 Hence,
 \begin{equation*}
  \abs{F_2} \le \norm{f}_\mathrm{Lip}\delta 
  \E{\iiiint_{\Pi_T^2} \abs{\nabla_x \test}\,dX}.
 \end{equation*}
 By a straightforward computation,
 \begin{equation*}
  \iiiint_{\Pi_T^2} \abs{\nabla_x \test}\,dX 
\le \frac{1}{2}T\left(C_\phi + \norm{\nabla J}_{L^1(\R^d)}\frac{1}{r}\right)\norm{\phi}_{L^1(\R^d)}.
 \end{equation*}
This proves \eqref{eq:F2F3bound}. 

 Next, we claim that 
 \begin{equation}\label{eq:LimitF1Bound}
  \limsup_{\varepsilon,r_0 \downarrow 0}F_1 \le 
	C_\phi\norm{f}_{\mathrm{Lip}}
	\E{\int_0^T\norm{u(t)
	-\etat(t)}_{1,\phi \ast J_r}\xi_\gamma(t)\,dt} 
	+ \ocal\left(\delta + r\right).
 \end{equation}
 Set
 \begin{equation*}
  \mathcal{F}_\delta(b,a) = 
	S_\delta'(b-a)(f(b)-f(a)).
 \end{equation*}
 Then
 \begin{equation*}
  \begin{split}
  \abs{\mathcal{F}_\delta(b,a)-\mathcal{F}_\delta(c,a)} 
	&= \big|\int_c^b \partial_z\big(S_\delta'(z-a)(f(z)-f(a))\big)\,dz\big| \\
	&\le 2\norm{f}_\mathrm{Lip}\delta 
	+ \norm{f}_\mathrm{Lip}\abs{b-c};
  \end{split}
 \end{equation*}
 whence
 \begin{multline*}
  \abs{\mathcal{F}_\delta(\ue(s,y),\etat(t,x))
	-\mathcal{F}_\delta(\ue(t,x),\etat(t,x))} \\
    \le \norm{f}_\mathrm{Lip}(2\delta 
	+ \abs{\ue(s,y)-\ue(t,y)} 
	+ \abs{\ue(t,y)-\ue(t,x)}),
 \end{multline*}
 and so
 \begin{align*}
  &\abs{F_1-\E{\iint_{\Pi_T} 
\mathcal{F}_\delta(\ue(t,x),\etat(t,x)) 
\cdot (\nabla 
\phi \ast J_r)(x)\xi_\gamma(t)\,dxdt}} \\
  &\le C_\phi\norm{f}_\mathrm{Lip}\E{\iint_{[0,T]^2}
\norm{\ue(s)-\ue(t)}_{1,\phi 
\star J_r}J_{r_0}^+(s-t)\xi_\gamma(t)\,dsdt} \\
  & \quad + C_\phi\norm{f}_\mathrm{Lip}
  \E{\int_0^T\iint_{\R^d \times \R^d}
  \abs{\ue(t,x+z)-\ue(t,x-z)}J_r(z)\xi_\gamma(t)\phi(x)\,dxdzdt} \\
  &\quad + 2\delta
\norm{f}_\mathrm{Lip}T
\norm{\nabla \phi}_{L^1(\R^d)},
 \end{align*}
 where we have made a change of variables 
as in Estimate~\ref{est:R}. 
Following the same reasoning 
as in that estimate we arrive at
 \begin{equation*}
  \limsup_{\varepsilon,r_0 \downarrow 0}F_1
    \leq 
    \E{\iint_{\Pi_T}\mathcal{F}_\delta(u(t,x),\etat(t,x)) 
	\cdot (\nabla \phi \ast
    J_r)(x)\xi_{\gamma}(t)\,dxdt}
   + \ocal\left(\delta + r\right).
 \end{equation*}
Inequality \eqref{eq:LimitF1Bound} 
follows from $\mathcal{F}_\delta(a,b) 
\le \norm{f}_\mathrm{Lip}\abs{a-b}$ and 
$\abs{\nabla \phi} \le C_\phi\phi$. 
Combining the above estimates for $F_1,F_2,F_3$ 
concludes the proof of the claim.
\end{proof}

\begin{estimate}\label{est:T1}
Let $\Term_1$ be defined in 
Lemma~\ref{lemma:DoublingArg}. Then
\begin{equation*}
  \abs{\Term_1} \le C\delta.
 \end{equation*}
\end{estimate}

\begin{proof}
Since $S_\delta'' = 2J_\delta$,
\begin{equation*}
  S_\delta''(\ue-\etat)(\sigma(\ue)
-\sigma(\etat))^2 \le 
2\norm{\sigma}_{\mathrm{Lip}}^2
J_\delta(\ue-\etat)\abs{\ue-\etat}^2 
\le 2\norm{\sigma}_{\mathrm{Lip}}^2
\norm{J}_\infty\delta.
 \end{equation*}
Due to \eqref{eq:symmetricMollExpr} 
and Young's inequality for convolutions,
 \begin{equation*}
  \iiiint_{\Pi_T^2}\test \,dX = 
\left(\int_0^T\int_0^T J_{r_0}^+(s-t)
\xi_\gamma(t)\,dsdt\right)\left(\int_{\R^d}
\phi \star J_r(x)\,dx\right) \le T\norm{\phi}_{L^1(\R^d)}.
 \end{equation*}
 The result follows.
\end{proof}

\begin{estimate}\label{est:T2}
Let $\Term_2$ be defined in 
Lemma~\ref{lemma:DoublingArg}. Then
 \begin{equation*}
  \lim_{r_0 \downarrow 0} \Term_2 = 0.
 \end{equation*}
\end{estimate}

\begin{proof}
This follows exactly as in \cite[Limit~5]{KarlsenStorroesten2014}. However, the 
assumption $\sigma \in L^\infty$ 
simplifies the analysis and allows for 
$\phi \in \mathfrak{N}$ instead of $C^\infty_c(\R^d)$.
\end{proof}
\begin{estimate}\label{est:T3}
 Let $\Term_3$ be defined in Lemma~\ref{lemma:DoublingArg}. Then
 \begin{equation*}
  \abs{\Term_3} \le C\frac{1}{\delta}
\E{\sum_{n = 0}^{N-1}\int_{t_n}^{t_{n+1}}
\norm{w^n-\vdt(t)}_{\phi \star J_r} \,dt}.
 \end{equation*}
\end{estimate}

\begin{proof}
Now, as $\etat = \psi + \vdt$,
\begin{equation*}
\abs{\int_\etat^{\psi+w}(\sigma(z)
-\sigma(\ue))\sigma'(z)\,dz} \leq 
2\norm{\sigma}_\infty\norm{\sigma}_{\mathrm{Lip}}
\abs{w-\vdt}.
\end{equation*}
Keep in mind that $w(t) = w^n$ for 
$t \in [t_n,t_{n+1})$. 
The estimate then follows 
from \eqref{eq:Sdelta-est} and \eqref{eq:symmetricMollExpr}.
\end{proof}

\begin{estimate}\label{est:T4}
Let $\Term_4$ be defined in 
Lemma~\ref{lemma:DoublingArg}. Then 
 \begin{equation*}
  \abs{\Term_4} \le C\sqrt{\Dt}
	\left(1 + \E{\int_0^T 
	\norm{\nabla w(t)}_{1,\phi \star J_r}}\right).
 \end{equation*}
\end{estimate}
\begin{proof}
 The estimate is established under 
the assumption that $\vdt$ is smooth in $x$. 
The general result follows by an 
approximation argument. 
Integrating by parts and using 
the chain rule,
 \begin{align*}
    \Term_4 & = \E{\iiiint_{\Pi_T^2} 
		\int_\ue^\etat S_\delta'(z-\ue)
		\left(f'(z-\psi)-f'(z)\right)\,dz \cdot 
		\nabla_x\test\,dX}\\
	    & = -\E{\iiiint_{\Pi_T^2} 
		S_\delta'(\etat-\ue)
		\left(f'(\vdt)-f'(\etat)\right) 
		\cdot \nabla_x \etat \,\test\,dX} \\
	    & \qquad + \E{\iiiint_{\Pi_T^2}
		\int_\ue^\etat 
		S_\delta'(z-\ue)f''(z-\psi)\,dz 
		\cdot \nabla_x\psi\,\test\,dX}.
 \end{align*}
 Next, we observe that 
 \begin{displaymath}
  \int_\ue^\etat S_\delta'(z-\ue)
	f''(z-\psi)\,dz = -\int_\ue^\etat 
	S_\delta''(z-\ue)f'(z-\psi)\,dz 
	+ S_\delta'(\etat-\ue)f'(\vdt).
 \end{displaymath}
Therefore,
\begin{align*}
\Term_4 &= \underbrace{
\E{\iiiint_{\Pi_T^2} 
S_\delta'(\etat-\ue)f'(\etat)\cdot 
\nabla_x \psi\,\test\,dX}}_{\Zerm_1} \\
&\qquad + \underbrace{\E{\iiiint_{\Pi_T^2} S_\delta'(\etat-\ue)
\left(f'(\etat)-f'(\vdt)\right) 
\cdot \nabla_x 
\vdt\, \test\,dX}}_{\Zerm_2},
 \end{align*}
cf.~\eqref{eq:FluxDiffEx}. 
Consider $\Zerm_2$. Since $\vdt(t)$ 
is $\fcal_{t_n}$-measurable for 
all $t \in [t_{n},t_{n+1})$,
 \begin{align*}
  \abs{\Zerm_2} &\leq \E{\iiiint_{\Pi_T^2} 	\abs{f'(\etat)-f'(\vdt)}\abs{\nabla_x \vdt} 
	\test\,dX} \\ &
	\leq \norm{f'}_{\mathrm{Lip}}
	\sum_{n = 0}^{N-1}
	\iiiint_{\Pi_T \times \Pi_n} 
	\E{\E{\abs{\psi}\Big |\fcal_{t_n}}
	\abs{\nabla_x \vdt}} \test\,dX.
\end{align*}
By definition,
 \begin{equation}\label{eq:psiEquationSecEst}
  \psi(t,x) = \int_{t_n}^t 
	\sigma(\psi(r,x) + w^n(x))\,dB(r), 
	\qquad t_n \le t < t_{n+1}.
 \end{equation}
Set
\begin{displaymath}
 \tilde{\psi}(t,\lambda) = \int_{t_n}^t \sigma(\tilde{\psi}(r,\lambda) + \lambda)\,dB(r), 
\end{displaymath}
so that $\psi(t,x) = \tilde{\psi}(t,w^n(x))$. Consequently,
\begin{displaymath}
 \E{\abs{\psi(t,x)}\Big| \fcal_{t_n}}(\omega) = \E{|\tilde{\psi}(t,\lambda)|}_{\lambda = w^n(x;\omega)}.
\end{displaymath}
By the Burkholder-Davies-Gundy inequality, there is a constant 
$c_1 > 0$ such that 
 \begin{displaymath}
  \E{|\tilde{\psi}(t,\lambda)|} 
	\leq c_1\E{\left(\int_{t_n}^t \sigma^2(\tilde{\psi}(r,\lambda) 
	+ \lambda)\,dr\right)^{1/2}} 
	\leq c_1\norm{\sigma}_\infty \sqrt{t-t_n},
 \end{displaymath}
independent of $\lambda \in \R$. It follows from 
Proposition~\ref{prop:fractional-BV} that 
 \begin{displaymath}
  \abs{\Zerm_2} \leq c_1\norm{\sigma}_\infty\norm{f'}_{\mathrm{Lip}}\sqrt{\Dt}
  \E{\int_0^T \norm{\nabla_x \vdt(t)}_{1,\phi \star J_r}\,dt} \leq C\sqrt{\Dt}.
 \end{displaymath}
 Consider $\Zerm_1$. In view of \eqref{eq:symmetricMollExpr},
 \begin{equation*}
    \abs{\Zerm_1} \le \norm{f}_{\mathrm{Lip}}
    \E{\iiiint_{\Pi_T^2}\abs{\nabla_x\psi}\test\,dX} \le \norm{f}_{\mathrm{Lip}}
\E{\iint_{\Pi_T}
\abs{\nabla_x\psi}(\phi \star J_r)\,dxdt}.
 \end{equation*}
 Differentiating \eqref{eq:psiEquationSecEst} yields, 
for $t_n \le t < t_{n+1}$,
\begin{equation*}
  \nabla_x \psi(t,x) = \int_{t_n}^t 
\sigma'(\psi(r,x) + w^n(x))(\nabla_x\psi(r,x) 
+ \nabla_x w^n(x))\,dB(r).
\end{equation*}
By Lemma~\ref{lemma:SmallTimeStepItoEst} below 
there is a constant $C > 0$, depending 
only on $\sigma$, such that 
\begin{equation*}
  \E{\abs{\nabla_x\psi(t,x)}} \le C \sqrt{t-t_n}	\E{\abs{\nabla w^n(x)}}, 
	\qquad t_n \le t < t_{n+1}.
\end{equation*}
 We conclude that 
 \begin{equation*}
  \abs{\Zerm_1} \le C\left(\E{\int_0^T
  \norm{\nabla w(t)}_{1,\phi \ast J_r}}\,dt\right)\sqrt{\Dt}.
 \end{equation*}
\end{proof}

\begin{lemma}\label{lemma:SmallTimeStepItoEst} 
Suppose $h:[t_n,t_{n+1}] 
\times \Omega \rightarrow \R^d$ 
is predictable and 
\begin{equation*}
  P\Bigl[\int_{t_n}^t 
\abs{h(s)}^2\,ds < \infty \Bigr] = 1.
 \end{equation*}
Suppose $X(t_n) \in L^p(\Omega,\fcal_{t_n},P;\R^d)$, 
$1 \le p < \infty$, 
and let $X:[t_n,t_{n+1}] \times \Omega \rightarrow \R^d$ satisfy
\begin{equation*}
  X(t) = X(t_n) 
	+ \int_{t_n}^t h(s)\,dB(s), 
	\qquad t \in [t_n,t_{n+1}].
 \end{equation*}
Suppose there exist a constant $K$ and $Y \in L^p(\Omega,\fcal_{t_n},P)$ such that
\begin{equation}\label{eq:hLinGrowthAndLip}
  \abs{h(t;\omega)} \le Y(\omega) + K\abs{X(t)}, \qquad t \in [t_n,t_{n+1}].
\end{equation}
Then, for all $t \in [t_n,t_{n+1}]$ and $\beta > p(c_p^{1/p}K)^2/2$,
\begin{equation*}
  \sup_{t_n \le s \le t}\E{\abs{X(s)}^p}^{1/p} \le C(\beta)
  e^{\beta(t-t_n)}\left(\E{\abs{X(t_n)}^p}^{1/p} + c_p^{1/p}\sqrt{t-t_n}\E{\abs{Y}^p}^{1/p}\right),
\end{equation*}
where $C(\beta) = \left(1-c_p^{1/p}K\sqrt{p/2\beta}\right)^{-1}$ 
and $c_p$ is the constant from the Burkholder-Davies-Gundy inequality.
\end{lemma}

\begin{proof}
Set
\begin{equation*}
 \norm{X}_{\beta,p,\tau} := \left(\sup_{t_n \le t 
 \le \tau}e^{-\beta(t-t_n)}\E{\abs{X(t)}^p}\right)^{1/p}.
\end{equation*}
The triangle inequality yields
\begin{equation*}
 \E{\abs{X(t)}^p}^{1/p} 
 \le \E{\abs{\int_{t_n}^t h(s)\,dB(s)}^p}^{1/p} 
+ \E{\abs{X(t_n)}^p}^{1/p}.
\end{equation*}
By the Burkholder-Davies-Gundy inequality,
\begin{equation*}
 \E{\abs{\int_{t_n}^t h(s)\,dB(s)}^p}^{1/p} \le c_p^{1/p}
\E{\left(\int_{t_n}^t h^2(s)\,ds\right)^{p/2}}^{1/p}.
\end{equation*}
Due to \eqref{eq:hLinGrowthAndLip} and the 
triangle inequality on $L^p(\Omega;L^2([t_n,t]))$,
\begin{displaymath}
 \E{\left(\int_{t_n}^t \abs{h(s)}^2\,ds\right)^{p/2}}^{1/p} 
 \leq \sqrt{t-t_n}\E{\abs{Y}^p}^{1/p} + K\E{\left(\int_{t_n}^t\abs{X(s)}^2\,ds\right)^{p/2}}^{1/p}.
\end{displaymath}
By Minkowski's integral inequality,
\begin{displaymath}
 \E{\left(\int_{t_n}^t\abs{X(s)}^2\,ds\right)^{p/2}}^{2/p} \leq \int_{t_n}^t\E{\abs{X(s)}^p}^{2/p}\,ds.
\end{displaymath}
Furthermore, 
\begin{align*}
 \int_{t_n}^t\E{\abs{X(s)}^p}^{2/p}\,ds 
     &= e^{2\beta(t-t_n)/p}\int_{t_n}^t\left(e^{-\beta(t-s)}e^{-\beta(s-t_n)}\E{\abs{X(s)}^p}\right)^{2/p}\,ds \\
     &\leq e^{2\beta(t-t_n)/p}\norm{X}_{\beta,p,t}^2\int_{t_n}^te^{-2\beta(t-s)/p}\,ds \\
     &=\frac{p}{2\beta}\left(e^{2\beta(t-t_n)/p}-1\right)\norm{X}_{\beta,p,t}^2.
\end{align*}
Summarizing, we arrive at 
\begin{align*}
 \E{\abs{X(t)}^p}^{1/p} &\leq \E{\abs{X(t_n)}^p}^{1/p} + c_p^{1/p}\sqrt{t-t_n}\E{\abs{Y}^p}^{1/p} \\
&\quad +c_p^{1/p}K\sqrt{\frac{p}{2\beta}}
\left(e^{2\beta(t-t_n)/p}-1\right)^{1/2}\norm{X}_{\beta,p,t}.
\end{align*}
Multiplying by $e^{-\beta(t-t_n)/p}$ and taking the supremum 
over $t_n \leq t \leq \tau$, we obtain
\begin{equation*}
 \norm{X}_{\beta,p,\tau} \le \E{\abs{X(t_n)}^p}^{1/p} 
+ c_p^{1/p}\sqrt{\tau-t_n}\E{\abs{Y}^p}^{1/p} 
+ c_p^{1/p}K\sqrt{\frac{p}{2\beta}}\norm{X}_{\beta,p,\tau}.
\end{equation*}
Choosing $\beta$ sufficiently large, i.e.
$c_p^{1/p}K\sqrt{p/2\beta} < 1$, we secure the bound
\begin{equation*}
\norm{X}_{\beta,p,\tau} \le 
\frac{1}{1-c_p^{1/p}K\sqrt{p/2\beta}}
\left(c_p^{1/p}\sqrt{\tau-t_n}\E{\abs{Y}^p}^{1/p} 
+ \E{\abs{X(t_n)}^p}^{1/p}\right).
\end{equation*}
The result follows upon multiplication by $e^{\beta(\tau-t_n)/p}$, since
\begin{equation*}
 e^{\beta(\tau-t_n)/p}\norm{X}_{\beta,p,\tau} 
= \left(\sup_{t_n \le t \le \tau}
e^{\beta(\tau-t)}\E{\abs{X(t)}^p}\right)^{1/p} 
\ge \sup_{t_n \le t \le \tau}\E{\abs{X(t)}^p}^{1/p}.
\end{equation*}
\end{proof}
\begin{estimate}\label{est:T5}
 Let $\Term_5$ be defined in 
Lemma~\ref{lemma:DoublingArg}. Then
 \begin{equation*}
  \Term_5 = \ocal(\varepsilon).
 \end{equation*}
\end{estimate}
\begin{proof}
 This follows as in \cite[Limit~6]{KarlsenStorroesten2014}.
\end{proof}
\begin{estimate}\label{est:T6}
 Let $\Term_6$ be defined in Lemma~\ref{lemma:DoublingArg}. Then
 \begin{equation*}
  \abs{\Term_6} \le 2\sum_{n=0}^{N-1}
\E{\norm{\scl(\Dt)u^n-w^n}_{1,\phi \star J_r}}.
 \end{equation*}
\end{estimate}

\begin{proof}
 First, we note that $\abs{S_\delta(b)-S_\delta(a)} 
\le \abs{b-a}$. This and \eqref{eq:symmetricMollExpr} yields  
 \begin{equation*}
  \abs{\Term_6} \le \sum_{n=0}^{N-1}
\E{\norm{\etat(t_{n+1})-\etat((t_{n+1})-)}_{1,\phi \star J_r}}.
 \end{equation*}
 Since
 \begin{equation*}
  \etat(t_{n+1})-\etat((t_{n+1})-) = 
	\ssde(t_{n+1},t_n)(\scl(\Dt)u^n-w^n) + \scl(\Dt)u^n-w^n,
 \end{equation*}
 the result follows from \eqref{eq:SSDE1Contraction}.
\end{proof}

\begin{proof}[Proof of Theorem~\ref{theorem:ConRate}]
Consider Lemma~\ref{lemma:DoublingArg}, and 
take the upper limits in \eqref{eq:DoublingIneqErrEst} as 
$r_0 \downarrow 0, \varepsilon \downarrow 0$, 
and $\gamma \downarrow 0$ (in that order). 
Next we recall that $w^n = w^n_k$. Letting 
$k \rightarrow \infty$, $w^n_k \rightarrow \scl(\Dt)u^n$ 
in $L^1(\Omega,L^1(\R^d,\phi))$. Due to the $L^1$-Lipschitz 
continuity of $\scl$ (cf. Proposition~\ref{prop:time-cont}) 
and the uniform $BV$-bound on the splitting approximation, it 
follows from Estimates \ref{est:L}--\ref{est:T6} that 
\begin{multline*}
  \E{\norm{u_0-u^0}_{1,\phi}} + 
	C_\phi\norm{f}_\mathrm{Lip} \int_0^{t_0} 
	\E{\norm{\eta_\Dt(t)-u(t)}_{1,\phi}}\,dt \\
  + \ocal\left(\delta + r + \sqrt{\Dt} +\frac{\delta}{r} 
	+ \frac{\Dt}{\delta}\right) \ge \E{\norm{\eta_\Dt(t_0)-u(t_0)}_{1,\phi}}.
\end{multline*}
Finally, we apply Gr\"onwall's inequality, and then 
choose $\delta = \Dt^{2/3}$ and $r = \Dt^{1/3}$.
\end{proof}

\section{Appendix}\label{sec:appendix}

\subsection{Proof of Proposition 
\ref{proposition:SumIsEntropySol}}\label{sec:ProofofSumIsEntropy}
The proof of Proposition \ref{proposition:SumIsEntropySol} 
is based on the following result.
\begin{lemma}\label{lemma:SumIsEntropySol}
Suppose $u,w \in L^2(\Omega,P,\fcal_{t_n};L^2(\R^d))$ 
and $w$ is smooth. Set 
 \begin{equation*}
  \psi(t) = (\ssde(t,t_n)-\mathcal{I})w,
	\quad v(t) = \scl(t-t_n)u, 
	\qquad
	t \in [t_n,t_{n+1}].
\end{equation*}
Then for all 
$(S,Q) \in \mathscr{E}$, all nonnegative 
$\test \in C^\infty_c(\Pi_n^2)$, and 
all $V \in \mathcal{S}$,
 \begin{equation*}
  R - L + \Term_1 + \Term_2 - \Term_3 + \Term_4 \ge 0,
 \end{equation*}
 where
 \begin{align*}
  \begin{split}
   L &= \E{\iint_{\Pi_n}\int_{\R^d} S(v(t_{n+1},x)
	+\psi(s,y)-V)\test(t_{n+1},x,s,y)\,dxdyds} \\
     & \quad + \E{\iint_{\Pi_n}\int_{\R^d} S(v(t,x) 
	+ \psi(t_{n+1},y)-V)\test(t,x,t_{n+1},y)\,dydxdt},
  \end{split} \\
  \begin{split}
   R &= \E{\iint_{\Pi_n}\int_{\R^d} S(v(t_n,x)
	+\psi(s,y)-V)\test(t_n,x,s,y)\,dxdyds} \\
     & \quad + \E{\iint_{\Pi_n}\int_{\R^d} S(v(t,x)
	+ \psi(t_n,y)-V)\test(t,x,t_n,y) \,dydxdt},
  \end{split}\\
  \Term_1 &= \E{\iiiint_{\Pi_n^2} S(v(t,x)+\psi(s,y)-V)(\partial_t 
	+ \partial_s)\test \dX}, \\
  \Term_2 &= \E{\iiiint_{\Pi_n^2} Q(v(t,x),V
	-\psi(s,y))\cdot \nabla_x\test \dX}, \\
  \Term_3 &= \E{\iiiint_{\Pi_n^2}S''(v(t,x) 
	+ \psi(s,y)-V)D_sV \sigma(\psi(s,y) + w(y))\test\,dX}, \\
  \Term_4 &= \frac{1}{2}\E{\iiiint_{\Pi_n^2}S''(v(t,x) 
	+ \psi(s,y)-V)\sigma^2(\psi(s,y) + w(y))\test\,dX},
 \end{align*}
 and $\Pi_n = [t_n,t_{n+1}] \times \R^d$.
\end{lemma}

\begin{proof}[Proof of Lemma~\ref{lemma:SumIsEntropySol}]
 The entropy inequality reads
 \begin{multline}\label{eq:EntIneqw}
    \int_{\R^d} S(v(t_n,x)-c)\test(t_n,x,s,y)-S(v(t_{n+1},x)-c)
	\test(t_{n+1},x,s,y)\,dx \\
    + \iint_{\Pi_n} S(v-c)\partial_t\test + Q(v,c) 
	\cdot \nabla_x\test\,dtdx \ge 0,
 \end{multline}
for all $c \in \R$ and all $s,y \in \Pi_n$. 
Specify $c = V-\psi(s,y)$ in 
\eqref{eq:EntIneqw}, integrate in $(s,y)$, and take 
expectations, to obtain
 \begin{equation}\label{eq:EntIneqwSubsInt}
  \begin{split}
  &\E{\iint_{\Pi_n}\int_{\R^d} S(v(t_n,x)
	+\psi(s,y)-V)\test(t_n,x,s,y) \,dxdsdy} \\
  &\quad -\E{\iint_{\Pi_n}\int_{\R^d} 
	S(v(t_{n+1},x)+\psi(s,y)-V)\test(t_{n+1},x,s,y)\,dxdsdy} \\
  & \quad + \E{\iiiint_{\Pi_n^2} S(v+\psi-V)\partial_t\test 
	+ Q(v,V-\psi) \cdot \nabla_x\test\,dX} \ge 0.
  \end{split}
 \end{equation}
Note that $v(t)$ is $\fcal_{t_n}$-adapted 
for all $t \in [t_n,t_{n+1}]$. To reveal the 
equation satisfied by $\psi$, let 
$\zeta(t) = \ssde(t,t_n)w$. By definition,
 \begin{equation*}
  \zeta(t,x) = w(x) + \int_{t_n}^t \sigma(\zeta(r,x))\,dB(r).
 \end{equation*}
 Since $\psi(t) = \zeta(t) - w$, 
 \begin{equation}\label{eq:psiEquation}
  \psi(t,x) = \int_{t_n}^t \sigma(\psi(r,x) + w(x))\,dB(r), 
	\qquad t \in [t_n,t_{n+1}].
 \end{equation}
 Fix $t,x \in \Pi_n, y \in \R^d$ and set
 \begin{equation*}
  X(s) := v(t,x) + \psi(s,y), \quad 
	F(X(s),V,s) := S(X(s)-V)\test(t,x,s,y), \qquad s \in [t_n,t_{n+1}].
 \end{equation*}
 By \eqref{eq:psiEquation},
 \begin{equation*}
  X(s) = v(t,x) + \int_{t_n}^s 
	\sigma(\psi(r,y) + w(y))\,dB(r).
 \end{equation*}
 By Theorem~\ref{theorem:AntIto},
 \begin{align*}
  S(X(t_{n+1})-V)&\test(t,x,t_{n+1},y) = S(X(t_n)-V)\test(t,x,t_n,y) \\ 
	     & + \int_{t_n}^{t_{n+1}}S(X(s)-V)
			\partial_s \test \,ds \\
	     & + \int_{t_n}^{t_{n+1}} 
			S'(X(s)-V)\sigma(\psi(s) + w)\test\,dB(s) \\
	     & - \int_{t_n}^{t_{n+1}}S''(X(s)-V)
			D_sV \sigma(\psi(s) + w)\test\,ds \\
	     & + \frac{1}{2}
			\int_{t_n}^{t_{n+1}}S''(X(s)-V)
			\sigma^2(\psi(s) + w)\test\,ds,
 \end{align*}
where the stochastic integral is interpreted as a Skorohod integral. 
Upon integrating in $t,x,y$ and taking expectations, 
 \begin{equation}\label{eq:EntExprSDE}
  \begin{split}
    & \E{\iint_{\Pi_n} \int_{\R^d} S(v(t,x) 
		+ \psi(t_n,y)-V)\test(t,x,t_n,y) \,dydtdx} \\ 
    & \quad-\E{\iint_{\Pi_n} \int_{\R^d} S(v(t,x) 
		+ \psi(t_{n+1},y)-V)\test(t,x,t_{n+1},y)\,dydtdx} \\
    & \quad+ \E{\iiiint_{\Pi_n^2}S(v(t,x) 
		+ \psi(s,y)-V)\partial_s \test \,dX} \\
    & \quad+ \frac{1}{2}\E{\iiiint_{\Pi_n^2}S''(v(t,x) 
		+ \psi(s,y)-V)(\sigma(\psi(s,y) 
		+ w(y)))^2\test\,dX} \\
    & \quad- \E{\iiiint_{\Pi_n^2}S''(v(t,x) 
		+ \psi(s,y)-V)D_sV \sigma(\psi(s,y) + w(y))
		\test\,dX} = 0.
  \end{split}
 \end{equation}
Adding \eqref{eq:EntIneqwSubsInt} and 
\eqref{eq:EntExprSDE} concludes the proof.
\end{proof}

\begin{proof}[Proof of Proposition~\ref{proposition:SumIsEntropySol}]
We use
\begin{equation}\label{eq:DefOfTestFunc}
  \test(t,x,s,y) = \frac{1}{2^d}\phi
	\left(\frac{t+s}{2},\frac{x+y}{2}\right)
	J_r\left(\frac{x-y}{2}\right)J_{r_0}(t-s)
\end{equation}
in Lemma~\ref{lemma:SumIsEntropySol} and 
then send $r_0, r$ to zero (in that order). 
The sought result for $V \in \mathcal{S}$ is a consequence of 
Limits \ref{limit:LR}--\ref{limit:I4} below. The extension 
to $V \in \D^{1,2}$ follows by an approximation 
argument as in \cite[Lemma~2.2]{KarlsenStorroesten2014}.  
\end{proof}

\begin{limit}\label{limit:LR}
Let $L, R$ be defined in 
Lemma~\ref{lemma:SumIsEntropySol} 
and $\test$ in \eqref{eq:DefOfTestFunc}. Then
 \begin{equation*}
  \begin{split}
  \lim_{r,r_0 \downarrow 0} L(r,r_0) &= 
	\E{\int_{\R^d} 
	S(v(t_{n+1},x)+\psi(t_{n+1},x))
	\phi(t_{n+1},x)\,dx}, \\
  \lim_{r,r_0 \downarrow 0} R(r,r_0) &
	= \E{\int_{\R^d} S(v(t_n,x)+\psi(t_n,x)-V)
	\phi(t_n,x)\,dx}.
\end{split}
\end{equation*}
\end{limit}

\begin{proof}
 Let us only consider the term
 \begin{equation*}
  \E{\iint_{\Pi_n}\int_{\R^d} 
	S(v(t_{n+1},x)+\psi(s,y)-V)
	\test(t_{n+1},x,s,y)\,dxdyds} =: \Zerm.
 \end{equation*}
The remaining terms can be treated 
in the same way. As a consequence 
of the dominated convergence theorem and 
Lemma~\ref{lemma:DoubLebCon}, 
 \begin{align*}
  &\lim_{r_0 \downarrow 0}\Zerm = 
	\frac{1}{2}E\Bigg[\int_{\R^d}
	\int_{\R^d} S(v(t_{n+1},x)+\psi(t_{n+1},y)-V) \\
  & \hphantom{XXXXXXXXXXXXXX}\times 
	\frac{1}{2^d}\phi\left(t_{n+1},\frac{x+y}{2}\right)
	J_r\left(\frac{x-y}{2}\right)\,dxdy\Bigg].
 \end{align*}
Moreover,
\begin{equation*}
  \lim_{r,r_0 \downarrow 0} \Zerm = 
	\frac{1}{2}\E{\int_{\R^d} 
	S(v(t_{n+1},x)+\psi(t_{n+1},x)-V)\phi(t_{n+1},x)}.
\end{equation*}
\end{proof}

\begin{limit}\label{limit:I1}
Let $\Term_1$ be defined in 
Lemma~\ref{lemma:SumIsEntropySol} 
and $\test$ in \eqref{eq:DefOfTestFunc}. Then
\begin{equation*}
\lim_{r,r_0 \downarrow 0} 
\Term_1 = \E{\iint_{\Pi_n} 
S(u(t,x)-V)\partial_t\phi(t,x) \,dxdt}.
\end{equation*}
\end{limit}

\begin{proof}
Observe that 
 \begin{equation*}
  (\partial_t + \partial_s)
	\test(t,x,s,y) = \frac{1}{2^d}
	\partial_1\phi	\left(\frac{t+s}{2},\frac{x+y}{2}\right)
	J_r\left(\frac{x-y}{2}\right)J_{r_0}(t-s).
 \end{equation*}
The result follows by the 
dominated convergence theorem and 
Lemma \ref{lemma:DoubLebCon}, consult the 
proof of Limit \ref{limit:LR}.
\end{proof}

\begin{limit}\label{limit:I2}
Let $\Term_2$ be defined in 
Lemma \ref{lemma:SumIsEntropySol} 
and $\test$ in \eqref{eq:DefOfTestFunc}. Then
 \begin{align*}
  \lim_{r,r_0 \downarrow 0} \Term_2 &
	= \E{\iint_{\Pi_n} Q(v+\psi,V)\cdot 
	\nabla\phi\,dxdt} \\
	&\quad +\E{\iint_{\Pi_n}
	\left(\int_V^{v + \psi}S'(z-V)
	\left(f'(z-\psi)-f'(z)\right)\,dz\right)\cdot 
	\nabla\phi\,dxdt} \\
	& \quad +\E{\iint_{\Pi_n}
	\left(\int_V^{v+\psi}
	S''(z-V)f'(z-\psi)\,dz\right)
	\cdot \nabla\psi\,\phi\,dxdt}.
 \end{align*}
\end{limit}

\begin{proof}
First observe that  
\begin{displaymath}
  (\nabla_x +\nabla_y)\test(t,x,s,y) 
	= \frac{1}{2^d}\partial_2\phi
	\left(\frac{t+s}{2},\frac{x+y}{2}\right)
	J_r\left(\frac{x-y}{2}\right)J_{r_0}(t-s).
\end{displaymath}
Integration by parts results in
\begin{align*}
   \Term_2&=E\Bigg[\iiiint_{\Pi_n^2} 
	Q(v(t,x),V-\psi(s,y)) \\
	& \hphantom{XXXXXXXX} \cdot 
	\frac{1}{2^d}\partial_2
	\phi\left(\frac{t+s}{2},\frac{x+y}{2}\right)
	J_r\left(\frac{x-y}{2}\right)
	J_{r_0}(t-s)\,dxdtdyds\Bigg] \\
	&\quad \qquad +\E{\iiiint_{\Pi_n^2} \nabla_y \cdot 	
	Q(v(t,x),V-\psi(s,y))\test(t,x,s,y)\,dxdtdyds} 
	\\ &=: \Term_2^1 + \Term_2^2.
\end{align*}
It is straightforward to show that
\begin{equation*}
 \begin{split}
  \lim_{r,r_0 \downarrow 0} 
	\Term_2^1 &= \E{\iint_{\Pi_n} 
	Q(v(t,x),V-\psi(t,x))\cdot 
	\nabla\phi(t,x)\,dxdt}.
\end{split}
\end{equation*}
Finally, we apply the identity
 \begin{equation*}
  Q(v,V-\psi) = Q(v+\psi,V) + 
	\int_V^{v + \psi}S'(z-V)\left(f'(z-\psi)-f'(z)\right)\,dz.
 \end{equation*}
 Consider $\Term_2^2$. By the chain rule, 
 \begin{multline*}
  \Term_2^2 = -\E{\iiiint_{\Pi_n^2} 
	\partial_2Q(v(t,x),V-\psi(s,y))\cdot \nabla_y\psi(s,y)
	\,\test(t,x,s,y)\,dxdtdyds}. 
 \end{multline*}
 Sending $r_0, r$ to zero, we arrive at
 \begin{equation*}
  \lim_{r,r_0 \downarrow 0} 
	\Term_2^2 = -\E{\iint_{\Pi_n} 
	\partial_2Q(v(t,x),V-\psi(t,x))\cdot 
	\nabla_x\psi(t,x)\,\phi(t,x)\,dxdt}.
 \end{equation*}
 Finally, note that
 \begin{align*}
  \partial_2Q(v,V-\psi) &= -\int_{V-\psi}^v
	S''(z-V+\psi)f'(z)\,dz \\
	&= -\int_V^{v+\psi}S''(z-V)f'(z-\psi)\,dz.
 \end{align*}
 This concludes the proof.
\end{proof}

\begin{limit}\label{limit:I3}
 Let $\Term_3$ be defined in 
Lemma \ref{lemma:SumIsEntropySol} 
and $\test$ in \eqref{eq:DefOfTestFunc}. Then
 \begin{equation*}
  \lim_{r,r_0 \downarrow 0} \Term_3 = 
	\E{\iint_{\Pi_n}S''(v(t,x) 
	+ \psi(t,x)-V)D_tV \sigma(\psi(t,x) +w(x))
	\phi(t,x)\,dxdt}.
 \end{equation*}
\end{limit}
\begin{proof}
The proof is a straightforward application of 
the dominated convergence 
theorem and Lemma \ref{lemma:DoubLebCon}.
\end{proof}
\begin{limit}\label{limit:I4}
 Let $\Term_4$ be defined in 
Lemma~\ref{lemma:SumIsEntropySol} and 
$\test$ by \eqref{eq:DefOfTestFunc}. Then
\begin{equation*}
 \lim_{r,r_0 \downarrow 0} 
\Term_4 = \frac{1}{2}\E{\iint_{\Pi_n}S''(v(t,x) 
+ \psi(t,x)-V)\sigma^2(\psi(t,x) +w(x))\phi(t,x)\,dxdt}.
 \end{equation*}
\end{limit}
\begin{proof}
 This term may be treated similarly as $\Term_3$.
\end{proof}

\subsection{Weighted $L^p$ spaces}\label{sec:WeightedLp}
In the next two lemmas we 
collect a few elementary properties 
of (weight) functions in $\mathfrak{N}$. 
For proofs, see \cite{KarlsenStorroesten2014}.
\begin{lemma}\label{lemma:PhiProp}
Suppose $\phi \in \mathfrak{N}$ and $0 < p < \infty$. 
Then, for $x,z\in \R^d$,
$$
\abs{\phi^{1/p}(x+z)-\phi^{1/p}(x)} 
\leq w_{p,\phi}(\abs{z})\phi^{1/p}(x), 
$$
where 
$$
w_{p,\phi}(r) = \frac{C_\phi}{p}r\left(1 + \frac{C_\phi}{p}r e^{C_\phi r/p}\right),
$$
which is defined for all $r \geq 0$. As a consequence it follows that if 
$\phi(x_0) = 0$ for some $x_0 \in \R^d$, then $\phi \equiv 0$ (and by 
definition $\phi \notin \mathfrak{N}$).
\end{lemma}
\begin{lemma}\label{lemma:ContMollWeightedNorm}
Fix $\phi \in \mathfrak{N}$, and let $w_{p,\phi}$ be defined in 
Lemma \ref{lemma:PhiProp}. Let $J$ be a mollifier as 
defined in Section~\ref{seq:prelim} and 
take $\phi_\delta = \phi \star J_\delta$ for $\delta > 0$. Then
\begin{itemize}
	\item[(i)] $\phi_\delta \in \mathfrak{N}$ with $C_{\phi_\delta} = C_\phi$.
	\item[(ii)] For any $u \in L^p(\R^d,\phi)$,
	$$
	\abs{\norm{u}_{p,\phi}^p-\norm{u}_{p,\phi_\delta}^p} 
	\leq w_{1,\phi}(\delta)\min\seq{\norm{u}_{p,\phi}^p,\norm{u}_{p,\phi_\delta}^p}.
	$$
	\item[(iii)] 
	\begin{displaymath}
	 \abs{\Delta \phi_\delta(x)} \leq \frac{1}{\delta}
	 C_\phi\norm{\nabla J}_{L^1(\R^d)}(1 + w_{1,\phi}(\delta))^2\phi_\delta(x).
	\end{displaymath}
 \end{itemize}
\end{lemma}

\subsection{A ``doubling of variables'' tool}
The following result follows along the lines 
of \cite[Lemma~2.7.2]{Serre1999}. 
See also \cite[\S~6]{KarlsenStorroesten2014}.

\begin{lemma}\label{lemma:DoubLebCon}
Suppose $u,v \in L^1_{\mathrm{loc}}(\R^d)$ and $F$ is 
Lipschitz on $\R^2$. Fix $\psi \in C_c(\R^d)$ and set
\begin{align*}
	\mathcal{T}_r &:= \int_{\R^d}\int_{\R^d} F(u(x),v(y))
	\frac{1}{2^d}\psi\left(\frac{x+y}{2}\right)
	J_r\left(\frac{x-y}{2}\right)\,dydx \\
	& \qquad\qquad \qquad 
	-\int_{\R^d} F(u(x),v(x))\psi(x)\,dx,
\end{align*}
where $J_r$ is defined in \eqref{eq:MollifierDef}. 
Then $\mathcal{T}_r \rightarrow 0$ as $r \downarrow 0$.

Similarly, let $G:[0,T] \times \R \rightarrow \R$ 
be measurable in the first variable and 
Lipschitz continuous in the second 
variable. With $w \in L^1([0,T])$, set 
$$
\mathcal{T}_{r_0}(s) = \int_0^T 
\abs{G(s,w(t))-G(s,w(s))}J_{r_0}(t-s)\,dt.
$$
Then $\mathcal{T}_{r_0}(s) \rightarrow 0$ for 
a.e.~$s$ as $r_0 \downarrow 0$. 

The above results do not rely on the 
the symmetry of $J$.
\end{lemma}

\subsection{A version of It\^{o}'s formula} 
Here we recall the particular anticipating It\^{o} formula applied in 
the proof of Lemma~\ref{lemma:EntOpSplit} and Lemma~\ref{lemma:SumIsEntropySol}. 
The proof of this follows \cite[Theorem~3.2.2]{Nualart2006} closely. However, due to the 
particular assumptions, certain points simplifies. 
See \cite[Theorem~6.7]{KarlsenStorroesten2014} for an outline of a proof.
\begin{theorem}\label{theorem:AntIto}
Let $X$ be a continuous process of the form
$$
X(t) = X_0 + \int_0^t u(s)\,dB(s) 
+ \int_0^t v(s)\,ds,
$$
where $u:[0,T] \times \Omega \rightarrow \R$ and 
$v:[0,T] \times \Omega \rightarrow \R$ are predictable processes, satisfying 
\begin{equation*}
	\E{\bigg(\int_0^T u^2(s,z)\,ds\bigg)^2} < \infty, 
	\qquad \E{\int_0^T v^2(s)\,ds} < \infty,
\end{equation*}
and $X_0 \in L^2(\Omega,\fcal_0,P)$. 
Let $F:\R^2 \times [0,T] \rightarrow \R$ be twice continuously 
differentiable. Suppose there exists a constant $C > 0$ such that for all 
$(\zeta,\lambda,t) \in \R^2 \times [0,T]$,
\begin{align*}
	& \abs{F(\zeta,\lambda,t)},\abs{\partial_3F(\zeta,\lambda,t)} 
	\leq C(1 + \abs{\zeta} + \abs{\lambda}), \\
	& \abs{\partial_1F(\zeta,\lambda,t)}, 
	\abs{\partial_{1,2}^2F(\zeta,\lambda,t)}, 
	\abs{\partial_1^2F(\zeta,\lambda,t)} \leq C.
\end{align*}
Let $V \in \Sm$. Then $s \mapsto \partial_1F(X(s),V,s)u(s)$ is 
Skorohod integrable, and
\begin{align*}
	F(X(t),V,t) &= F(X_0,V,0) \\
	&\qquad +\int_0^t\partial_3F(X(s),V,s)\,ds \\
	&\qquad +\int_0^t\partial_1F(X(s),V,s)u(s,z)\,dB(s) \\	
	&\qquad +\int_0^t\partial_1F(X(s),V,s)v(s)\,ds \\
	&\qquad +\int_0^t\partial_{1,2}^2F(X(s),V,s)D_sVu(s)\,ds \\
	&\qquad +\frac{1}{2}\int_0^t \partial_1^2F(X(s),V,s)u^2(s)\,ds,
	\quad \text{$dP$-almost surely}.
\end{align*}
\end{theorem}

\subsection{Young measures}\label{seq:AppendixYoung}
The purpose of this section is to provide a 
reference for some results concerning 
Young measures and their use in representation 
formulas for weak limits. For a more general introduction, see for instance \cite{Florescu2012, Malek1996, Valadier1995}.

Let $(X,\mathscr{A},\mu)$ be a $\sigma$-finite 
measure space and $\mathscr{P}(\R)$ the 
set of probability measures on $\R$. 
In this paper, $X$ is typically $\Pi_T \times \Omega$. 
A \emph{Young measure} from $X$ into $\R$ 
is a function $\nu:X \rightarrow \mathscr{P}(\R)$ 
such that $x \mapsto \nu_x(B)$ is $\mathscr{A}$-measurable 
for every Borel measurable set 
$B \subset \R$. We denote by 
$\Young{X,\mathscr{A},\mu;\R}$, 
or $\Young{X;\R}$ if the measure 
space is understood, the set of all 
Young measures from $X$ into $\R$. 
The following theorem is proved 
in \cite[Theorem~6.2]{Pedregal1997} in the case 
that $X \subset \R^n$ and $\mu$ 
is the Lebesgue measure:
 
\begin{theorem}\label{theorem:YoungMeasureLimitOfComposedFunc}
Fix a $\sigma$-finite measure space 
$(X,\mathscr{A},\mu)$. 
Let $\zeta:[0,\infty) \rightarrow [0,\infty]$ be a 
continuous, non decreasing function 
satisfying $\lim_{\xi \rightarrow \infty}\zeta(\xi) 
= \infty$ and $\seq{u^n}_{n \ge 1}$ a sequence of 
measurable functions such that 
\begin{equation*}
  \sup_{n} \int_X \zeta(\abs{u^n})d\mu(x)  
	< \infty.
\end{equation*}
Then there exist a subsequence 
$\seq{u^{n_j}}_{j \ge 1}$ 
and $\nu \in \Young{X,\mathscr{A},\mu;\R}$ such 
that for any Carath\'eodory function 
$\psi:\R \times X \rightarrow \R$ with 
$\psi(u^{n_j}(\cdot),\cdot) 
\rightharpoonup \overline{\psi}$ 
in $L^1(X)$, we have 
\begin{equation*}
  \overline{\psi}(x) = 
	\int_{\R} \psi(\xi,x)\,d\nu_x(\xi).
 \end{equation*}
\end{theorem}

The proof is based on the embedding of $\Young{X;\R}$ 
into $L^\infty_{w*}(X,\Rad{\R})$. 
Here $\Rad{\R}$ denotes 
the space of Radon measures on $\R$. 
The crucial observation is that $(L^1(X,C_0(\R)))^*$ 
is isometrically isomorphic to 
$L^\infty_{w*}(X,\Rad{\R})$ also 
in the case that $(X,\mathscr{A},\mu)$ is an 
abstract $\sigma$-finite measure space. It is 
relatively straightforward to go through the proof 
and extend it to the more general case 
\cite[Theorem~2.11]{Malek1996}. The result then follows as an application of Alaoglu's theorem combined with the Eberlein-\v{S}mulian theorem. 
Note, however, due to our use of weighted $L^p$ 
spaces, it suffices with the 
version for finite measure spaces.

\subsection{Weak compactness in $L^1$.}
To apply Theorem \ref{theorem:YoungMeasureLimitOfComposedFunc} 
it is necessary to know if $\seq{\psi(\cdot,u^n(\cdot))}_{n \ge 1}$ 
has a subsequence converging weakly in $L^1(X)$. 
The key result is the well-known Dunford-Pettis Theorem. 

\begin{definition}\label{def:equiintegrability}
 Let $\mathcal{K} \subset L^1(X,\mathscr{A},\mu)$.
 \begin{itemize}
  \item[(i)] $\mathcal{K}$ is 
	\emph{uniformly integrable} if 
	for any $\varepsilon > 0$ there 
  exists $c_0(\varepsilon)$ such that 
  \begin{equation*}
   \sup_{f \in \mathcal{K}} \int_{\abs{f} \ge c} 
	\abs{f} \,d\mu \le \varepsilon \; 
	\mbox{ whenever $c \ge c_0(\varepsilon)$.}
  \end{equation*}
  \item[(ii)] $\mathcal{K}$ has 
	\emph{uniform tail} if for any 
	$\varepsilon > 0$ there 
  exists $E \in \mathscr{A}$ 
	with $\mu(E) < \infty$ such that 
  \begin{equation*}
   \sup_{f \in \mathcal{K}}\int_{X \setminus E}
 		\abs{f} \,d\mu \le \varepsilon.
  \end{equation*}
 \end{itemize}
 If $\mathcal{K}$ satisfies both (i) 
and (ii) it is said to be \emph{equiintegrable}.
\end{definition}

\begin{remark}
 Note that (ii) is void when $\mu$ is finite.
\end{remark}

\begin{theorem}[Dunford-Pettis]\label{theorem:DunfordPettis}
 Let $(X,\mathscr{A},\mu)$ be a 
$\sigma$-finite measure space. 
A subset $\mathcal{K}$ of $L^1(X)$ 
is relatively weakly sequentially compact 
if and only if it is equiintegrable.
\end{theorem}
 
There are a couple of well known reformulations 
of uniform integrability.
\begin{lemma}\label{lemma:UniformIntCriteria}
 Suppose $\mathcal{K} \subset L^1(X)$ is bounded. 
Then $\mathcal{K}$ is uniformly integrable if and only if:
 \begin{itemize}
  \item[(i)] For any $\varepsilon > 0$ there 
	exists $\delta(\varepsilon) > 0$ such that
 \begin{equation*}
  \sup_{f \in \mathcal{K}}\int_E \abs{f} \,d\mu 
	\le \varepsilon \; \mbox{ whenever 
	$\mu(E) \le \delta(\varepsilon)$.}
 \end{equation*}
  \item[(ii)] There is an increasing 
	function $\Psi:[0,\infty) \rightarrow [0,\infty)$ 
  such that $\Psi(\zeta)/\zeta \rightarrow \infty$ 
	as $\zeta \rightarrow \infty$ and
  \begin{equation*}
   \sup_{f \in \mathcal{K}}\int_X 
	\Psi(\abs{f(x)}) \,d\mu(x) < \infty.
  \end{equation*}
 \end{itemize}
\end{lemma}

\def\ocirc#1{\ifmmode\setbox0=\hbox{$#1$}\dimen0=\ht0 \advance\dimen0
  by1pt\rlap{\hbox to\wd0{\hss\raise\dimen0
  \hbox{\hskip.2em$\scriptscriptstyle\circ$}\hss}}#1\else {\accent"17 #1}\fi}

\end{document}